\documentclass[11 pt,leqno]{amsart}

\usepackage[all]{xy}
\usepackage {amsfonts}
\usepackage{amsthm}
\usepackage{amssymb,amsxtra}
\usepackage{mathabx}
\usepackage{latexsym}
\usepackage{amsmath}
\usepackage{mathrsfs}
\usepackage{dsfont}
\usepackage{lmodern}
\usepackage{nicefrac,mathtools}
\usepackage{calc}

\usepackage{enumitem} 
\setlist[enumerate,1]{label=\textup{(\arabic*)}}

\theoremstyle{plain}
\newtheorem{thm}{Theorem}[section]

\newtheorem{lem}[thm]{Lemma}
\newtheorem{prop}[thm] {Proposition}

\newtheorem{cor}[thm]{Corollary}

\theoremstyle{definition}
\newtheorem{defn}[thm]{Definition}
\newtheorem{rem}[thm]{Remark}
\newtheorem{ex}[thm]{Example}

\newcommand{\eqn}{\begin{equation}}
\newcommand{\eqne}{\end{equation}}

\newcommand*{\UMult}{\mathcal UM}
\newcommand*{\Mult}{\mathcal M}

\newcommand{\donto}{\stackrel{\mathrm{dense}}{\onto}} 
\newcommand{\anti}{\stackrel{\mathrm{anti}}{\cong}} 

\mathchardef\mhyphen="2D 

\DeclareMathOperator{\Aut}{Aut}

\DeclareMathOperator{\PBij}{PBij}
\DeclareMathOperator{\PHomeo}{PHomeo}

\DeclareMathOperator{\PAut}{PAut}
\DeclareMathOperator{\iso}{iso}

\DeclareMathOperator{\op}{op}
\DeclareMathOperator{\red}{r}

\DeclareMathOperator{\alg}{alg}
\DeclareMathOperator{\tight}{tight}
\DeclareMathOperator{\spa}{spa}
\DeclareMathOperator{\reg}{reg}

\DeclareMathOperator{\SPIso}{SPIso}
\DeclareMathOperator{\PIso}{PIso}

\DeclareMathOperator{\clsp}{\overline{span}}
\DeclareMathOperator{\spane}{span}

\DeclareMathOperator{\supp}{supp}

\DeclareMathOperator{\Bis}{Bis}

\newcommand{\id}{\mathrm{id}} 
\newcommand*{\Ad}{\textup{Ad}}

\newcommand{\B}{\mathcal B}
\newcommand{\LL}{\mathcal L}

\newcommand{\G}{\mathcal G} 

\newcommand{\EE}{\mathcal E}
\newcommand{\RR}{\mathcal R}

\newcommand{\F}{\mathbb F}

\newcommand{\p}{\varphi}

\renewcommand{\L}{\mathcal L}

\newcommand{\FF}{\mathcal F}

\newcommand{\C}{\mathbb C}
\newcommand{\R}{\mathbb R}

\newcommand{\N}{\mathbb N}
\newcommand{\T}{\mathbb T}

\newcommand{\cst}{\ifmmode\mathrm{C}^*\else{$\mathrm{C}^*$}\fi}

\newcommand*{\onto}{\twoheadrightarrow}


\begin{document}
\author{Krzysztof Bardadyn}

\author{Bartosz Kwa\'sniewski}

\author{Andrew McKee}

\date{\today}

\title{\small Banach algebras associated to twisted \'etale groupoids:\\
inverse semigroup disintegration
 \\
and representations on $L^p$-spaces}

\begin{abstract} 
We introduce Banach algebras associated to  twisted \'etale  groupoids $(\mathcal{G},\mathcal{L})$ and to twisted inverse semigroup actions. 
This provides a unifying framework for numerous recent papers on $L^p$-operator algebras and the theory of groupoid $C^*$-algebras. 
We prove  disintegrations theorems that allow  to study Banach algebras associated to $(\mathcal{G},\mathcal{L})$ as universal Banach algebras generated by $C_0(X)$ and a twisted inverse semigroup $S$ of partial isometries subject to some relations. 
They work best when  the target of a representation is a dual Banach algebra. For representations on dual Banach spaces, they 
 allow to extend representations to twisted Borel convolution algebras, which is crucial when the groupoid is non-Hausdorff.

We establish fundamental norm estimates and hierarchy for full and reduced $L^p$-operator algebras for $(\mathcal{G},\mathcal{L})$ and  $p \in [1,\infty]$, whose special cases have been studied recently by Gardella--Lupini, Choi--Gardella--Thiel and Hetland--Ortega.  
We show that in the constructions of $L^p$-analogues of Cuntz or graph algebras, by Phillips and Corti\~{n}as--Rodr\'{\i}guez, the use of spatial partial isometries is not an assumption, in fact it is forced by the relations. 
We also introduce tight inverse semigroup Banach algebras that cover ample groupoid Banach algebras, and discuss Banach algebras associated to  directed graphs. 

Our results cover non-Hausdorff \'{e}tale groupoids and both real and complex algebras. 
Some of the results are new  already for complex $C^*$-algebras.
\end{abstract}

\maketitle

%
%
\section{Introduction}
\label{sec:Introduction}

The theory of Banach algebras generated by groups is a classical part of  harmonic analysis. 
In the realm of $C^*$-algebras it was extended with huge success to transformation groups, groupoids and even more general structures and actions.  
This inspired a number of attempts to generalize this theory to Banach algebras or at least to operator algebras acting on $L^p$-spaces.  
A series of initial preprints about Banach algebra crossed products and $L^p$-operator algebras  
appeared in the 2010s, see \cite{DDW}, \cite{PhLp1}, \cite{PhLp2a}, \cite{Phillips}.
Phillips's program to create an analogue of the $C^*$-theory for $L^p$-operator algebras has noticeably sparked in recent years, see the  survey paper \cite{Gardella} and/or \cite{PH}, \cite{Gardella_Lupini17}, \cite{cortinas_rodrogiez}, \cite{cgt}, \cite{Gardella_Thiel2}, \cite{Austad_Ortega}, \cite{Hetland_Ortega}. 
A number of significant $C^*$-algebraic constructions and results have been carried over to the $L^p$-setting. 
For $p\neq 2$ the proofs usually require different techniques, and also  some new phenomena occur. 
An initial motivation behind the present paper was to solve \cite[Problem 8.2]{Gardella_Lupini17} which asks whether a standard simplicity criterion for groupoid $C^*$-algebras works for $L^p$-groupoid algebras. 
Working on this, we realized we can develop a neat theory of Banach algebras associated to twisted \'etale groupoids that not only allows us to prove strong results, such as simplicity and pure infiniteness, see \cite{BK}, \cite{BKM}, but also gives a deep insight into representation theory that allows us to define and study interesting Banach algebras in terms of generators and relations. 
This yields a general theoretical background that covers $L^p$-analogs of Cuntz algebras \cite{PhLp2a}, graph $L^p$-algebras \cite{cortinas_rodrogiez},  Cuntz-like Banach algebras \cite{Daws_Horwath}, (twisted) transformation group Banach algebras \cite{DDW}, \cite{Phillips},   AF-algebras, $\ell^p$-uniform Roe algebras \cite{Spakula_Willett}, \cite{Chung_Li}, groupoid $L^p$-operator algebras  \cite{Gardella_Lupini17}, \cite{cgt}, or reduced twisted groupoid Banach algebras  \cite{Austad_Ortega}, \cite{Hetland_Ortega}, which have so far often been studied using ad hoc methods. 
This theory has a potential to be applied to numerous other constructions. To illustrate this, in the present paper we introduce Banach algebras associated to inverse semigroups and to directed graphs. 
Further examples and applications can be found in  \cite{BKM}. 
Also, since all classifiable simple $C^*$-algebras are modeled by twisted \'etale groupoids \cite{Li}, it is reasonable to expect such models to play a crucial role in other contexts. 
On top of that we consider general Banach algebras over the field $\F=\R,\C$ of real or complex numbers, so in particular our results apply to real $C^*$-algebras,  which has recently have been experiencing a renaissance  due to numerous important applications, see \cite{Rosenberg}, \cite{Blecher} and references therein.
Real crossed products and real Cuntz $C^*$-algebras were already considered, see \cite{Schroder}, \cite{Rosenberg}.

One of the fundamental ideas, that stands behind the success of the theory of $C^*$-algebras associated to \'etale groupoid $C^*$-algebras, and Cartan $C^*$-subalgebras in particular, is that algebras associated to such groupoids can be viewed as crossed products by actions of (discrete) inverse semigroups, cf. \cite{Renault_book}, \cite{Paterson}, \cite{Sims}, 
\cite{Kumjian0}, \cite{Exel}, \cite{Buss_Exel}, \cite{Buss_Exel2}, \cite{BHM} \cite{Kwa-Meyer}. 
We make this idea a starting point and the main tool of our analysis.
In particular, we generalize the construction of crossed products by twisted inverse semigroup actions for $C^*$-algebras~\cite{Sieben98}, \cite{Buss_Exel}, to the realm of Banach algebras. 
Since the category of (real or complex) $C^*$-algebras with $*$-homomorphisms as morphisms is a full subcategory of Banach algebras with contractive homomorphisms as morphisms, we decided to work in this Banach algebra category. 
Thus by a \emph{representation of a Banach algebra} $A$ in a Banach algebra $B$ we mean a contractive homomorphism $\pi:A\to B$. 
For an inverse semigroup $S$ a representation is a semigroup homomorphism $v : S\to B_1\subseteq B$ into contractive elements of $B$. 
Then the operators $\{v_t\}_{t\in S}\subseteq B$ are \emph{partial isometries} in the sense of Mbekhta~\cite{Mbekhta} (which are the usual partial isometries when $B$ is a $C^*$-algebra). 
For a \emph{twisted  action} $(\alpha, u)$, \`{a} la Buss--Exel~\cite{Buss_Exel}, of the inverse semigroup $S$ on a Banach algebra $A$ we define the crossed product $A \rtimes_{(\alpha, u)} S$ as a Banach algebra universal for \emph{covariant representations} $(\pi,v)$ of $(\alpha, u)$,
where $\pi$ is a representation of $A$ in $B$ and $v$ is a twisted representation of $S$ in $B$ satisfying natural relations similar to those in~\cite{Sieben98}, \cite{Buss_Exel}.
To get uniqueness of \emph{disintegration} of  representations of $A \rtimes_{(\alpha, u)} S$ we need to impose a \emph{normalization} condition on covariant representations. 
In different situations this could mean different things.
When the domains of the action are not unital this requires weak limits and so normalization works well for instance when $(B,B_*)$ is a dual Banach algebra in the sense of Runde, see \cite{Runde}, \cite{Daws}. The latter case covers  covariant representations on reflexive Banach spaces $E$. In general, one needs to pass to the double dual algebra $B''$ with one of the Arens products. 
We discuss these issues in detail in Section \ref{sec:Inverse_semigroup_crossed_products}.

For a twisted \'etale groupoid $(\G,\LL)$,   with a locally compact Hausdorff unit space $X$,  the Connes $*$-algebra $\mathfrak{C}_c(\G,\LL)$ of quasi-continuous compactly supported sections  has plenty of natural Banach algebra completions.
We decided to work only with norms that coincide with the supremum norm $\|\cdot\|_{\infty}$ for sections with supports in open bisections (apparently all norms considered so far in the literature have this property). 
For any unital  inverse semigroup $S$ of open bisections covering $\G$ we define the \emph{twisted groupoid Banach algebra} $F^S(\G,\LL)$  of $(\G,\LL)$.
If $S$ is wide and consists of bisections where the bundle $\LL$ is trivial (which can be always arranged), we have an inverse semigroup action $(\alpha, u)$ of $S$ on $C_0(X)$ that can be used to recover $(\G,\LL)$ (this is implicit in \cite{Buss_Exel}, \cite{Buss_Exel2}, and we make it explicit in Subsection \ref{subsec:twisted_groupoids_vs_twisted_actions}).
Our first main result (Theorem~\ref{thm:disintegration}) shows that   we have an isometric isomorphism
\[
    F^S(\G,\LL)\cong C_0(X)\rtimes_{(\alpha,u)} S
\]
that establishes a bijective correspondence between representations of $F^S(\G,\LL)$ and normalized covariant representations of $(\alpha,u)$. 
In our second disintegration theorem we use that the action $(\alpha, u)$ of $S$ on $C_0(X)$ extends naturally to an action $(\widetilde{\alpha}, u)$ of $S$ on the algebra $\B(X)$ of bounded Borel functions on $X$. We show that  the associated crossed product may be identified with  a closure of certain convolution 
algebra $\B^S(\G,\LL)$ of Borel sections of $\LL$ (see Theorem \ref{thm:Borel_extension_of_representations})
so that we have
$$
F^S(\G,\LL)=\overline{\mathfrak{C}_c(\G,\LL)}^{\| \cdot \|_{\max}^{S}}\subseteq   \overline{\B^S(\G,\LL)}^{\| \cdot \|_{\max}^{S}}\cong\B(X)\rtimes_{(\widetilde{\alpha},u)} S.
$$
Moreover, every representation of $F^S(\G,\LL)$ on a dual Banach space $E$ extends to a representation of $\B(X)\rtimes_{(\widetilde{\alpha},u)} S$,
and so it disintegrates to a covariant representation of $(\widetilde{\alpha}, u)$ on $E$. 
This can be viewed (see Corollary \ref{cor:Borel_for_C*-algebras}) as a far reaching generalization of similar recent results in \cite{BussMartinez}, \cite{BGHL}, which   were recently used in the study of amenability of the groupoid $C^*$-algebra $C^*(\G)$ when $\G$ is non-Hausdorff. We believe this tool will be useful in other contexts.

These results can be applied to Banach algebras $F_{\RR}^S(\G,\LL)$ which are universal for representations of $F^S(\G,\LL)$ from a fixed class $\RR$ of representations.
For $p\in [1,\infty]$  we define the \emph{universal twisted groupoid $L^p$-operator algebra} $F^p(\G,\LL)$ as $F_{\RR}^S(\G,\LL)$, where $\RR$ is the class of all  representations on spaces $L^p(\mu)$  whose restriction to $C_0(X)$ is positive in Banach lattice sense. Here the main structural result (Theorem \ref{thm:initial_on_L^p_full}) is that  this definition does not depend on the choice of $S$, and in the complex case  $F^p(\G,\LL)$ is universal for all representations  in all $L^p$-operator algebras.
In fact, our disintegration theorem allows us to prove (see Theorem~\ref{thm:L_p_norm_estimates})
that denoting by $\|\cdot \|_{L^p}$  the norm in $F^p(\G,\LL)$,  for all $p,q\in [1,\infty]$ with $1/p+1/q=1$, we have
the following  estimates
\[
    \| f \|_{L^p} \leq \| f \|_{L^1}^{1/p} \, \| f \|_{L^\infty}^{1/q} \leq \|f\|_{I}, \qquad f\in  \mathfrak{C}_c(\G,\LL) 
\]
where  
\[
    \| f \|_{L^1} := \sup_{x\in X}\sum_{d(\gamma)=x} |f(\gamma)|,\qquad  \|f\|_{L^\infty} := \sup_{x\in X}\sum_{r(\gamma)=x} |f(\gamma)|
\]
 and  $\|f\|_{I} := \max\{\|f\|_{L^1},\|f\|_{L^{\infty}}\}$ is Hahn's \emph{$I$-norm}. 
These estimates  are already interesting when $p=q=2$, as then they show that the maximal $C^*$-norm on $\mathfrak{C}_c(\G,\LL)$ is $I$-norm bounded, a fact that for non-Hausdorff $\G$ 
is quite non-trivial and was only recently proved in \cite{Clark_Zimmerman} under the assumption that $\G$ is  second countable and there is no twist. 
These estimates imply  that $L^p$-operator algebra crossed products agree with $L^p$-operator algebras of the corresponding transformation groupoids --  the problem  formulated below \cite[Proposition 6.5]{cgt}. They also show that our definition  agrees with the one given by Gardella--Lupini \cite{Gardella_Lupini17} (in the complex, untwisted, separable, finite $p$ case). In general, the algebra $F^p(\G,\LL)$ is a universal model for \emph{the reduced twisted groupoid $L^p$-operator algebra} $F^p_{\red}(\G,\LL)$ given by a concrete representation, which was recently studied by Hetland--Ortega \cite{Hetland_Ortega}, and by   Choi--Gardella--Thiel \cite{cgt}. One more  important consequence of the above norm estimates is that for $p=1,\infty$ the  full and reduced groupoid algebras always coincide.

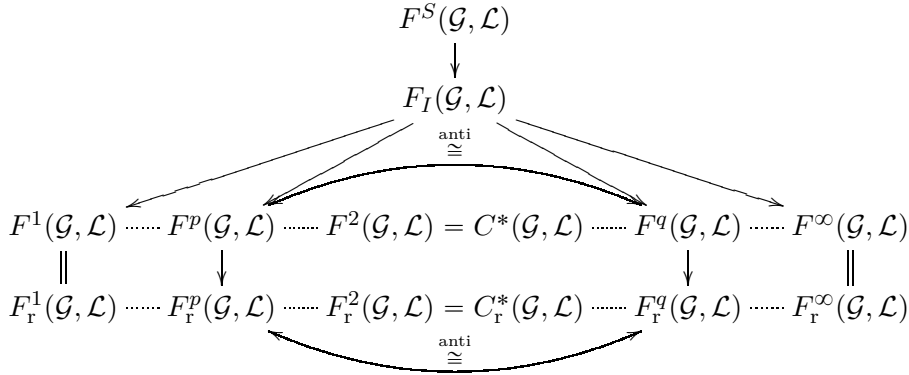
\begin{figure}[ht]
$
\xymatrixcolsep{1pc} \xymatrixrowsep{1pc} 
\xymatrix{
 &  &  \ar@{->}[lldd]  \ar@{->}[ldd] F_I(\G,\LL)   \ar@{->}[ddr]  \ar@{->}[ddrr]&  &  
\\
 &  &  &  &  
\\  
F^1(\G,\LL) \ar@{=}[d] \ar@{.}[r] & F^{p} (\G,\LL) \ar@{->}[d]  \ar@{.}[r]   \ar@/^2pc/[rr]^{\anti} & F^{2} (\G,\LL)= C^*(\G,\LL)   \ar@{.}[r]  & \ar@/_2pc/[ll] F^{q} (\G,\LL) \ar@{->}[d] \ar@{.}[r] & 
 F^\infty(\G,\LL)  \ar@{=}[d] 
\\
F^1_{\red}(\G,\LL) \ar@{.}[r] 
& F^{p}_{\red} (\G,\LL)  \ar@{.}[r]   \ar@/_2pc/[rr]  & F^2_{\red} (\G,\LL)= C^*_{\red}(\G,\LL)   \ar@{.}[r]  &  F^{q}_{\red} (\G,\LL)   \ar@/^2pc/[ll]_{\anti} \ar@{.}[r] & 
 F^\infty_{\red}(\G,\LL) 
}
$
\caption{Hierarchy of twisted groupoid Banach algebras}\label{fig.algebras}
\end{figure} 
The hierarchy of the discussed  algebras is summarized in Figure~\ref{fig.algebras}, where
all algebras are completions of $\mathfrak{C}_c(\G,\LL)$, in particular, $F_I(\G,\LL) := \overline{\mathfrak{C}_c(\G,\LL)}^{\|\cdot\|_{I}}$, and $C^*(\G,\LL)$ and $C^*_{\red}(\G,\LL)$ are standard $C^*$-algebras associated to $(\G,\LL)$. The algebras in the middle column are Banach $*$-algebras, the horizontal anti-isomorphisms are given by the involution in $\mathfrak{C}_c(\G,\LL)$, $1/p+1/q=1$, and the downward arrows are representations extending identities on $\mathfrak{C}_c(\G,\LL)$:

As another application of our disintegration theorems, we  give a geometric description of representations  of $F^{p} (\G,\LL)$ on spaces $L^p(\mu)$,  where $\mu$ is a localizable measure. 
Following Phillips~\cite{PhLp1} and Gardella~\cite{Gardella} we define \emph{spatial partial isometries} as certain \emph{weighted composition operators} on $L^p(\mu)$. 
They form a natural inverse semigroup $\SPIso(L^p(\mu))$ of partial isometries in $B(L^{p}(\mu))$. 
In fact, using the well established notion of an \emph{$L^p$-projection}, see \cite{Agniel}, \cite{BDEGGMM}, we define \emph{$L^p$-partial isometries} on an arbitrary Banach space $E$. 
For $p\neq 2$, they form an inverse semigroup, and when $E=L^p(\mu)$ they coincide with spatial partial isometries by the (generalized) Lamperti Theorem (Theorem~\ref{thm:spatial^partial_isometries_description}). 
Our third main result (Theorem \ref{thm:spatial_representations_of_groupoid_algebras}) says that for every $p\in [1,\infty]$,
the algebra $F^p(\G,\LL)$ is universal for \emph{spatial covariant representations}, i.e.\ covariant representations $(\pi,v)$ where  $v:S\to \SPIso(L^p(\mu))$ takes values in spatial partial isometries and $\pi:C_0(X)\to B(L^p(\mu))$ acts by multiplication operators (by functions from  $L^{\infty}(\mu)$).
This in particular gives a natural bijective correspondence between spatial representations of $F^p(\G,\LL)$ and $F^{p'}(\G,\LL)$ for any $p,p'\in (1,\infty)$, 
which is quite surprising as  in many cases there is no non-zero continuous homomorphism from $F^p(\G,\LL)$ to $F^{p'}(\G,\LL)$.  In the complex case, for $p \in (1,\infty)\setminus\{2\}$,  every non-degenerate representation of $F^p(\G,\LL)$ on $L^{p}(\mu)$ is spatial.

As an illustration, and having future applications in mind, see e.g.\ \cite{BKM}, we introduce the universal and the \emph{tight Banach algebra of an inverse semigroup} $S$.
The new paradigm that appears here is that for general representations $v : S \to B$ into a Banach algebra $B$ we need to impose a condition that certain associated projections are \emph{jointly $\F$-contractive} (see Definitions~\ref{defn:jointly_contractive} and \ref{defn:representation_semigroup_contractive}), which is related to a number of similar conditions  studied by various authors, cf. \cite{Agniel}, \cite{Bernau_Lacey2}, \cite{Jamison}, \cite{Stacho_Zalar}. 
This condition is automatic when $B$ is a $C^*$-algebra or when $v : S \to \SPIso(L^p(\mu))$ takes values in spatial partial isometries, and therefore has not been visible in previous works.
The tight algebras are modeled by the tight groupoid of Exel \cite{Exel}.  Based on the algebraic setting of Steinberg--Szak\'{a}cs~\cite{Steinberg_Szakacs} and our results, we characterise when representations of such algebras are isometric on the unit subalgebra, which leads to convenient improvements even for the already existing theory of tight $C^*$-algebras.
 
We  define the \emph{Banach algebra $F(Q)$ of a directed graph} $Q$ as the universal algebra for a Banach $Q$-family of partial isometries and projections,
and use our results to show it is naturally isomorphic to a groupoid Banach algebra and a tight inverse semigroup Banach algebra. 
Our results imply that non-degenerate representations of the complex Banach algebra $F(Q)$ on $L^p$-spaces, for $p\in (1,\infty)\setminus\{2\}$, are necessarily given by $Q$-families with values in spatial partial isometries $\SPIso(L^p(\mu))$. 
Thus the use of spatial partial isometries in the definition of the graph $L^p$-operator algebras $F^p(Q)$ in \cite{cortinas_rodrogiez} and \cite{PhLp1} is not an assumption --- it is forced by the groupoid model. When $Q$ has only one vertex, $F(Q)$ is a Banach Cuntz algebra which seems to be more tractable than the associated Cuntz-like Banach $*$-algebra of  Daws--Horv\'{a}th~\cite{Daws_Horwath} (cf.\ Remark~\ref{rem:Cuntz-algebras}).



The present paper is organized as follows. 
After preliminaries on $L^p$-operator algebras, groupoids and partial isometries in Banach algebras, in Section \ref{sec:Preliminaries}, we define crossed products for twisted inverse semigroup actions on Banach algebras in Section~\ref{sec:Inverse_semigroup_crossed_products}. 
Section~\ref{sec:GroupoidBanachAlgebras} introduces the title concept - twisted groupoid Banach algebras, and presents the general inverse semigroup disintegration theorems. 
In Section~\ref{sec:Groupoid Lp-operator algebras} we analyze representations and $L^p$-operator algebras associated to $(\G,\LL)$. Section \ref{subsec:Banach_algebras_inverse_semigroups} is devoted to inverse semigroup algebras and tight groupoids.

%
\subsection*{Acknowledgements}

The first two named authors were supported by the National Science Center (NCN) grant no.~2019/35/B/ST1/02684.
The second named author  was  supported by the National Science Centre, Poland, through the WEAVE-UNISONO grant no. 2023/05/Y/ST1/00046. 
The second named author would like to thank Chris Phillips, Eusebio Gardella and Hannes Thiel for sharing their expert knowledge so generously, during various meetings and via e-mail. 
We thank Jan Gundelach for the careful reading of the first draft of the paper and his comments.
Finally, we  thank the anonymous referee for finding a gap in the proof of Theorem~\ref{thm:L_p_norm_estimates} when $\G$ is non-Hausdorff.
To fix it efficiently, we developed the Borel extension-disintegration  described in  Subsection~\ref{sec:Borel_disintegration}.

\section{Representations of inverse semigroups and $C_0(X)$}
\label{sec:Preliminaries}

Most of the facts presented here are probably known to experts, so we included only sketches of, or coordinates leading to, the arguments when we were not able to find a reference in appropriate generality. 
We discuss representations of $C_0(X)$ on  $L^p$-spaces with $p\in [1,\infty]$, and the key result here is Theorem \ref{thm:L^p_representations of C_0(X)}. For an introduction to $L^p$-operator algebras we refer to \cite{Gardella}. We generalize spatial partial isometries of Phillips and Gardella by defining $L^p$-partial isometries on an arbitrary Banach space $E$. The link between the two is provided by a version of the Banach--Lamperti theorem (Theorem \ref{thm:spatial^partial_isometries_description}).
Also we recall the fundamental relationship between actions of inverse semigroups on topological spaces and \'etale groupoids.

\subsection{Spatial isometries on $L^p$-spaces}

Throughout we fix $p \in [1,\infty]$ and the field $\F = \R,\C$ of real or complex numbers. 
If $(\Omega,\Sigma, \mu)$ is a measure space we write $L^p(\mu)$ for the corresponding Lebesgue space viewed as a Banach space over $\F$. 
Occasionally, to emphasize which field is used, we will write $L^p(\mu, \F)$. 
Recall that the Banach space $L^p(\mu)$ can be recovered from $\mu : \Sigma \to [0,\infty]$, or in fact from the measure $[\mu] : [\Sigma] \to  [0,\infty]$ defined on the Boolean algebra $[\Sigma]$ obtained from $\Sigma$ by identifying sets of measure zero (the presentation in \cite{Gardella}, \cite{Gardella_Thiel2} uses this `point-free' approach to $L^p$-spaces, but we stay with the old-fashioned approach).
Also if $p<\infty$, then $L^p(\mu) \cong L^p(\mu_{f})$ where $\mu_{f}(A) = \sup\{\mu(B) : A\supseteq B\in \Sigma, \mu(B)<\infty \}$ is the \emph{semi-finite} part of $\mu$, and the isomorphism is identity on integrable simple functions. 
Recall that a measure $\mu$ is called \emph{localizable} if it is semi-finite and $[\Sigma]$ is Dedekind complete.
In general, if $\mu$ is semi-finite, then $[\mu]$ extends uniquely to the Dedekind completion of $[\Sigma]$ giving a localizable measure on the completion \cite{Fremlin}. 
Using the Loomis--Sikorski Theorem we may represent this abstract measure by a concrete localizable measure space $(\overline{\Omega},\overline{\Sigma}, \overline{\mu})$ (see \cite{Fremlin}, \cite{BGL}) and again $L^p(\mu)\cong L^p(\overline{\mu})$ by the isometric isomorphism sending characteristic functions to characteristic functions. 
Thus when considering $L^p$-spaces with $p<\infty$ we may always assume the measure is localizable, or even decomposable (strictly localizable), see \cite[Corollary on page 136]{Lacey} or \cite[Theorem 7.17]{BGL}).
For $p = \infty$ the assumption that $\mu$ is localizable is also natural. For instance, $\mu$ is localizable if and only if $L^{\infty}(\mu)$ is canonically isomorphic to the dual of $L^1(\mu)$ if and only if $L^{\infty}(\mu)$ is canonically a von~Neumann algebra acting on $L^2(\mu)$ (and every abelian $W^*$-algebra is of this form), see \cite{BGL} and references therein.

\begin{defn}\label{de:SetIsomorphism}
A \emph{set isomorphism} from a measure space $(\Omega_\mu,\Sigma_\mu, \mu)$ onto a measure space $(\Omega_\nu,\Sigma_\nu, \nu)$ is a map $\Phi : \Sigma_\nu \to \Sigma_\mu$ that descends to an isomorphism  of Boolean algebras $[\Phi] : [\Sigma_{\nu}] \to [\Sigma_{\mu}]$. 
Thus, since the Boolean structure is determined by its preorder, it suffices to assume
\begin{enumerate} 
    \item\label{item:set_automorphism1} $\nu(C\setminus D)=0$ if and only if $\mu(\Phi(C)\setminus \Phi(D)) = 0$ for all $C, D\in \Sigma_{\nu}$;
        \item\label{item:set_automorphism2} for each $C\in \Sigma_{\mu}$ there is $\Phi^*(C)\in \Sigma_{\nu}$ such that $\Phi(\Phi^*(C)) = C$ up to a $\mu$-null set.
\end{enumerate} 
Condition \ref{item:set_automorphism2} gives a map $\Phi^* : \Sigma_{\mu}\to \Sigma_{\nu}$, which is a  set isomorphism that descends to an isomorphism $[\Phi^*] : [\Sigma_{\mu}] \to [\Sigma_{\nu}]$ inverse to $[\Phi]:[\Sigma_{\nu}]\to [\Sigma_{\mu}]$.
\end{defn}

We denote by $L_0(\mu)$ the algebra of $\F$-valued measurable functions identified up to equality $\mu$-almost everywhere. 
A set isomorphism $\Phi : \Sigma_\nu \to \Sigma_\mu$ uniquely defines a linear operator $T_{\Phi} : L_0(\nu) \to L_0(\mu)$ such that 
\[ 
    T_{\Phi} 1_{C} = 1_{\Phi(C)}, \qquad C\in \Sigma_{\nu},
\]
and $T_{\Phi}$ preserves (monotone) limits, cf. \cite[Proposition 5.6]{PhLp1}. 
This is an algebra isomorphism $L_0(\nu) \cong L_0(\mu)$ whose inverse is $T_{\Phi^*}$.
We call $T_{\Phi}$ a (generalized) \emph{composition operator}. 
It preserves the range of functions and in particular it restricts to an isometric isomorphism  $L^{\infty}(\nu) \cong L^{\infty}(\mu)$ \cite[Proposition 7.16]{BGL}. 
However, unless $\Phi$ preserves measures, it does not preserve the $L^p$-spaces for $p < \infty$. 
To force this, we will consider \emph{weighted composition operators} $a T_{\Phi} : L_0(\nu) \to L_0(\mu)$ where $a \in L_0(\mu)$ and $[a T_{\Phi}] \xi (x) := a(x) T_{\Phi}(\xi) (x)$.
Note that $\mu\circ \Phi$ is a measure equivalent to $\nu$. Thus if $\nu$ is localizable then 
the Radon--Nikodym derivative $\frac{\mu\circ \Phi}{\nu}$ exists, and is strictly positive $\nu$-almost everywhere, see \cite[Theorem 2.7]{Gardella_Thiel2}, \cite[Theorem 4.4]{BGL}.
Similar comments apply to the measures $\nu \circ \Phi^*$ and $\mu$.

\begin{lem}\label{le:SpatialIsometries}
Consider localizable measures $\mu$ and $\nu$ and $p\in[1,\infty]$.    
For any set isomorphism $\Phi:\Sigma_\nu\to \Sigma_\mu$ the weighted composition operator 
\[
    U_{\Phi} := \left( \frac{d\nu\circ\Phi^{*}}{d\mu} \right)^{\frac{1}{p}} T_{\Phi}
\]
is an invertible isometry $U_{\Phi} : L^p(\nu) \to L^p(\mu)$.
Moreover, the group $UL^{\infty}(\mu)$ of measurable modulus one functions $\omega \in L_0(\mu)$  embeds into the group of invertible isometries of $L^p(\mu)$ as multiplication operators. 
Thus for any $\omega\in UL^{\infty}(\mu)$ the weighted composition operator $\omega U_{\Phi} : L^p(\nu) \to L^p(\mu)$ is an invertible isometry.
\end{lem}
\begin{proof}
This is straightforward and known, see, for instance, \cite[Lemmas 3.2, 3.3]{Gardella_Thiel2}. 
\end{proof}

\begin{defn}[\cite{PhLp1}]\label{de:SpatialIsometries}
We call the invertible isometries $\omega U_{\Phi}$ described in Lemma~\ref{le:SpatialIsometries} 
\emph{spatial}.
\end{defn}

\begin{prop}\label{prop:group_of_spatial_isometries} 
We have $(\omega U_{\Phi})^{-1} = T_{\Phi^*}(\overline{\omega}) U_{\Phi^*}$ and $\omega U_{\Phi} \circ  \upsilon  T_{\Psi}= \omega T_{\Phi}(\upsilon) U_{\Phi \circ \Psi}$ for all localizable measures $\mu$, $\nu$, $\eta$, all $p\in[1,\infty]$, set isomorphisms $\Phi : \Sigma_\nu \to \Sigma_\mu$, $\Psi : \Sigma_\nu \to \Sigma_\eta$ and functions $\omega \in UL^{\infty}(\mu)$, $\upsilon \in UL^{\infty}(\nu)$. 

In particular, invertible spatial isometries on $L^p(\mu)$, for localizable $\mu$ and $p \in [1,\infty]$, form a group which is isomorphic to the semi-direct product $UL^{\infty}(\mu) \rtimes \Aut([\Sigma_\mu])$
where $[\Phi] \omega =T_{\Phi}(\omega)$ for $[\Phi]\in \Aut([\Sigma_\mu])$ and $\omega\in UL^{\infty}(\mu)$.
\end{prop}
\begin{proof}
The first part follows from (the calculations in the proof of) \cite[Lemmas 3.3, 3.4]{Gardella_Thiel2}. It readily implies the second part of the assertion,
see also \cite[Theorem 3.7]{Gardella_Thiel2} or \cite[Proposition 2.10]{Gardella}.
\end{proof}

\begin{prop}[Banach--Lamperti theorem]\label{prop:Banach_Lamperti} 
Let $p \in [1,\infty] \setminus \{ 2 \}$ and $\mu , \nu$ be localizable measures. 
Then every invertible isometry from $L^p(\mu)$ to $L^p(\nu)$ is spatial.
\end{prop}
\begin{proof}  
For $p \in (1,\infty)\setminus \{2\}$, $\F = \C$, and $\mu = \nu$ this is proved in \cite[Theorem 3.7]{Gardella_Thiel2}, see also \cite[Theorem 2.12]{Gardella}, but the proof works for two localizable measures $\mu$ and $\nu$, and also for $\F=\R$ and $p=1$. 
In fact since $L^{\infty}(\mu) \cong L^1(\mu)'$ the case $p=1$ can be deduced from the case $p=\infty$. 
The assertion for $p=\infty$ follows from the Banach--Stone Theorem combined with the Gelfand Theorem. Indeed,  we have $L^{\infty}(\mu)\cong C(X)$ and $L^{\infty}(\nu)\cong C(Y)$ where $X$, $Y$ are compact Hausdorff (hyperstonean) spaces, and every invertible isometry $C(Y)\to C(X)$ is of the form $a T$ where $a\in C(X)$, $|a|=1$, and $T:C(Y)\to C(X)$ is an isometric algebra isomorphism. 
The Boolean algebras of idempotent elements in $C(Y)$ and $C(X)$ are canonically isomorphic to $[\Sigma_{\nu}]$ and $[\Sigma_{\mu}]$, respectively. 
Thus $T$ induces a set isomorphism $\Phi : \Sigma_\nu\to \Sigma_\mu$ such that the action of
$T_{\Phi}:L^\infty(\nu)\to L^\infty(\mu)$ on idempotents in $L^\infty(\nu)$ agrees with the action of $T$ on the corresponding idempotents in $C(Y)$.
As the idempotents are linearly dense in the considered  spaces we conclude that the isometry $L^{\infty}(\nu)\to L^{\infty}(\mu)$ corresponding to $a T$ is $\omega T_{\Phi}$ where 
$\omega\in L^{\infty}(\mu)$ is the function corresponding to $a\in C(X)\cong L^{\infty}(\mu)$. 
\end{proof}

\subsection{$L^p$-operator algebras and representations of $C_0(X)$}

Throughout this paper $A$ and $B$ will be  Banach algebras, and $B(E)$ will denote the Banach algebra of bounded operators on a Banach space $E$.

\begin{defn}\label{de:Representations}
A \emph{representation of $A$ in  $B$} is a contractive algebra homomorphism $\pi : A \to B$. 
If $B = B(E)$ for some Banach space $E$, we call $\pi:A\to B(E)$ a \emph{representation on the space $E$}, and we say $\pi$ is \emph{non-dedegenerate} if $\overline{\pi(A)E} = E$. 
\end{defn}


\begin{defn}[\cite{Phillips}] \label{def:L^p_operator_algebra}
Let $p\in[1,\infty]$. 
A Banach algebra $A$ is an \emph{$L^p$-operator algebra} if there is an isometric representation $\pi : A \to B(L^p(\mu))$ for some measure $\mu$.
\end{defn}
By an \emph{approximate unit} in a Banach algebra we will always mean a contractive two-sided approximate unit. 
If $p\in (1,\infty)$, every  $L^p$-operator algebra $A$  with an approximate unit has an isometric non-degenerate representation  $\pi : A \to B(L^p(\mu))$,   \cite[Theorem 4.3]{Gardella_Thiel1}.
Unlike for $C^*$-algebras there is still no axiomatic definiton of an $L^p$-operator algebra.
\begin{defn}\label{def:C*-algebra}
A \emph{$C^*$-algebra} is a Banach $*$-algebra $A$ such that $\| a^*a \| = \| a \|^2$, $a\in A$, and if $\F=\R$ we also require $\| a^*a \| \leq \| a^*a + b^*b \|$ for all $a,b\in A$. 
We call a norm with these properties a \emph{$C^*$-norm}.
\end{defn}

\begin{rem}\label{rem:complex_the_real_C*}
It is well known that a Banach $*$-algebra $A$ (complex or real) is a $C^*$-algebra if and only if it is isometrically $*$-isomorphic to a $*$-subalgebra of bounded operators on some $L^2$-space (complex or real).
Also a real Banach algebra $A$ is a real $C^*$-algebra if and only if its complexification $A_{\C} := A + i A$ admits a (necessarily unique) $C^*$-norm that extends that of $A$, see \cite{Li_book}.
A homomorphism $A\to B$ between two $C^*$-algebras is contractive if and only if it is $*$-preserving 
(these well known facts for complex $C^*$-algebras \cite[A.5.8]{BM} also hold in the real case \cite[Proposition 5.4.1 and Theorem 5.6.4]{Li_book}). 
\end{rem}

Let $X$ be a locally compact Hausdorff space. 
The algebra $C_0(X)$ of continuous functions on $X$ which vanish at infinity, equipped with the supremum norm $\|\cdot\|_{\infty}$, is an $L^p$-operator algebra for any $p\in [1,\infty]$. 
By \cite[Theorem 5.3]{Gardella_Thiel1} a complex $C^*$-algebra $A$ is an $L^p$-operator algebra for $p\in[1,\infty)\setminus\{2\}$ if and only if it is commutative and hence of the form $C_0(X)$.
A real commutative Banach algebra $A$ is isometrically isomorphic to $C_0(X)$ for some locally compact Hausdorff space if and only if $\| a \|^2 = \| a^2 \| \leq \| a^2 + b^2 \|$ for all $a,b\in A$, see \cite{Arens}. 

\begin{prop}[Kaplansky--Bonsall]
\label{prop:minimality_of_sup_norm} 
Let $A$ be a $C^*$-algebra (real or complex). 
The $C^*$-norm on $A$ is minimal among all submultiplicative norms on $A$. 
Thus every injective representation $\pi : A \to B$ in a Banach algebra $B$ is automatically isometric. 
More generally, every representation $\pi : A \to B$ descends to an isometric representation $\pi:A/\ker(\pi)\to B$. 
\end{prop}
\begin{proof} 
The first part of the assertion in the complex case is \cite[Theorem 10]{Bns}, but the proof also works in the real case by applying \cite[Lemma 5.1.8]{Li_book}. 
The second part is equivalent to the first. For the last part it suffices to note that the quotient $A/\ker(\pi)$ is again a $C^*$-algebra, which is well known when $\F=\C$, but also holds when $\F=\R$, see \cite[Proposition 5.4.1]{Li_book}.
\end{proof}


We now slightly extend \cite[Proposition 2.7]{cgt}, see also \cite[Proposition 2.12]{BP} and references therein, which describes hermitian operators acting on complex $L^p$-spaces.
This forces the representations of the complex algebra $C_0(X)$ on $L^p$ and $C_0$-spaces to act by multiplication operators.

\begin{prop}\label{prop:hermitian_operators} 
Assume $\F=\C$  and let $E = L^p(\mu)$ for a localizable $\mu$ and $p \in [1,\infty] \setminus \{ 2 \}$, or $E = C_0(\Omega)$ for a locally compact Hausdorff space $\Omega$.  
An element $a\in B(E)$ is hermitian (i.e.\ $\| e^{ita} \| = 1$ for all $t\in \R$) if and only if $a$ is a multiplication operator by a bounded real-valued function (from $L^{\infty}(\mu,\R)$ if $E=L^p(\mu)$ and from $C_{b}(\Omega,\R)$ if $E=C_0(\Omega)$).
\end{prop}
\begin{proof}
When $E = L^p(\mu)$ and $p < \infty$ this is proved in \cite[Proposition 2.7]{cgt}. 
The proof relies on the Banach--Lamperti Theorem and it also works for $p = \infty$, with the Banach--Lamperti theorem extended as in Proposition~\ref{prop:Banach_Lamperti}. 
We show it also works for $E = C_0(\Omega)$ with the Banach--Lamperti theorem replaced by the Banach--Stone theorem. 
We recall the argument for the `only if' part (the `if' part is straightforward).
Let $a \in B(C_0(\Omega))$ be hermitian. We may assume that $\| a \| \leq \frac{\pi}{2}$. 
The elements $u_t := e^{ita}$, $t\in \R$, form a continuous group of invertible isometries. 
By the Banach--Stone theorem for each $t\in \R$ there is $\omega_{t} \in C_b(\Omega)$ with $|\omega_t| = 1$ 
and a homeomorphism $\varphi_t : \Omega \to \Omega$ such that $[u_{t}\xi](x) = \omega_t(x) \xi(\varphi_t(x))$. 
Using this, a simple calculation,  for $s,t\in \R$, gives
$
    \| u_t - u_s \| = \max \big\{ \| \omega_t - \omega_s \|_{\infty}, 2(1-\delta_{\varphi_{t} , \varphi_s}) \big\}.
$
Since $t \mapsto u_t$ is norm continuous and $u_0 = 1$ is the identity operator, we get that $\varphi_t = \id$ for all $t\in\R$. 
Hence each $u_t$ is the operator of multiplication by $\omega_t$. 
Since $\| a \| \leq \frac{\pi}{2}$ the spectrum of $u_1 = e^{ia}$ is contained in the half circle $\{ e^{it} : t\in [-\pi/2,\pi/2] \}$ where the holomorphic inverse $\log$ of $\exp$ is defined. 
Applying analytic functional calculus to $u_1$ we get $i a = \log(u_1)$. 
Since $m: C_b(\Omega) \to B(C_0(\Omega))) ;\ m(h)\xi (x) := h(x)\xi(x)$ is a unital homomorphism, we obtain $a = -i \log(u_1) = -i \log(m(\omega_1))= m(-i \log(\omega_1))$.
Thus $a$ is a multiplication operator by $-i \log(\omega_1)$, and since $a$ is bounded and hermitian we must have $-i \log(\omega_1)\in C_b(\Omega,\R)$.
\end{proof}

Recall that $L^p$-spaces and $C_0$-spaces are naturally Banach lattices with the lattice structure given by positive (non-negative) functions. 
Operators preserving positive functions (and therefore the lattice structure) are called \emph{positive}. 

\begin{lem}\label{lem:representaion_complexification}
Let  $E$ be a real Banach space of the form $L^p(\mu,\R)$, $p\in [1,\infty]$ or $C_0(\Omega,\R)$, and let $E_\C$ be its Banach lattice complexification, that is either $L^p(\mu,\C)$  or $C_0(\Omega,\C)$.
A representation $\pi:C_0(X,\R) \to B(E)$ which is positive in the sense that it sends positive functions to positive operators, complexifies via the formula 
$
    \big( \pi_\C (a + i b) \big) (\xi +i \eta) := \pi(a)\xi - \pi(b)\eta + i \big( \pi(a)\eta +\pi(b) \xi \big),
$
$\xi, \eta\in L^p(\mu, \R)$, $a, b\in C_0(X,\R)$,
to a representation $\pi_\C : C_0(X,\C) \to B(E_{\C})$.
\end{lem}
\begin{proof} 
It is straightforward to see that $\pi_\C$ is a complex algebra homomorphism. 
Since $\pi$ is positive, both $\pi$ and $\pi_{\C}$ take values in the lattice of regular operators, see \cite[Section 3.2]{Abramovich_Aliptantis}. 
If $E=L^p(\mu, \R)$, then it is Dedekind complete and so by the Riesz--Kantorovich formula \cite[Corollary 3.25]{Abramovich_Aliptantis}, $| \pi_{\C}(a+ib) \xi | \leq \pi_{\C}(|a+ib|)|\xi|$ for all $\xi \in E_{\C}$.
This implies that $\|\pi_{\C}(a+ib)\|\leq \|\pi_{\C}(|a+ib|)\|=\|\pi(|a+ib|)\|\leq \||a+ib|\|_{\infty}=\|a+ib\|_{\infty}$.
Hence $\pi_{\C}$ is contractive. 
When $E=C_0(\Omega,\R)$ we may pass to duals to reduce to the $L^1$-case. 
Namely, $E'\cong L^1(\mu)$ for a measure $\mu$ and the formula $\pi'(a):=\pi(a)'$ defines a positive representation $\pi':C_0(X,\R) \to B(L^1(\mu))$. 
\end{proof}

\begin{thm} \label{thm:L^p_representations of C_0(X)} 
Let $\pi:C_0(X) \to B(E)$ be a non-degenerate representation where  $E = L^p(\mu)$ for a localizable $\mu$ and $p \in [1,\infty] \setminus \{ 2 \}$, or $E = C_0(\Omega)$ for a locally compact Hausdorff space $\Omega$.   
When $\F=\R$ we also assume that $\pi$ is  positive.  Then $\pi$ is given by multiplication operators: there is a representation $\pi_0 : C_0(X)\to L^{\infty}(\mu)$ if $E = L^p(\mu)$, or  $\pi_0 : C_0(X) \to C_b (\Omega)$ if $E = C_0(\Omega)$, such that 
 \[
   [\pi(a)\xi] (x) = \pi_0(a)(x) \xi (x), \qquad a\in C_0(X) ,\ \xi \in E.
\]
\end{thm}
\begin{proof} By Lemma \ref{lem:representaion_complexification} it sufices to consider the case $\F=\C$.
If $X$ is not compact we may extend $\pi : C_0(X) \to B(E)$ to a unital representation $\pi:C(X^+) \to B(E)$, where $X^+$ is the one-point compactification of $X$, see \cite[Theorem 4.1]{Gardella_Thiel1}. 
Then $\pi$ maps hermitian elements to hermitian ones by \cite[Lemma 2.4]{cgt}. 
Hermitian elements in $C(X^+)$ are real valued functions $C(X^+,\R)$. 
Hence $\pi(C(X^+,\R))$ consists of multiplication operators by Proposition~\ref{prop:hermitian_operators}.
Since $C(X^+) = C(X^+,\R) + i C(X^+,\R)$, we conclude that there is a desired $\pi_0$.
\end{proof}

When $\F=\R$ the positivity assumption in Lemma \ref{lem:representaion_complexification} and Theorem \ref{thm:L^p_representations of C_0(X)} is crucial:

\begin{ex}[Representations from involutive isometries] \label{ex:real_representation_from_bicontractive} 
Let $U\in B(E)$ be an involutive isometry on a Banach space $E$, i.e.\ $\|U\|=1$ and $U^2=1$.
Then $P := \frac{1+U}{2}$ and $1-P:=\frac{1-U}{2}$ are contractive projections, so  $P$ is a \emph{bicontractive projection} (cf. Remark~\ref{rem:types_of_projections} below). 
If $E$ is an $L^p$-space or a predual of an $L^1$-space, then every bicontractive projection is of this form, see \cite{Bernau_Lacey2}.
The formula 
$$
\pi(a)=a(1)P +a(2)(1-P)=\frac{a(1)+a(2)}{2} 1+ \frac{a(1)-a(2)}{2}U
$$
defines a homomorphism $\pi:C(\{1,2\})\to B(E)$ satisfying  $\|\pi(a)\|\leq \frac{|a(1)+a(2)|}{2} + \frac{|a(1)-a(2)|}{2}$.
If $\F=\R$, then $\frac{|a(1)+a(2)|}{2} + \frac{|a(1)-a(2)|}{2}=\max\{|a(1)|,|a(2)|\}=\|a\|_{\infty}$ and thus $\pi$ is contractive. For $\F=\C$ this is not true. 
In particular, let $\F=\R$, $E:=\ell^p(\{1,2\})\cong \R^2$, $p\in [1,\infty]$  and $U\xi:=(\xi(2), \xi(1))$. 
Then 
\[
    \pi(a)\xi = \left( \frac{a(1)+a(2)}{2}\xi(1) +\frac{a(1)-a(2)}{2}\xi(2), \frac{a(1)-a(2)}{2}\xi(1) +\frac{a(1)+a(2)}{2}\xi(2)\right)
\]
is an isometric unital representation $\pi : C(\{1,2\}) \to B(\ell^p(\{1,2\}))$ of the real algebra $C(\{1,2\})$ that does not act by multiplication operators.
For $p\neq 2$, $\pi$ is not even isometrically conjugate to a representation given by multiplication operators (there are only four isometries on $\ell^p(\{1,2\})$ by Proposition~\ref{prop:Banach_Lamperti}). For $p=2$, see Lemma~\ref{lem:Hilbert_space_representations of C_0(X)} below.
\end{ex}

\begin{cor} \label{cor:degeneracy_of_L_infty_representations}
The complex Banach algebra $C_0(X)$ admits an isometric non-degenerate representation on $L^{\infty}(\mu)$ for localizable $\mu$ if and only if $X$ is compact.
\end{cor}
\begin{proof}
If $X$ is not compact and $\pi:C_0(X) \to B(L^{\infty}(\mu))$ is isometric and non-degenerate then, in the notation of Theorem~\ref{thm:L^p_representations of C_0(X)}, $\pi_0(C_0(X))$ is a non-unital subalgebra of $L^{\infty}(\mu)$. 
Hence $\overline{\pi(C_0(X))L^{\infty}(\mu)} = \overline{\pi_0(C_0(X))L^{\infty}(\mu)}$ is a non-trivial  (non-unital) ideal in $L^{\infty}(\mu)$,
 which contradicts the non-degeneracy of $\pi$. If $X$ is compact, then the embedding   $\pi_0:C(X)\to\ell^{\infty}(X)$
gives an isometric unital representation of $C(X)$ on $\ell^{\infty}(X)$.
\end{proof}

\begin{rem}\label{rem:C_0_and_Lindenstrauss} 
Corollary~\ref{cor:degeneracy_of_L_infty_representations} is perhaps one of the reasons why Phillips~\cite[Page 3]{Phillips} suggests that instead of $L^\infty$-algebras ``it may well be more appropriate to consider'' algebras isometrically represented on \emph{$C_0$-spaces}, i.e.\ spaces $C_0(\Omega)$ for a locally compact Hausdorff space $\Omega$. 
In fact, it is also natural to consider pre-duals of $L^1$-spaces, i.e. \emph{Lindenstrauss spaces}, cf.\ Corollary~\ref{cor:L_infty_lindenstrauss_spaces_etc} below. 
\end{rem}

For $p=2$, Theorem~\ref{thm:L^p_representations of C_0(X)} holds up to a unitary conjugacy. 

\begin{lem} \label{lem:Hilbert_space_representations of C_0(X)} 
For any representation $\pi:C_0(X) \to B(H)$ on a Hilbert space $H$, there is a localizable measure $\mu$ and a unitary $U : H\to L^{2}(\mu)$ such that $U \pi(\cdot)U^* : C_0(X) \to B(L^{2}(\mu))$ is given by multiplication operators, i.e.\ there is  a representation $\pi_0 : C_0(X)\to L^{\infty}(\mu)$ such that $U\pi(a)U^*\xi = \pi_0(a)\cdot \xi$.
\end{lem}
\begin{proof}
The algebra $\pi(C_0(X))$ acts on the orthogonal complement of $H_0:=\overline{\pi(C_0(X))H}$ as zero, and we may identify $H_0^{\bot}$ with  $\ell^2(J)$ for some $J$.
The compression of $\pi$ to $H_0$ decomposes into a direct sum of cyclic representations $\{\pi_i\}_{i\in I}$.
By the GNS construction and the Markov--Riesz Theorem each cyclic part $\pi_i:C_0(X)\to B(H_i)$ is equivalent to a representation of $C_0(X)$ given by multiplication on $L^2(\mu_i)$, where $\mu_i$ is a regular (and hence localizable) measure on $X$. 
Taking the direct sum of all the arising measure spaces we get a space $L^2(\mu)$ with localizable $\mu$ and a unitary $U$ with desired properties.
\end{proof}

\subsection{Inverse semigroups and \'etale groupoids}

Let $S$ be a semigroup. 
An element $t\in S$ is called \emph{partially invertible} if there is an element $t^*\in S$ such that $t = t t^* t$ and $t^* = t^* t t^*$. 
Then $t^*t$ and $tt^*$ are idempotents, and we call $t^*$ a \emph{generalized inverse} for $t$. 
The semigroup $S$ is called an \emph{inverse semigroup} if every element in $S$ has a unique generalized inverse. 
Assume this. Then the  map $t \mapsto t^*$ is an anti-multiplicative involution, and any semigroup homomorphism between two inverse semigroups is automatically $*$-preserving. 
The set of idempotents
\[
    E(S) := \{ e \in S : e^2 = e \} = \{ tt^* : t \in S \} = \{ t^*t : t \in S \}
\]
is an abelian semigroup (see for instance \cite[Proposition 2.1.1]{Paterson}). 
A partial order on $S$ is defined by $s \le t$ if and only if $s = t s^* s$, if and only if there is $e \in E(S)$ with $s = t e$, $s,t \in S$,  see \cite{Paterson}. 
If $S$ has a unit $1$ then $e\in E(S)$ if and only if $e \le 1$. 

A \emph{partial bijection} on a set $X$ is a bijection $h : U \to h(U)$ between subsets $U$, $h(U)$ of $X$. 
The set $\PBij(X)$ of all partial bijections on $X$ form an inverse semigroup with composition of $h : U \to h(U)$ and $f : V \to f(V)$ defined as the bijection $h\circ f : f^{-1}(U\cap f(V)) \to h(U\cap f(V))$. 
Then $h^* = h^{-1} : h(U) \to U$,  $h\in E(\PBij(X))$ if and only if $h = \id|_U$ for some $U \subseteq X$, and $h \leq f$ if and only if $f$ is an extension of $h$. 
Every inverse semigroup $S$ can be viewed as a subsemigroup of $\PBij(S)$, where $h_t : S_{t^*} \to S_{t}$ is given by $h_t(s) := ts$ and $S_{t} := \{ s\in S : ss^* \leq t^*t \}$ for all $t\in S$, see \cite[Proposition 2.1.3]{Paterson}.
A \emph{partial homeomorphism} on a topological space $X$ is a homeomorphism $h : U \to h(U)$ between two open subsets $U$ and $h(U)$ of $X$.
Partial homeomorphisms form an inverse subsemigroup of $\PBij(X)$ that we denote by $\PHomeo(X)$.

\begin{defn}\label{defn:inverse_semigroup_action}
An \emph{action of an inverse semigroup $S$ on a topological space $X$} is a non-degenerate semigroup homomorphism $h : S \to \PHomeo(X)$: a family  of partial homeomorphisms $h_t : X_{t^*} \to X_{t}$ such that $h_t \circ h_s = h_{t s}$ for all $s,t\in S$, and  $\bigcup_{t\in S} X_t = X$.
\end{defn}

A groupoid is a small category $\G$ in which every arrow is invertible. 
If $\G$ (identified with the set of arrows) is equipped with a topology that makes  the range, domain, composition, and inverse of arrows  continuous, we call $\G$ a \emph{topological groupoid}. 
We denote by $X \subseteq \G$  the \emph{unit space} of $\G$. 
Let $\G$ be an \emph{\'etale groupoid} with unit space $X$, i.e.\ $\G$ is a topological groupoid where the range and domain maps $r,d : \G \to X \subseteq \G$ are local homeomorphisms. 
In particular, $X$ is an open subset of $\G$. 
We will assume that $X$ is a locally compact Hausdorff space, in which case $\G$ is necessarily locally compact and locally Hausdorff, but $\G$ is Hausdorff if and only if $X$ is closed in $\G$, see for instance \cite[Lemma 2.3.2]{Sims}.
A \emph{bisection} of a groupoid $\G$ is an open set $U \subseteq \G$ such that the restrictions $r|_U$ and $d|_{U}$ are homeomorphisms onto open sets of $\G$. 
So a topological groupoid is \'etale if and only if it can be covered by bisections. 

\begin{ex}[Inverse semigroup actions from \'etale groupoids]\label{ex:actions_from_groupoids}
The set of bisections $\Bis(\G)$ of an \'etale groupoid $\G$ forms a unital inverse semigroup with multiplication and involution given by 
\[
    U \cdot V := \{ \gamma \eta : \text{$\gamma \in U,\ \eta\in V$ are composable} \},  \qquad  
    U^* = U^{-1} := \{ \gamma^{-1} : \gamma \in U \} .
\]
Then $X$ is a unit of $\Bis(\G)$. 
For each $U\in \Bis(\G)$ we have a homeomorphism 
\(
    h_U := r \circ d|_{U}^{-1} : d(U) \to r(U) ,
\)  
and these homeomorphisms define a \emph{canonical action} $h : \Bis(\G) \to \PHomeo(X)$ on the unit space $X$ \cite[Proposition 5.3]{Exel}. 
In fact we may view this action as a restriction of a partial action $\tilde{h} : \Bis(\G) \to \PHomeo(\G)$ on the groupoid $\G$, where $\tilde{h}_{U} : r^{-1}(d(U)) \to r^{-1}(r(U))$ is given by $\tilde{h}_U(\gamma)=d|_{U}^{-1}(r(\gamma))\gamma$. 
In other words, $\tilde{h}_U(\gamma) = \eta\gamma$, where $\eta$ is the unique element in $U$ that can be composed with $\gamma$. 
Using this interpretation one  sees that $\tilde{h}$ is indeed an inverse semigroup action. It restricts to the canonical action $h$ in the sense that  $h= r \circ\tilde{h}|_{X}$.
\end{ex}


\begin{ex}[\'Etale groupoids from inverse semigroup actions] \label{ex:groupoids_from_actions}
For any inverse semigroup action $h : S \to\PHomeo(X)$ the \emph{transformation groupoid} $S \ltimes_{h} X$ is defined as follows, see \cite[p.\ 140]{Paterson} or \cite[2.1]{Kwa-Meyer}. In  \cite[Section~4]{Exel} it is called the groupoid of germs, but we believe this term should be reserved for the different construction that is used for instance in \cite{Re}.
The arrows of $S\ltimes_{h} X$ are equivalence classes of pairs $(t,x)$ for $x\in X_{t^*} \subseteq X$; two pairs $(t,x)$ and $(t',x')$ are equivalent if $x = x'$ and there is $v\in S$ with $v\le t, t'$ and $x \in X_{v^*}$.
The range and domain maps $r,d: S\ltimes_{h} X \rightrightarrows X$ and the multiplication are defined by $r([t,x]) := h_t(x)$, $d([t,x]) := x$, and $[s,h_t(x)] \cdot [t,x] = [s\cdot t,x]$ for $x \in X_{(st)^*}$. 
We give $S\ltimes_{h} X$ the unique topology such that the sets $U_t := \{ [t,x] : x \in X_{t^*} \}$ are open bisections of $S\ltimes_{h} X$. 
Then $S\ltimes_{h} X$ is an \'etale groupoid and the map $S\ni t \mapsto U_t \in \Bis(S\ltimes_{h} X)$ is a semigroup homomorphism. 
Composing it with the canonical homomorphism $\Bis(S\ltimes_{h} X) \to \PHomeo(X)$ we recover the action $h$. 
Composing it with the canonical homomorphism $\Bis(S\ltimes_{h} X) \to \PHomeo(S\ltimes_{h} X)$ we
get an extension $\tilde{h}$ of $h$ where 
$
    \tilde{h}_{t} : \{ [s,x] : x\in h_{s}^{-1}(X_{t^*}) \} \to \{ [s,x] : x \in h_{s}^{-1}(X_{t}) \}$ is given by $\tilde{h}_t([s,x]) := [ts,x].
$
In this way every inverse semigroup action can be viewed as coming from an \'etale groupoid, and then it can be extended to an action on this groupoid.
Conversely, for any \'etale groupoid $\G$ and any inverse subsemigroup $S\subseteq \text{Bis}(\G)$, acting canonically on $X$, we have a continuous open groupoid homomorphism $S\ltimes_{h} X \ni [U,x] \mapsto (d|_U)^{-1}(x) \in \G$. 
This is an isomorphism $S\ltimes_{h} X\cong \G$ if and only if $S$ is \emph{wide} in the sense that $S$ covers $G$ and $U\cap V$ is a union of bisections in $S$ for all $U,V\in S$, see \cite[Proposition~5.4]{Exel}, \cite[Proposition 2.2]{Kwa-Meyer}. In particular, every \'etale groupoid is a transformation groupoid.
\end{ex}

\subsection{$L^p$-partial isometries on Banach spaces and spatial partial isometries} 
\label{subsect:Partial isometries}

Partial isometries acting on arbitrary Banach spaces were defined by Mbekhta~\cite[Definition 4.1]{Mbekhta}, see also \cite[8.1]{cgt}, and in Banach algebras the definition can be phrased as follows.

\begin{defn}\label{defn:partial_isometries}
A \emph{partial isometry in a Banach algebra} $B$ is a partially invertible element $t$ in the semigroup $B_1$ of contractive elements, so $t \in B_1$ and there exists $t^* \in B_1$ such that $t = t t^* t$ and $t^* = t^* t t^*$. 
If $B$ is unital and complex we say that $t \in B$ is an \emph{$MP$-partial isometry} if $t$ is contractive and admits a contractive Moore--Penrose generalized inverse $t^*\in B$, i.e.\ $t^*$ is a generalised inverse of $t$ and the idempotents $t t^*$ and $t^* t$ are hermitian. 
\end{defn}

If $B$ is a $C^*$-algebra (real or complex) then a contractive generalized inverse to a partial isometry $t\in B$ is necessarily the adjoint element $t^*\in B$, see \cite[3.1, 3.3]{Mbekhta} and Remark~\ref{rem:complex_the_real_C*}, so in this case both notions in Definition~\ref{defn:partial_isometries} coincide with the usual notion of partial isometry in a $C^*$-algebra. 
However, for many Banach spaces $E$, including $L^1(\mu)$, $c_0$, $\ell^{\infty}$, a partial isometry $t\in B(E)$ may have many contractive generalized inverses. 
On the other hand, if the Moore--Penrose generalized inverse exists it is necessarily unique. 
So the $MP$-partial isometries have uniquely determined generalized inverse $MP$-partial isometries.  
Partial isometries usually do not form a semigroup (the product of two partial isometries $s,t\in B$ in a $C^*$-algebra $B$ is a partial isometry if and only if the projections $s^*s$ and $tt^*$ commute, see \cite{Halmos}).
For $p\in[1,\infty]\setminus \{2\}$ and any Banach space $E$ we introduce a natural inverse semigroup of partial isometries in $B(E)$. 
Recall, see \cite{BDEGGMM}, \cite{Agniel}, that an \emph{$L^p$-projection on $E$} is an idempotent $P\in B(E)$ such that for all $\xi \in E$ we have 
$$
\|\xi\|^p = \| P\xi \|^p + \| (1-P)\xi \|^p
$$
 when $p<\infty$ and $\|\xi\|=\max\{\|P\xi\|,\|(1-P)\xi\|\}$ when $p=\infty$.  
The $L^2$-projections on a Hilbert space are simply the orthogonal projections. 
If $p\neq 2$, then all $L^p$-projections on $E$ commute and they form a Boolean algebra $\mathbb{P}_{p}(E)\subseteq B(E)$. 
The Banach algebra $C_p(E) := \clsp \mathbb{P}_{p}(E)$  generated by $\mathbb{P}_{p}(E)$ is called the \emph{$p$-Cunningham algebra} of $E$, and we have $C_p(E) \cong C_0(X)$ where $X$ is a totally disconnected space (the Stone dual to $\mathbb{P}_{p}(E)$), see \cite{BDEGGMM} for details.
We extend the notion of $L^p$-projection to partial isometries as follows:

\begin{defn} 
An \emph{$L^p$-partial isometry} on a Banach space $E$, $p\in [1,\infty]$, is a contraction $T\in B(E)$ that has a contractive generalized inverse $T^*\in B(E)$ such that $T T^*$ and $T^* T$ are $L^p$-projections. 
We denote by $L^p\mhyphen\PIso(E)$ the set of all $L^p$-partial isometries on $E$.
\end{defn}

\begin{prop}
For $p\in [1,\infty]\setminus \{2\}$ the set $L^p\mhyphen\PIso(E)$ is an inverse semigroup whose idempotents are $\mathbb{P}_{p}(E)$.
Passing to duals yields an injective antimultiplicative map  $L^{p}\mhyphen\PIso(E)\ni P \mapsto P'\in L^{q}\mhyphen\PIso(E')$ where $1/p+1/q=1$. 
\end{prop}
\begin{proof} 
Let $S, T\in L^p\mhyphen\PIso(E)$ and choose the corresponding generalized inverses  $S^*, T^* \in L^p\mhyphen\PIso(E)$, so that the corresponding idempotents $TT^*$, $T^*T$, $SS^*$, $SS^*$ are in  $\mathbb{P}_{p}(E)$. 
By commutativity of these idempotents, $S^*T^*$ is a generalized inverse of $TS$. Assume $p<\infty$.
Using the $L^p$-projection property of $SS^*$, $T^*T$, $TT^*$ and that $T^*$ is an isometry on the range of $T$ (and vice versa), for any $x\in E$, we get
\[
\begin{split}
    \|TSS^*T^*x\|^p &= \| T^*TSS^*T^*x \|^p = \|SS^*(T^*T)T^* x \|^p = \|SS^*T^*x\|^p \\
        &= \|T^*x\|^p - \|(1-SS^*)T^*x\|^p = \|TT^*x\|^p-\|T^*x - SS^* (T^*T)T^*x\|^p \\
        &= \|x\|^p - \|(1-TT^*)x\|^p - \|TT^*x - T SS^* T^*x\|^p \\
        &= \|x\|^p - \|(1-TSS^*T^*)x\|^p.
\end{split}
\]
Hence $TSS^*T^*\in \mathbb{P}_{p}(E)$. 
By symmetry we also get $S^*T^*TS\in \mathbb{P}_{p}(E)$. 
Hence $TS\in L^p\textup{-PIso}(E)$, and so $L^p\mhyphen\PIso(E)$ is a semigroup. 
Assume in addition that $T$ is the generalized inverse of $S$. 
Using that $S^*S$ is an $L^p$-projection and $STS=S$ we get
\[
    \| TSx \|^p = \| S^*S TSx \|^p + \| (1-S^*S)TSx \|^p = \|TSx\|^p + \|TSx-S^*S x\|^p .
\]
Hence $\|TSx-S^*S x\|^p=0$ and so $TS=S^*S$. 
Similarly one shows that $ST = SS^*$. 
This implies that $T=S^*$, (cf. \cite[Propositon 2.4]{kwa-leb}), and proves the uniqueness of the generalized inverse in $L^p\mhyphen\PIso(E)$. Thus $L^p\mhyphen\PIso(E)$ is an inverse semigroup. For $p=\infty$ one readily adapts the above proof to see that $L^{\infty}\mhyphen\PIso(E)$ is a semigroup. 
As we noted above every element in $L^{\infty}\mhyphen\PIso(E)$ has a generalized inverse, and it is unique as we claim that we may  embed $L^{\infty}\mhyphen\PIso(E)$ into the  inverse semigroup opposite to 
$L^{1}\mhyphen\PIso(E')$. Indeed, by \cite[Lemma  1.4]{BDEGGMM} an idempotent $P\in B(E)$ is $L^p$-projections if and only if 
$P'\in B(E')$ is an $L^q$-projection, where $1/p+1/q=1$ and $p,q\in [1,\infty]$.  This implies  that  $L^{p}\mhyphen\PIso(E)\ni P \mapsto P'\in L^{q}\mhyphen\PIso(E')$  is a well-defined injective anithomorphism of semigroups, where $1/p+1/q=1$, $p,q\in [1,\infty]\setminus 2$.
\end{proof}

If $E=L^p(\mu)$ for a localizable measure $\mu$ and $p \in [1,\infty]\setminus \{2\}$, then $P\in B(L^p(\mu))$ is an $L^p$-projection if and only if $P$ is an operator of multiplication by a characteristic function, see \cite[Proposition 4.9]{BDEGGMM} (when $\F=\C$ this follows  from Proposition~\ref{prop:hermitian_operators} applied to the identity representation of the $p$-Cunningham algebra on $L^p(\mu)$).
In particular, we have an isomorphism of Boolean algebras $\mathbb{P}_{p}(L^p(\mu)) \cong [\Sigma_{\mu}]$. 
We now discuss the relationship between $L^p\mhyphen\PIso(L^p(\mu))$ and spatial partial isometries, originally introduced (for $\sigma$-finite measures)   by Phillips ~\cite{PhLp1}, \cite{Phillips}.

 
\begin{defn}
A \emph{subspace} of a measure space $(\Omega,\Sigma_{\mu}, \mu)$ is $(D,\Sigma_D, \mu|_D)$ where $D\in \Sigma$, $\Sigma_{D} := \{ C\cap D : C\in \Sigma\}$ and $\mu|_D := \mu|_{\Sigma_D}$. 
A \emph{partial set automorphism} of $(\Omega,\Sigma, \mu)$ is a set isomorphism between two subspaces of $(\Omega,\Sigma, \mu)$. 
So a partial set automorphism is a  map $\Phi : \Sigma_{D_{\Phi^*}} \to \Sigma_{D_{\Phi}}$ ($D_{\Phi},D_{\Phi^*}\in \Sigma$) that descends to a Boolean isomorphism  $[\Phi] : [\Sigma_{D_{\Phi^*}}] \to [\Sigma_{D_{\Phi}}]$. 
We define $\PAut([\Sigma_{\mu}]) := \{ [\Phi] : \text{$\Phi$ is a partial set automorphism} \}$ for the set of partial automorphisms of the Boolean algebra $[\Sigma_{\mu}]$. 
\end{defn}

By definition $\PAut([\Sigma_{\mu}])$ is the set of isomorphisms between ideals in the Boolean algebra $[\Sigma_{\mu}]$. It forms an inverse semigroup that can be identified with an inverse semigroup of operators on $L_0(\mu)$.
Indeed, a partial automorphism $\Phi$ defines the composition operator isomorphism $T_{\Phi} : L_0(\mu|_{D_{\Phi^*}}) \to  L_0(\mu|_{D_{\Phi}})$ whose inverse is given by $T_{\Phi^*}$. 
Identifying $L_0(\mu|_{D_{\Phi^*}})$ and $L_0(\mu|_{D_{\Phi}})$ with subspaces of $L_0(\mu)$ we see that $\Phi$ determines uniquely a linear operator $T_{\Phi} : L_0(\mu)\to L_0(\mu)$ such that  $T_{\Phi}$ preserves (monotone) limits and
\[
    T_{\Phi} 1_{C} = 1_{\Phi(C\cap D_{\Phi^*})}, \qquad C \in \Sigma.
\]
We still call $T_{\Phi} : L_0(\mu) \to L_0(\mu)$ a (generalized) \emph{composition operator}.
For any two partial automorphisms $\Phi : \Sigma_{D_{\Phi^*}}\to \Sigma_{D_{\Phi}}$ and 
$\Psi : \Sigma_{D_{\Psi^*}} \to \Sigma_{D_{\Psi}}$ we have $T_{\Phi}\circ T_{\Psi} = T_{\Phi\circ \Psi}$, 
where $\Phi\circ \Psi : \Sigma_{\Psi^*(D_{\Phi^*} \cap D_{\Psi})}\to \Sigma_{\Phi(D_{\Phi^*} \cap D_{\Psi})}$ is a well defined partial automorphism of $(\Omega,\Sigma, \mu)$.
Also $T_{\Phi^*}$ is the unique generalized inverse for $T_{\Phi}$ among the composition operators.
In particular, composition operators on $L_0(\mu)$ form an inverse semigroup naturally isomorphic to $\PAut([\Sigma_{\mu}])$.
By the very definition of the partial automorphism $\Phi : \Sigma_{D_{\Phi^*}} \to \Sigma_{D_{\Phi}}$, the map $\mu\circ \Phi:\Sigma_{D_{\Phi^*}}\to [0,\infty]$ is a measure on $D_{\Phi^*}$ equivalent to $\mu|_{D_{\Phi^*}}$. 
Thus we may consider the Radon--Nikodym derivative $\frac{d\mu\circ\Phi}{d\mu|_{D_{\Phi^*}}}$ as a function on $\Omega$ which is zero outside $D_{\Phi^*}$. 
Similar remarks apply to the inverse partial automorphism $\Phi^*:\Sigma_{D_{\Phi}}\to \Sigma_{D_{\Phi^*}}$.

\begin{prop}\label{prop:spatial^partial_isometries}
Let $p\in [1,\infty]$ and $\mu$ be localizable. 
For any $[\Phi]\in \PAut([\Sigma_{\mu}])$ the weighted composition operator $U_{\Phi}   := \left( \frac{d\mu\circ\Phi^{*}}{d\mu|_{D_{\Phi}}} \right)^{\frac{1}{p}}T_{\Phi}$ is a well defined partial isometry on $L^p(\mu)$, and the map $\PAut([\Sigma_{\mu}])\ni [\Phi] \mapsto U_{\Phi}\in B(L^p(\mu))$ is a semigroup embedding. 
The semigroup generated by $U_{\Phi}\in B(L^p(\mu))$, $[\Phi] \in \PAut([\Sigma_{\mu}])$, and  $UL^{\infty}(\mu)\subseteq B(L^p(\mu))$ is an inverse semigroup of partial isometries of the form
\begin{equation}\label{eq:spatial^partial_isometry}
    \omega U_{\Phi}  := \omega \left( \frac{d\mu\circ\Phi^{*}}{d\mu|_{D_{\Phi}}} \right)^{\frac{1}{p}}T_{\Phi}
\end{equation}
where $\Phi : \Sigma_{D_{\Phi^*}} \to \Sigma_{D_{\Phi}}$ is a partial set automorphism and $\omega : D_{\Phi}\to \{ z\in \F: |z|=1 \}$ is measurable. 
The operator $\omega U_{\Phi}$ determines $\Phi$ and $\omega$ uniquely up to sets of measure zero. 
Moreover, 
\begin{equation}\label{eq:spatial^partial_isometries_relations}
    (\omega U_{\Phi})^* = T_{\Phi^*}(\overline{\omega}) U_{\Phi^*}, \qquad (\omega U_{\Phi})\circ  (\upsilon U_{\Psi})=\omega T_{\Phi}(\upsilon) U_{\Phi\circ\Psi} .
\end{equation}
\end{prop}
\begin{proof}
This follows from \cite[Lemma 6.12, 6.17]{PhLp1} where $\sigma$-finite measures were considered, but the arguments work for localizable measures.
\end{proof}

\begin{defn}\label{def:spatial^partial_isos}
We denote by $\SPIso(L^p(\mu))$ the inverse subsemigroup of operators \eqref{eq:spatial^partial_isometry} and call them \emph{spatial partial isometries} on $L^p(\mu)$ \cite[Definition 6.4]{PhLp1},
\cite[Definition 6.2]{Gardella}. 
\end{defn}

\begin{rem}
The semilattice of idempotents in $\SPIso(L^p(\mu))$ is isomorphic to $[\Sigma_{\mu}]$, as it consists of multiplication operators of idempotent elements in $L^{\infty}(\mu)$.
The group of invertible elements in $\SPIso(L^p(\mu))$ is isomorphic to $UL^{\infty}(\mu) \rtimes \Aut([\Sigma_\mu])$, see Proposition~\ref{prop:group_of_spatial_isometries}. 
\end{rem}

\begin{thm}[Banach--Lamperti theorem for partial isometries] \label{thm:spatial^partial_isometries_description}
Let $p\in [1,\infty]\setminus \{2\}$. 
A contraction $T \in  B(L^p(\mu))$ is a spatial partial isometry if and only if it has a contractive generalized inverse $T^*\in B(L^p(\mu))$ such that both $TT^*$ and $T^*T$ are multiplication operators on $L^p(\mu)$.
Thus spatial partial isometries coincide with $L^p$-partial isometries on $L^p(\mu)$: 
\[
    \SPIso(L^p(\mu))= L^p\mhyphen\PIso(L^p(\mu)) , 
\]
and in the complex case also with $MP$-partial isometries in $B(L^p(\mu))$. 
\end{thm}
\begin{proof} 
Using \eqref{eq:spatial^partial_isometries_relations} we see that 
\[
    (\omega U_{\Phi})^* \omega U_{\Phi} = T_{\Phi^*}(\overline{\omega}) U_{\Phi^*} \circ \omega U_{\Phi}= U_{\Phi^*\circ\Phi}=U_{\id_{\Sigma_{D_{\Phi^*}}}} = 1_{D_{\Phi^*}}
\]
is the operator of multiplication by the characteristic function of the domain $D_{\Phi^*}$ of $\Phi$.  
Conversely, if $T\in  B(L^p(\mu))$ is a partial isometry with a  contractive generalized inverse $T^*\in B(L^p(\mu))$ such that both $T^*T$  and $TT^*$ are multiplication operators on $L^p(\mu)$, say by functions $1_{D^{*}}$ and $1_{D}$ respectively, then $T$ restricts to an invertible isometry  $T : L^p(\mu|_{D^*}) \to L^p(\mu|_{D})$. 
Hence $T$ is of the form \eqref{eq:spatial^partial_isometry} by Proposition~\ref{prop:Banach_Lamperti}. 
In particular, $\SPIso(L^p(\mu)) = L^p\mhyphen\PIso(L^p(\mu))$. 
If $\F=\C$ and $T\in B(L^p(\mu))$ is an $MP$-partial isometry and $T^*$ is its Moore--Penrose generalized inverse, then the idempotents $T^*T$ and $T T^*$ are  operators of multiplication by characteristic functions of some measurable sets, by Proposition~\ref{prop:hermitian_operators}.
\end{proof}

\section{Twisted inverse semigroup actions and their crossed products}
\label{sec:Inverse_semigroup_crossed_products}

We introduce Banach algebra crossed products by twisted inverse semigroup actions, that generalize $C^*$-algebraic constructions introduced by Buss and Exel in~\cite{Buss_Exel}.

%
\subsection{The crossed product}

By an \emph{ideal} in a Banach algebra we  always mean a closed two-sided ideal. 
A \emph{partial automorphism} on a Banach algebra $A$ is an isometric isomorphism $\alpha : I \to J$ between two ideals $I,J$ of $A$. 
The partial automorphisms on $A$ form an inverse subsemigroup of $\PBij(A)$ that we denote by $\PAut(A)$. 

\begin{defn} \label{defn:inverse_semigroup_action_on_Banach_algebra} 
An \emph{action of a unital inverse semigroup $S$ on a Banach algebra $A$} is a unital semigroup homomorphism $\alpha : S \to \PAut(A)$.
Thus an action of $S$ on $A$ consists of (closed, two-sided) ideals $I_t$ of $A$ and isometric isomorphisms $\alpha_t : I_{t^*} \to I_{t}$ for all $t\in S$, such that $\alpha_1 = \id_A$ and $\alpha_s \circ \alpha_t = \alpha_{s t}$ for all $s,t\in S$. 
\end{defn}

\begin{rem} 
If $A$ is a $C^*$-algebra then all ideals $I_t$ are closed under involution, and all isometric isomorphisms $\alpha_t: I_{t^*}\to I_{t}$ are necessarily $*$-preserving. 
Hence Definition~\ref{defn:inverse_semigroup_action_on_Banach_algebra} is consistent with the well established $C^*$-algebraic version \cite{Sieben}.
\end{rem}

\begin{ex} \label{ex:commutative_example}
If $A=C_0(X)$ for a locally compact Hausdorff space $X$ then $\PAut(A) \cong \PHomeo(X)$. 
Indeed, every ideal $I$ in $A$ is of the form $I = C_0(U)$ for an open set $U \subseteq X$ (we treat elements in $C_0(U)$ as functions in $C_0(X)$ that vanish outside $U$), and any partial automorphism $\alpha : I \to J$ is given by the composition $\alpha(a) = a \circ h^{-1}$ with a partial homeomorphism $h : U \to V$, where $I = C_0(U)$ and $J = C_0(V)$.
Thus inverse semigroup actions on the Banach algebra $C_0(X)$ are equivalent to inverse semigroup actions on the topological space $X$.
\end{ex}

Twisted inverse semigroup actions on $C^*$-algebras were first introduced by Sieben~\cite{Sieben98}, but a more general definition given by Buss and Exel in \cite{Buss_Exel} is needed for a number of problems.
To generalize this to Banach algebras it seems reasonable to assume that all ideals involved in an action have \emph{(contractive, two-sided) approximate units}. 
This in particular guarantees that the corresponding multiplier algebras have good properties and various approaches to multipliers coincide. 
We refer to \cite{Dales}, see also \cite{Daws} and references therein, for more details. 
Following \cite[Section 2]{Daws}, for a Banach algebra $A$ we define the \emph{multiplier algebra} as 
\[
    \Mult(A) := \big\{ (L,R) : \text{$L , R : A \to A$ satisfy $a L(b) = R(a) b$ for all $a,b\in A$} \big\} .
\]
Then $\Mult(A)$ is automatically a unital closed subalgebra of the $\ell^\infty$-direct sum $B(A)\oplus B(A)^{\op}$, see  [24, Proposition 2.5].
In particular, the maps in the pair $(L,R)\in\Mult(A)$ are automatically bounded linear operators, and   $\Mult(A)$ is a Banach algebra with the norm $\|(L,R)\| = \max\{\|L\|,\|R\|\}$. 
Also $\Mult(A)$ is complete in the strict topology, which is defined by the seminorms $(L,R) \mapsto \|L(a)\| + \|R(a)\|$,  $a\in A$. 
Every $a\in A$ yields a multiplier $(L_a,R_a)$ where $L_a(b) := ab$, $R_a(b) := ba$, and if $A$ has a contractive approximate unit the resulting  map is isometric and allows us to treat $A$ as an ideal of $\Mult(A)$. 
Then any isometric automorphism $\alpha:A\to A$ extends uniquely to an isometric automorphism $\overline{\alpha}:\Mult(A)\to \Mult(A)$. 
We denote by 
\[
    \UMult(A) := \{u\in \Mult(A): \text{$u$ is invertible and $\|u\| = \| u^{-1} \| =1$} \} ,
\] 
the group of invertible isometries in $\Mult(A)$.
\begin{defn}[{cf. \cite[Definition 4.1]{Buss_Exel}}] \label{defn:twisted actions}  
A \emph{twisted action} of an inverse semigroup $S$ on a Banach algebra $A$ is a pair $(\alpha,u)$, where $\alpha = \{\alpha_t\}_{t\in S}$ is a family of partial automorphisms $\alpha_t : I_{t^*} \to I_{t}$ of $A$, such that each ideal $I_{t}$ contains an approximate unit and their union is linearly dense in $A$, 
and $u=\{u(s, t)\}_{s,t\in S}$ is a family of unitary multipliers $u(s, t)\in \UMult(I_{s t})$, $t,s \in S$, 
such that the following holds for all $r,s,t\in S$ and $e,f\in E(S)$:
\begin{enumerate}[label={(A\arabic*)}]
    \item\label{enu:twisted actions1} $\alpha_s\circ \alpha_t = \Ad_{u(s,t)}\alpha_{s t}$;
    \item\label{enu:twisted actions2} $\alpha_r \big( a u(s,t) \big) u(r,s t) = \alpha_r(a) u(r,s) u(rs,t)$ for $a\in I_{r^*} \cap I_{s t}$; 
    \item\label{enu:twisted actions3} $u(e,f) = 1_{e f}$ and $u(t,t^*t) = u(t t^*,t) = 1_t$, where $1_t$ is the unit of $\Mult(I_t)$; 
    \item\label{enu:twisted actions4} $u(t^*,e) u(t^*e,t) a = u(t^*,t) a$ for all \(a\in I_{t^* e t}\).
  \end{enumerate}
\end{defn}

\begin{rem}\label{rem:twisted_relations} 
All properties in \cite[Lemma~4.6]{Buss_Exel} hold for the maps and ideals involved in a twisted action on a Banach algebra as defined above. 
In particular, for all $s,t\in S$ and $e\in E(S)$ we have
\[
    I_t = I_{tt^*}, \quad \overline{\alpha}_t \big( u(t^*,t) \big) = u(t,t^*), \quad \alpha_e = \id|_{I_e}, \quad \alpha_s (I_{s^*}\cap I_t) = I_{st} ,
\] 
and $I_{s}\subseteq I_{t}$ whenever $s\leq t$. 
We will use these relations, often without warning.
\end{rem}

The authors of \cite{Sieben}, \cite{Sieben98}, \cite{Buss_Exel} consider covariant representations on Hilbert spaces. 
We propose a more general definition of representations in Banach algebras. 
Our definition works best when the target algebra has a predual, which we explain in detail in Subsection~\ref{ssec:CovariantRepsDualBanachAlgebras} below.
Therefore in general we will use the \emph{double dual} $B''$ of a Banach algebra $B$ (which has a predual $B'$ by construction). 
Recall that $B''$ is again naturally a Banach algebra with either of the Arens products, in which $B$ sits as a $B'$-weakly dense Banach algebra. 
We will consider $B''$ equipped with the \emph{second Arens product}, which for $a,b \in B''$ is determined by the formula $a \cdot b := B'\mhyphen\lim_{\beta} B'\mhyphen\lim_{\alpha} a_\alpha b_{\beta}$
where $\{a_\alpha\}_{\alpha}$, $\{b_\beta\}_{\beta}\subseteq B$ are nets that are convergent to $a$ and $b$ respectively, in the weak$^*$ topology on $B''$ induced by $B'$. 
The first Arens product is given by $a \square b :=  B'\mhyphen\lim_{\alpha} B'\mhyphen\lim_{\beta}  a_\alpha b_{\beta}$, and $B$ is called \emph{Arens regular} if the two products coincide, 
see \cite{Dales} for more details. 
Our decision to use the second Arens product is arbitrary, and in any case we will be primarily interested in products $a v\in B$ where $a\in B$ and $v\in B''$, 
in which case the first and second Arens  products always agree. 
Also recall that every $C^*$-algebra (complex or real) is Arens regular. 

For any representation $\pi : A \to B$ the double adjoint $\pi'' : A'' \to B''$ is a  $B'$-weakly continuous representation that extends $\pi$, and $\pi''$ is isometric whenever $\pi$ is.
In particular, for every Banach subalgebra $A\subseteq B$ we may identify $A''$ with a subalgebra of $B''$.
Also any contractive  approximate unit  $\{\mu_{i}\}$ in $A$ converges $A'$-weakly  to a left identity $1_A$ in $A''$ (and a right identity for the first Arens product in $A''$) and the map $(L,R)\mapsto L''(1_A)$ is an isometric embedding allowing us to assume that $\Mult(A) \subseteq \{ b \in A'' : \text{$ba, ab \in A$ for all $a\in A$} \} \subseteq A''$, see \cite[Proposition 2.9.16 and Theorem 2.9.49]{Dales}.
In particular, for each ideal $I$ in $A$, which has an approximate unit, we will adopt the identifications $I\subseteq \Mult(I)\subseteq I''\subseteq A''$. 
  
\begin{defn}\label{defn:covariant_representation_in_algebra}
Let $(\alpha,u)$ be a twisted action of $S$  on $A$. 
A \emph{covariant representation  of $(\alpha,u)$ in a Banach algebra $B$} is a pair $(\pi,v)$, where $\pi : A \to B$ is a representation and $v : S \to (B'')_{1}$ is a map into contractive elements in $B''$,  such that: 
\begin{enumerate}[label={(CR\arabic*)}]
    \item\label{item:covariant_representation1} $v_t \pi(a) = \pi(\alpha_{t}(a)) v_t \in B$ for all $a\in I_{t^*}$;
    \item\label{item:covariant_representation2} $\pi(a)v_s v_t =\pi(au(s,t)) v_{st}$ for all $a\in I_{st}$, $s,t\in S$;
    \item\label{item:covariant_representation3} $\pi(a)v_e = \pi(a)$ for all $a\in I_{e}$, $e\in E(S)$.
\end{enumerate} 
We call $B(\pi,v) := \clsp\{ \pi(a_t)v_t : a_t\in I_{t}, \, t\in S \}$ the \emph{range} of $(\pi,v)$.
We say that $(\pi,v)$ is injective, isometric, etc.\ if $\pi$ has that property. 
We also say that $(\pi,v)$ is \emph{$B'$-normalized} if 
\begin{enumerate}[label={(CR\arabic*)}] \setcounter{enumi}{3}
    \item\label{item:covariant_representation4} $v_{t}= B'\mhyphen\lim_{i} (\pi(\mu_{i}^t)v_{t})$, for all $t\in S$, where $\{\mu^t_i\}_{i}$ is an approximate unit in $I_t=I_{tt^*}$.
\end{enumerate}
\end{defn}

\begin{rem}\label{rem:general_covariant_rep}
Condition~\ref{item:covariant_representation1} is equivalent to $\pi(a)v_t=v_t\pi(\alpha_{t}^{-1}(a))\in B$ for all $a\in I_{t}$, $t\in S$. 
It implies that each $v_t$ can be viewed as a ``partial mutliplier'' of $B$, and  $B(\pi,v)\subseteq B$.
Conditions~\ref{item:covariant_representation1} and \ref{item:covariant_representation3} imply that $v_{e}\pi(a)=\pi(a)$ for all $a\in I_{e}$, $e\in E(S)$.
Employing also \ref{item:covariant_representation2} we get 
\begin{equation}\label{eq:CovarianceConditionAlternate}
    v_t\pi(a)v_{t^*} = \pi \big( \alpha_t(a) u(t,t^*) \big) 
\end{equation}
for all $a\in I_{t^*}$, $t\in S$. 
Also, \ref{item:covariant_representation3} and \ref{item:covariant_representation4} imply  that $v_{e}=\pi''(1_e)$ for every $e\in E(S)$. 
\end{rem} 

Notice that the elements $v_t$ of Definition~\ref{defn:covariant_representation_in_algebra} are not required to be partial isometries. The definition of the generalised inverse $v_t^*$ of $v_t$ in the proof below involves both $v_{t^*}$ and the twist $u$, so the covariance conditions \eqref{eq:CovarianceConditionAlternate} above and \ref{item:normalized_covariant_representation1} below look different. 

\begin{prop}\label{prop:normalized_twisted_rep}
A pair $(\pi,v)$ is a $B'$-normalized covariant representation of $(\alpha,u)$ if and only if $\pi : A \to B$ is a representation and $v : S\to (B'')_{1}$ takes values in partial isometries such that every $v_t$, $t\in S$, admits a contractive generalised inverse $v_t^*$, satisfying $\pi(a) v_t \in B$, for $a\in I_t$, $t\in S$, and
\begin{enumerate}
    \item\label{item:normalized_covariant_representation1} $v_t \pi(a) v_t^* = \pi(\alpha_{t}(a))$ for all $a\in I_{t^*}$, $t\in S$;
    \item\label{item:normalized_covariant_representation2} $v_s v_t = \pi''(u(s,t))v_{st}$ for all  $s,t\in S$;
    \item\label{item:normalized_covariant_representation3} $v_t v_t^* = \pi''(1_{t})$ and $v_t^* v_t  = \pi''(1_{t^*})$ for all $t\in S$.
\end{enumerate}
Moreover, for any covariant representation $(\pi,v)$ of $(\alpha,u)$ in $B$ there is a unique $B'$-normalized covariant representation $(\pi,\tilde{v})$ such that $\pi(a)v_t=\pi(a)\tilde{v}_{t}$ for all $a\in I_{t}$ and $t\in S$, namely $\tilde{v}_{t} = B'\mhyphen\lim_{i} (\pi(\mu_{i}^t)v_{t}) = v_t \pi''(1_{t^*})$, $t\in S$. 
\end{prop}
\begin{proof} 
Assume $(\pi,v)$ is a pair with the properties described in the assertion. 
For every $t\in S$ and $a\in I_{t^*}$, using \ref{item:normalized_covariant_representation3} and \ref{item:normalized_covariant_representation1} we have 
\[
    v_t \pi(a) = v_t \pi(a) \pi''(1_{t^*}) = v_t \pi(a) v_{t}^*v_t = \pi \big( \alpha_t(a) \big) v_t  ,
\] 
which is \ref{item:covariant_representation1}. 
Condition~\ref{item:normalized_covariant_representation2} immediately gives \ref{item:covariant_representation2}.
By \ref{item:normalized_covariant_representation3} we have $v_t = v_t v_t^*v_t = \pi''(1_{t}) v_t = v_t\pi''(1_{t^*})$. 
This implies \ref{item:covariant_representation4} and that for $e\in E(S)$ we get $v_e v_e= \pi''(u(e,e))v_e=\pi''(1_e)v_e=v_e$. 
Therefore $\pi''(1_e)= v_e^*v_e=v_e^*v_e v_e=\pi''(1_e) v_e=v_e$, 
and using this one gets \ref{item:covariant_representation3}.
Hence $(\pi,v)$ is a $B'$-normalized covariant representation of $(\alpha,u)$.

For the converse let $(\pi,v)$ be any $B'$-normalized covariant representation. 
For each $t\in S$ choose an approximate unit $\{\mu^t_i\}_{i}$ in $I_t = I_{tt^*}$. 
Using \ref{item:covariant_representation1} and $B'$-continuity of multiplication in $B''$ in the second variable we get
\[
    B' \mhyphen\lim_{i} \big( \pi(\mu^t_i) v_t \big) = B'\mhyphen\lim_{i} \big( v_t \pi(\alpha_{t}^{-1}(\mu^t_i)) \big) = v_t B'\mhyphen\lim_{i} \pi \big( \alpha_{t}^{-1}(\mu^t_i) \big) = v_{t}\pi''(1_{t^*}),
\]
because $\{\alpha_{t}^{-1}(\mu^t_i)\}_{i}$ is an approximate unit in $I_{t^*}$ which is $B'$-convergent to $1_{t^*}$. 
Thus $v_t = B'\mhyphen\lim_{i} (\pi(\mu^t_i) v_t) = v_{t} \pi''(1_{t^*})$. 
Also \ref{item:covariant_representation3} and \ref{item:covariant_representation4} imply that $v_e = \pi''(1_e)$ for all $e\in E(S)$. 
We put $v_t^*:=v_{t^*}\pi''(u(t,t^*)^{-1})$. 
Using \ref{item:covariant_representation1}, \ref{enu:twisted actions1} and \ref{item:covariant_representation2}, and being careful 
to use only $B'$-continuity of multiplication in $B''$ in the second variable, we get
\[
\begin{split}
    v_t v_t^* &= B'\mhyphen\lim_{i} v_t v_{t^*}\pi \big( \mu^{t}_i u(t,t^*)^{-1} \big)= B'\mhyphen\lim_{i} \pi \big( \alpha_t \circ \alpha_{t^*} \big( \mu^t_i u(t , t^*)^{-1} \big) \big) v_t v_{t^*} \\
                &=  B'\mhyphen\lim_{i} \pi \big( u(t,t^*)  \mu^t_i u(t , t^*)^{-1} u(t , t^*)^{-1} \big) v_t v_{t^*} \\
        &= B'\mhyphen\lim_{i} \pi \big( u(t,t^*) \mu^t_i u(t , t^*)^{-1} u(t , t^*)^{-1} u(t , t^*) \big) v_{t t^*} \\
        &= B'\mhyphen\lim_{i} \pi''\big( u(t , t^*) \mu^t_i u(t , t^*)^{-1} \big) v_{t t^*}= B'\mhyphen\lim_{i} v_{t t^*} \pi''\big( u(t , t^*) \mu^t_i u(t , t^*)^{-1} \big)\\
				&=v_{t t^*}\pi''(1_{t}) = \pi''(1_{t}).
\end{split}
\]
By Remark~\ref{rem:twisted_relations} we have $\overline{\alpha}_{t^*}(u(t,t^*)^{-1}) = u(t^*,t)^{-1}$, so a similar calculation gives
\[
\begin{split}
    v_t^* v_t &= B'\mhyphen\lim_{i} v_{t^*} \pi \big(  u(t , t^*)^{-1} \mu^{t}_i\big) v_t 
        = B'\mhyphen\lim_{i} \pi \big(  u(t^* , t)^{-1} \alpha_{t^*}(\mu^t_i) \big) v_{t^*} v_t \\
        &= B'\mhyphen\lim_{i} \pi \big( u(t^* , t)^{-1} \alpha_{t^*}(\mu^t_i) u(t^* , t) \big) v_{t^* t} 
				= B'\mhyphen\lim_{i} v_{t^* t}  \pi \big(u(t^* , t)^{-1}  \alpha_{t^*}(\mu^t_i) u(t^* , t) \big) \\
        &=v_{t^* t}\pi''(1_{t^*})=  \pi''(1_{t^*}) .
\end{split}
\]
This finishes the proof of \ref{item:normalized_covariant_representation3} and shows that $v_t$ and $v_t^*$ are mutual generalized inverses. 
For $a \in I_{t^*}$ we get 
\[
\begin{split}
    v_t \pi(a) v_t^* &= v_t \pi(a) v_{t^*} \pi'' \big( u(t,t^*)^{-1} \big) = \pi \big( \alpha_t(a) \big) v_{t}v_{t^*} \pi'' \big( u(t,t^*)^{-1} \big) \\
        &= \pi \big( \alpha_t(a) \big) \pi'' \big( u(t,t^*) \big) v_{tt^*} \pi''(u(t,t^*)^{-1}) \\
        &= \pi \big( \alpha_t(a) \big) \pi'' \big( u(t,t^*) u(t,t^*)^{-1} \big) = \pi \big( \alpha_t(a) \big) ,
\end{split}
\]
which is \ref{item:normalized_covariant_representation1}.
Furthermore, $\alpha_s''(1_{s^*}1_{t})=1_{st}$ because $\alpha_s(I_{s^*}\cap I_t)=I_{st}$ and therefore
\[
    v_s v_t = v_s \pi''(1_{s^*}) \pi''(1_{t}) v_t = \pi''(1_{st}) v_s v_t = \pi''\big( u(s,t) \big)v_{st},
\]
which proves \ref{item:normalized_covariant_representation2}. 
Hence $(\pi,v)$ satisfies all properties in the assertion.

Now let $(\pi,v)$ be any covariant representation of $(\alpha,u)$ in $B$. 
As above we see that the limit $\tilde{v}_{t} := B'\mhyphen\lim_{i} ( \pi ( \mu_{i}^t ) v_{t} ) = v_t \pi''( 1_{t^*} )$ exists and does not depend on the choice of an approximate unit $\{\mu^t_i\}_{i}$ in $I_t$, $t\in S$. 
Clearly if $(\pi,\tilde{v})$ is a $B'$-normalized covariant representation satisfying $\pi(a) v_t = \pi(a) \tilde{v}_{t}$ for all $a\in I_{t}$ and $t\in S$, then each $\tilde{v}_{t}$ must be given by such a limit, giving \ref{item:covariant_representation4}.
For $a\in I_{t^*}$ we get  
\[
    \tilde{v}_t \pi(a) = v_t \pi''(1_{t^*}) \pi(a) = v_t\pi(a) = \pi \big( \alpha_{t}(a) \big) v_t= \lim_{i} \pi \big( \alpha_{t}(a) \big) \pi(\mu_i^{t}) v_{t} = \pi \big( \alpha_{t}(a) \big) \tilde{v}_t.
\]
This shows that $(\pi,\tilde{v})$ satisfies \ref{item:covariant_representation1} and  $\pi(a)v_t=\pi(a)\tilde{v}_{t}$ for all $a\in I_{t}$ and $t\in S$.
Using this we get $\pi(a)\tilde{v}_e =\pi(a)v_e = \pi(a)$ for all $a\in I_{e}$, $e\in E(S)$, so \ref{item:covariant_representation3} holds. 
Finally, for all $s,t\in S$ and $a\in I_{st}$ the calculation 
\[
    \pi(a) \tilde{v}_s \tilde{v}_t = v_s \pi \big( \alpha_{s^*}(a) \big) \tilde{v}_t = \pi(a)v_s v_t = \pi \big( a u(s,t) \big) v_{st} = \pi \big( a u(s,t) \big) \tilde{v}_{st} 
\]
gives \ref{item:covariant_representation2}. 
Thus $(\pi,\tilde{v})$ is a $B'$-normalized covariant representation of $(\alpha,u)$.
\end{proof}

\begin{rem}\label{rem:unital_actions}
 If  each $I_t$, $t\in S$, is unital then a covariant representation $(\pi,v)$ is $B'$-normalized if and only if $v_t = \pi(1_t)v_t \in B$ for each $t\in S$, and Proposition~\ref{prop:normalized_twisted_rep} 
could be formulated without the use of the bidual algebra $B''$ and the extended representation $\pi''$. 
\end{rem}

\begin{lem}\label{lem:range_of_covariant_rep}
Let $(\pi,v)$ be a covariant representation of $\alpha$ in a Banach algebra $B$.
The range $B(\pi,v)$ is a Banach subalgebra of $B$. 
The spaces  $A_t = \{ \pi(a_t)v_t: a_t\in I_{t} \}$, $t\in S$, form a grading of $B(\pi,v)$ over the inverse semigroup $S$ in the sense that 
\[
    B(\pi,v) = \clsp \{ A_t : t\in S \}, \qquad A_s A_t \subseteq A_{st}, \qquad \text{$s\leq t$ implies $A_s\subseteq A_t$} ,
\]
for all $s,t \in S$. 
In fact, we have
\begin{enumerate}
    \item\label{enu:range_of_covariant_rep1} $\pi(a_s)v_s \cdot \pi(a_t)v_t = \pi \big( \alpha_s(\alpha_{s}^{-1}(a_s)a_t) u(s,t) \big)v_{st}$ and $\alpha_s \big( \alpha_{s}^{-1} (a_s) a_t \big) \in I_{st}$ for all $a_t\in I_{t}, a_s\in I_{s}$, $s,t\in S$;
    \item\label{enu:range_of_covariant_rep2} $s\leq t$ implies $I_{s}\subseteq I_{t}$ and $\pi(a)v_s = \pi \big( a u(ss^*,t)^{-1} \big) v_t$ for any $a \in I_{s}$;
\end{enumerate}
If in addition $A$ and $B$ are $C^*$-algebras then $(\pi(a_t) v_t)^* = \pi \big( \alpha_{t}^{-1}(a_{t}^*) u(t^*,t)^{-1} \big)v_{t^*}$ for all $a_t\in I_{t}$, $t\in S$, so $A_t^* = A_{t^*}$, and $\{A_t\}_{t\in S}$ is a saturated grading of $B(\pi,v)$ over the inverse semigroup of $S$ in the sense of \cite[Definition 6.15]{Kwa-Meyer0}, see also \cite[Definition 7.1]{Exel}.
\end{lem}
\begin{proof} 
\ref{enu:range_of_covariant_rep1} For any 
$a_t\in I_{t}, a_s\in I_{s}$, $s,t\in S$, we have $\alpha_{s}^{-1}(a_s) a_t\in I_{s^*} I_{t}\subseteq I_{s^*}\cap I_{t}$ and so
$\alpha_{s}(\alpha_{s}^{-1}(a_s)a_t)\in I_{st}$ because $\alpha_s(I_{s^*}\cap I_{t})=I_{st}$, see Remark \ref{rem:twisted_relations}.
Using \ref{item:covariant_representation1} we get
\[
    \pi(a_s)v_s \cdot \pi(a_t)v_t = v_s \pi \big( \alpha_{s}^{-1}(a_s) \big) \cdot \pi(a_t) v_t = v_s \pi \big( \alpha_{s}^{-1}(a_s) a_t \big) v_t = \pi \Big( \alpha_s \big( \alpha_{s}^{-1}(a_s) a_t \big) u(s,t) \Big) v_{st}. 
\]

\ref{enu:range_of_covariant_rep2} If $s\leq t$ then $I_{s}\subseteq I_{t}$ by Remark~\ref{rem:twisted_relations}. 
Using \ref{item:covariant_representation3} and \ref{item:covariant_representation2}, for any $a \in I_{s}$ we get  $ \pi(a)v_t=\pi(a)v_{ss^*} v_t= \pi(au(ss^*,t))v_{ss^*t} = \pi(au(ss^*,t))v_{s}$. 

If $B$ is a $C^*$-algebra then $B''$ is again a $C^*$-algebra -- the enveloping $W^*$-algebra of $B$ (this also holds in the real case, see \cite[Theorem 5.5.3]{Li_book}). 
In particular, $\pi''(1_{t})v_t = v_t\pi''(1_{t^*})$ is a partial isometry whose adjoint is $\pi''(u(t^*,t)^{-1})v_{t^*}$, see (the proof of) Proposition~\ref{prop:normalized_twisted_rep}. 
Thus for $a_t\in I_{t}$, $t\in S$, we get
$
    \big( \pi(a_t) v_t \big)^* = \Big( v_t \pi''(1_{t^*}) \pi \big( \alpha_{t}^{-1}(a_t) \big) \Big)^* = \pi \big( \alpha_{t}^{-1}(a_t) \big)^* \big( \pi''(1_{t^*}) v_t \big)^* = \pi \big( \alpha_{t}^{-1}(a_t) \big)^* \pi'' \big( u(t^*,t)^{-1} \big) v_{t^*}.
$
If in addition $A$ is a $C^*$-algebra then both $\pi$ and $\alpha_{t}$ are necessarily $*$-preserving, and thus $(\pi(a_t)v_t)^* = \pi(\alpha_{t}^{-1}(a_{t}^*)u(t^*,t)^{-1})v_{t^*}$.
\end{proof}

For any twisted action $(\alpha,u)$ the subspace $\ell^1(\alpha,u) := \{ f \in \ell^1(S,A) : \text{$f(t) \in I_{t}$ for all $t\in S$} \}$ of $\ell^1(S,A)$ is a Banach algebra with multiplication
\[
    (f * g)(r) := \sum_{st = r} \alpha_s \Big( \alpha_{s^*} \big( f(s) \big) g(t) \Big) u(s,t),
\]
see (the proof of) \cite[Proposition 3.1]{Sieben98}. 
If $A$ is a $C^*$-algebra then in fact $\ell^1(\alpha,u)$ is a Banach $*$-algebra with involution
given by $f^*(t) := \alpha_{t}^{-1}(f(t^*)^*) u(t^*,t)^{-1}$, $t\in S$.
Lemma~\ref{lem:range_of_covariant_rep} (see also \cite[Proposition 3.7]{Sieben98}) implies that every covariant representation $(\pi,v)$ of $\alpha$ in a Banach algebra $B$ `integrates' to a representation  $\pi\times v : \ell^1(\alpha,u) \to B$ given by
\[
    \pi\times v(f) := \sum_{t\in S} \pi \big( f(t) \big) v_t, \qquad f \in \ell^1(\alpha,u),
\]
and $\overline{\pi\times v(\ell^1(\alpha,u))} = B(\pi,v)$ is a Banach algebra. 

\begin{defn}
For a fixed class $\RR$ of covariant representations of $(\alpha,u)$  we define the \emph{$\RR$-crossed product} for $(\alpha,u)$, denoted $A\rtimes_{(\alpha,u),\RR} S$, as the Hausdorff completion of $\ell^1(\alpha,u)$ in the seminorm
\[
    \| f \|_{\RR} := \sup \{ \| \pi\times v(f)\| : (\pi,v)\in \RR \} .
\]
The \emph{universal Banach algebra crossed product} $A\rtimes_{(\alpha,u)} S$ is the crossed product associated to the class of all covariant representations of $(\alpha,u)$.
\end{defn}

\begin{rem}\label{rem:group_crossed^products}
If $S = G$ is a group then we necessarily have  $I_t = A$ for all $t\in G$, so   $\alpha : G \to \Aut(A)$, and $A \rtimes_{(\alpha,u)}G=\ell^{1}(\alpha,u)$ is $\ell^1(G,A)$ as a Banach space. 
When $A=C_0(X)$  such crossed products are studied in  \cite{BK}, see also \cite{BKM}.
For general $A$, but when the twist $u\equiv 1$ is trivial,  $(A,G,\alpha)$ is a Banach algebra dynamical system in the sense of \cite{DDW}, and the algebras $A\rtimes_{\alpha, \RR} G$ are examples of crossed products considered in \cite{DDW}. 
\end{rem}

\subsection{Covariant representations in dual Banach algebras and on Banach spaces}
\label{ssec:CovariantRepsDualBanachAlgebras}

We fix a twisted action $(\alpha,u)$ of an inverse semigroup $S$ on $A$. 
We will now discuss situations, beyond the one described in Remark~\ref{rem:unital_actions}, when in the study of algebras generated by covariant representations we do not need to pass through the bidual algebra $B''$. 

\begin{defn} \label{de:PredualNormalizedRepresentation}
Let $(\pi,v)$ be a covariant representation of $(\alpha,u)$ in a Banach algebra $B$. 
If $B$ has a predual Banach space $B_*$, we say that $(\pi,v)$ is \emph{$B_*$-normalized} if, for every $t \in S$, we have $v_{t} = B_*\mhyphen\lim_{i} ( \pi(\mu_{i}^t) v_{t} )$, for an approximate unit $\{ \mu^t_i \}_i$ in $I_t = I_{tt^*}$, where $B_*\mhyphen\lim$ denotes the limit in the weak$^*$ topology on $B$ induced by $B_*$. 
In particular, then $v : S \to B$ takes values in $B$ (rather than in $B''$).
\end{defn}

A \emph{dual Banach algebra} is a pair $(B, B_*)$ where $B$ is a Banach algebra, $B_*$ is a predual of $B$, and in addition multiplication in $B$ is separately $B_*$-continuous, see \cite{Runde}, \cite{Daws}. Examples of dual Banach algebras include all $W^*$-algebras (complex or real) with their unique preduals, and all algebras $B(E)$ of all bounded operators on a reflexive Banach space $E$, equipped with the canonical predual Banach space $E' \widehat{\otimes} E$.
By \cite[Theorem 5.6]{Ilie_Stokke}, every representation $\pi : A \to B$ into a dual Banach algebra $(B, B_*)$ has a unique extension to a representation $\overline{\pi} : \Mult(A) \to B$ which is strictly-$B_*$-continuous; it is given by  $\overline{\pi}(m) := B_*\mhyphen\lim_{i} \pi(m \mu_{i})$ for an approximate unit $\{\mu_{i}\}_i$ in $A$. 

\begin{prop}\label{prop:B_*_normalized_twisted_rep}
Assume $(B,B_*)$ is a dual Banach algebra. 
For any covariant representation $(\pi,v)$ of $(\alpha,u)$ in $B$ there is a unique $B_*$-normalized covariant representation $(\pi,\tilde{v})$ such that $\pi(a) v_t = \pi(a) \tilde{v}_{t}$ for all $a\in I_{t}$ and $t\in S$.
A pair $(\pi,v)$ is a $B_*$-normalized covariant representation of $(\alpha,u)$ in $B$ if and only if $\pi : A \to B$ is a representation and $v : S\to B_{1}$ is a map such that every $v_t$, $t\in S$, admits a contractive generalised inverse $v_t^*$, satisfying
\[
    v_t \pi(a) v_t^* = \pi \big( \alpha_{t}(a) \big), \quad v_s v_t = \pi_{st} \big( u(s,t) \big) v_{st} , \quad \text{$v_t v_t^* =\pi_t(1_{t})$ and $v_t^* v_t  =\pi_{t^*}(1_{t^*})$},
\]
where $a\in I_{t^*}$, $s,t\in S$ and $\pi_t : \Mult(I_t) \to B$ is the unique strictly-$B_*$-continuous extension of $\pi|_{I_t}$.
\end{prop}
\begin{proof} 
Follow the proof of Proposition~\ref{prop:normalized_twisted_rep} (some arguments can be simplified as the  multiplication in $B$ is $B_*$-continuous in both variables). 
\end{proof}

\begin{rem}
For any Banach algebra $B$ the pair $(B'',B')$ is a dual Banach algebra if and only if $B$ is Arens regular. 
Thus Proposition~\ref{prop:normalized_twisted_rep} cannot be deduced from Proposition~\ref{prop:B_*_normalized_twisted_rep}. 
We were not able to find a nice common generalization of these two statements.
\end{rem}

\begin{rem} 
Proposition~\ref{prop:B_*_normalized_twisted_rep} implies that a $B_*$-normalized representation of 
an untwisted action $\alpha$ (an action with a trivial twist) in a dual Banach algebra $B$ is a pair $(\pi,v)$, where $\pi : A \to B$ is a representation and $v : S \to B_{1}$ is a semigroup homomorphism, such that $v_t \pi(a) v_{t^*} = \pi(\alpha_{t}(a))$ and $v_e = \pi_e (1_{e})$ for $a\in I_{t^*}$, $s,t\in S$, $e\in E(S)$.
\end{rem}

We now relate the above discussion to representations on Banach spaces. 
Note that if $\pi : A \to B(E)$ is a representation then, for each $t\in S$, $\pi$ yields a non-degenerate representation $I_t \to B(\overline{\pi(I_{t})E})$ which thus uniquely extends to a representation 
\[
    \overline{\pi}_t : \Mult(I_t) \to B(\overline{\pi(I_{t})E}) ;\ \overline{\pi}_t(m) \big( \pi(a)x \big) := \pi ( ma ) x , \qquad m\in \Mult(I_t),\ a\in I_{t},\ x\in E , 
\]
see for instance \cite[Theorem 4.1]{Gardella_Thiel1}.

\begin{defn}\label{defn:covariant_representation_on_space}
A \emph{covariant representation of $(\alpha, u)$ on a Banach space $E$} is a pair $(\pi,v)$ where $\pi : A \to B(E)$ is a representation and $v : S \to B(E)_{1}$ has the property that every $v_t$, $t\in S$, admits a contractive generalised inverse $v_t^*$, satisfying: 
\begin{enumerate}[labelindent=40pt,label={(SCR\arabic*)},itemindent=1em] 
    \item\label{item:covariant_representation1'} $v_t \pi(a) v_{t}^* = \pi(\alpha_t(a))$ for all $a\in I_{t^*}$, $t\in S$;
    \item \label{item:covariant_representation2'} the ranges of $v_t $ and $v_t^*$ are  $\overline{\pi(I_{t})E}$ and $\overline{\pi(I_{t^*})E}$ respectively, for all $t\in S$;
    \item\label{item:covariant_representation3'} $v_s v_t = \overline{\pi}_{st}(u(s,t)) v_{st}$ for all $s,t\in S$.
\end{enumerate}
Equality in \ref{item:covariant_representation3'} holds pointwise and it makes sense due to \ref{item:covariant_representation2'}.
We  say that $(\pi,v)$ is non-degenerate, injective, etc.\ if $\pi$ has that property.
\end{defn}

\begin{rem}\label{rem:space_covariant_reps} 
If $\F = \C$, $E = H$ is a Hilbert space and $A$ is a $C^*$-algebra then Definition~\ref{defn:covariant_representation_on_space} is equivalent to \cite[Definition 3.2]{Sieben98} and \cite[Definition 6.2]{Buss_Exel}. 
\end{rem}

\begin{lem}\label{lem:covariant_reps_implies_MP^partial_isos}
Let $(\pi,v)$ be a covariant representation of $(\alpha,u)$ on a Banach space $E$. 
Then $(\pi,v)$ is a covariant representation in the Banach algebra $B(E)$ such that each $v_e\in B(E)$, $e\in E(S)$, is a  strong limit of $\{\pi(\mu_i^e)\}_{i}$, where  $\{\mu_{i}^e\}_{i}$ is a contractive approximate unit in $I_{e}$.
\end{lem}
\begin{proof} 
The arguments in the beginning of the proof of Proposition~\ref{prop:normalized_twisted_rep} show that $(\pi,v)$ is a covariant representation in the Banach algebra $B(E)$. 
Then each $v_e\in B(E)$, $e\in E(S)$, is a projection onto $\overline{\pi(I_e)E}$ which commutes with elements of $\pi(I_e)$, and thus $v_e$ is a  strong limit of $\{\pi(\mu_i^e)\}_{i}$.
\end{proof}

\begin{cor} \label{cor:covariant_reps_implies_MP^partial_isos}
If $A$ is a complex $C^*$-algebra and $(\pi,v)$ is a non-degenerate covariant representation of $(\alpha,u)$ on a complex Banach space $E$ then each $v_t$, $t \in S$, is an $MP$-partial isometry on  $E$.
\end{cor}
\begin{proof} 
Since $\pi$ is non-degenerate we may extend it to the unitization of $A$ and so we may assume that $A$ and $\pi$ are in fact unital. 
Since $A$ is a $C^*$-algebra, for each $e\in E(S)$ we may choose an approximate unit $\{\mu_{i}^e\}_{i}$ in $I_e$ consisting of hermitian operators. 
Then $\{\pi(\mu_{i}^e)\}_{i}$ are hermitian operators by \cite[Lemma 2.4]{cgt}. 
As $v_e$ is a strong limit of $\{\pi(\mu_{i}^e)\}_{i}$, this implies that $v_e$ is also hermitian. 
Indeed, recall that $a\in B(E)$ is hermitian if and only if its numerical range $V(a) = \{ f(ax): x\in E,\ f\in E',\ \| f \| = \| x \| = 1 = f(x) \}$ is real. 
Hence for any $x\in E,\ f\in E'$ with $\| f \| = \| x \| = 1 = f(x)$ we get $f(v_ex) = f(\lim_i \pi( \mu_{i}^e) x)\in \R$, and so $v_e$ is hermitian. 
As $v_t v_t^* = v_{tt^*}$ and $v_t^* v_t = v_{t^*t}$, where $tt^*, t^*t\in E(S)$, for all $t\in S$, we conclude that $\{v_t\}_{t\in S}\subseteq B(E)_{1}$ are $MP$-partial isometries.
\end{proof}

\begin{lem}\label{lem:reflexive_covariant_representations}
Assume that $E$ is a reflexive Banach space, and recall that the projective tensor product Banach space $B_* := E' \mathbin{\widehat{\otimes}} E$ is naturally a predual of $B := B(E)$ making it a dual Banach algebra.
A covariant representation $(\pi,v)$ of $(\alpha,u)$ in $B(E)$ is $B_*$-normalized if and only if $(\pi,v)$ is a covariant representation on the Banach space $E$.
\end{lem}
\begin{proof}
We identify $f \otimes \xi \in E' \mathbin{\widehat{\otimes}} E = B_*$ with the functional map $B\ni b\mapsto f \otimes \xi (b) := f(b\xi)\in \F$.
Assume $(\pi,v)$ is $B_*$-normalized, and so $\{v_t\}_{t\in S}\subseteq B(E)$ are partial isometries with the associated generalized inverses $\{v_t^*\}_{t\in S}\subseteq B(E)$, satisfying the relations described in Proposition~\ref{prop:B_*_normalized_twisted_rep}, which include \ref{item:covariant_representation1'}. 
For $e\in E(S)$ we have $v_e = \pi_e(1_e) = B_*\mhyphen\lim\pi(\mu_i^e)$, where $\{\mu_{i}^e\}_{i}$ is an approximate unit in $I_{e}$.
Hence $f(v_e\xi) = \lim_{i} f(\pi(\mu_i^e)\xi)$ for all $(f,\xi)\in E'\times E$, which implies that for any $\xi\in E$ we have $v_e\xi\in \overline{\pi(I_{e})E}^{E'}$. 
As the weak and norm closures of convex sets coincide this means that the range of $v_e$ is $\overline{\pi(I_{e})E}$.
Thus, for each $t\in S$, we have $v_{tt^*} = \pi_t(1_t)$ is a projection onto the space $\overline{\pi(I_{t})E} = \overline{\pi(I_{tt^*})E}$. 
This  implies \ref{item:covariant_representation2'} because $v_t = \pi_t(1_t) v_t$ and  $v_t^* = \pi_t(1_{t^*})v_{t}^*$. 
We also have $\pi_t(m) = \overline{\pi}_t(m) \pi_t(1_t)$ for $m\in \Mult(I_t)$, which immediately gives \ref{item:covariant_representation3'} from $v_s v_t = \pi_{st}(u(s,t))v_{st}$. 

Conversely, assume $(\pi,v)$ is a covariant representation in the sense of Definition~\ref{defn:covariant_representation_on_space}.
By Lemma~\ref{lem:covariant_reps_implies_MP^partial_isos} each $v_{tt^*}$ is a strong limit of $\{\pi(\mu_i^t)\}_{i}$, where  $\{\mu_{i}^t\}_{i}$ is an approximate unit in $I_t = I_{tt^*}$. 
Using this, one sees that $v_{tt^*} = \pi_t(1_t)$ and $\pi_t(m) = \overline{\pi}_t(m)\pi_t(1_t)$ for $m \in \Mult(I_t)$. 
Now one readily gets the relations in Proposition~\ref{prop:B_*_normalized_twisted_rep}.
\end{proof}

\begin{cor} \label{cor:normalization_to_spatial_cov_rep}
For any covariant representation $(\pi,v)$ of $(\alpha,u)$ in $B(E)$, where $E$ is a reflexive Banach space, there is a unique covariant representation $(\pi,\tilde{v})$ on $E$ such that $\pi(a) v_t = \pi(a) \tilde{v}_{t}$ for all $a\in I_{t}$ and $t\in S$ (in fact $(\pi,\tilde{v})$ is $E' \mathbin{\widehat{\otimes}} E$-normalized).
\end{cor}
\begin{proof} 
Combine Lemma~\ref{lem:reflexive_covariant_representations} and Proposition~\ref{prop:B_*_normalized_twisted_rep}.
\end{proof}

\begin{cor}\label{cor:Cstar_algebraic_inverse_semigroup_crossed^product}
Let $(\alpha,u)$ be a twisted inverse semigroup action on a $C^*$-algebra $A$. 
Let $\RR_{C^*}$ be the class of all covariant representations of $(\alpha,u)$ in some $C^*$-algebra, and let $\RR_{\textup{Hil}}$ be the class of all covariant representations of $(\alpha,u)$ on some Hilbert space. 
Then 
\(
    A \rtimes_{(\alpha,u),\RR_{C^*}} S = A \rtimes_{(\alpha,u),\RR_{\textup{Hil}}} S ,
\)
and  this coincides with the crossed product $A \rtimes_{(\alpha,u)}S$ introduced in \cite{Buss_Exel}, \cite{Sieben98}. 
\end{cor}
\begin{proof}
Since $\RR_{\textup{Hil}} \subseteq \RR_{C^*}$ we get $\|\cdot\|_{\RR_{\textup{Hil}}} \leq \|\cdot\|_{\RR_{C^*}}$. 
This is equality, because for any covariant representation $(\pi,v)\in \RR_{C^*}$ in a $C^*$-algebra $B$, we may assume that the enveloping $W^*$-algebra $B''$ is faithfully represented on a Hilbert space $H$. 
Then, by Corollary~\ref{cor:normalization_to_spatial_cov_rep}, the renormalized covariant representation  $(\pi,\tilde{v})$ is a  covariant representation of $\alpha$ on $H$, so $(\pi,\tilde{v})\in \RR_{\textup{Hil}}$. 
The claim follows because Corollary~\ref{cor:normalization_to_spatial_cov_rep} implies $\pi\times v(f) = \pi\times \tilde{v}(f)$ for all $f \in \ell^1(\alpha,u)$.
\end{proof}

\section{Groupoid Banach algebras  and inverse semigroup disintegration}
\label{sec:GroupoidBanachAlgebras}

Let $\G$ be an \'etale groupoid with locally compact Hausdorff unit space $X$. 
By a twist over $\G$ we  mean a Fell line bundle \cite{Kumjian}.
This is equivalent to twists introduced by Kumjian in \cite{Kumjian0}, see \cite[2.5.iv]{Kumjian}.
More specifically, let $\LL$ be a  line bundle over $\G$. 
Equivalently, $\LL = \bigsqcup_{\gamma\in \G } L_\gamma$ is a topological space such that the canonical projection $\LL \onto \G$ is continuous and open, each fiber $L_\gamma \cong \F$ is a one dimensional Banach space with a structure consistent with the topology of $\LL$, in the sense that the maps 
\[
   \bigsqcup_{\gamma\in \G} L_\gamma \times L_\gamma \ni (z, w) \mapsto z + w\in \LL, \quad 
    \F\times \LL\ni (\lambda,z) \mapsto \lambda z\in \LL , \quad 
    \LL\ni z\mapsto  |z|\in \R
\]
are continuous. 
Then $\LL$ is necessarily locally trivial. 
For any open $U \subseteq \G$ we write $C(U,\LL)$ for the linear space of continuous sections in the restricted bundle 
$\LL|_U := \bigsqcup_{\gamma\in U } L_\gamma$  (i.e. continuous maps $f : U \to \LL$ with $f(\gamma)\in L_\gamma$ for $\gamma \in U$).
Similarly we denote by $C_c(U,\LL)$, $C_0(U,\LL)$, $C_b(U,\LL)$, $C_u(U,\LL)$ the spaces of continuous sections with compact support, vanishing at infinity, bounded and unitary, respectively. 
Then $C_0(U,\LL)$ is a Banach space with the norm $\|f\|_{\infty} := \max_{\gamma\in U}| f(\gamma)|$. 
Since $\LL$ is locally trivial, every point in $\G$ has an open neighbourhood $U$ and unitary section $f \in C(U,\LL)$, i.e. $|f(\gamma)| = 1$, for all $\gamma \in U$.
This section induces an isometric isomorphism $C_0(U,\LL)\cong C_0(U)$, as for every $g\in C_0(U,\LL)$ there is a unique element $g/f \in C_0(U)$ such that $g(\gamma) = (g/f)(\gamma) f(\gamma)$, $\gamma \in U$.

\begin{defn} 
A \emph{twist over the groupoid}  $\G$ is a  line bundle $\LL=\bigsqcup_{\gamma\in \G } L_\gamma$ over $\G$ together with a continuous \emph{multiplication} ${\cdot}:\bigsqcup_{d(\gamma)=r(\eta)\in \G } L_{\gamma}\times L_{\eta} \to \LL$, which is associative (whenever the product makes sense), makes the modulus  multiplicative, and induces bilinear maps $L_{\gamma}\times L_{\eta}\longrightarrow L_{\gamma\eta}$; 
and a continuous \emph{involution} $^{*}: \LL\to \LL$ that restricts to conjugate linear maps $ L_{\gamma}\to L_{\gamma^{-1}}$, preserves the modulus (norm), is anti-multiplicative and $z\cdot z^*=|z|^2\in L_{r(\gamma)}\cong \F$ for every $z\in L_{\gamma}$, $\gamma \in \G$. 
See \cite{Kumjian} for more details.
\end{defn}

\begin{lem}\label{lem:inverse_semigroup_of_trivialtwist_bisections} 
Let $(\G,\LL)$ be a twisted groupoid. 
The set $S(\LL)$ of bisections $U$ for which the restricted line bundle $\LL|_{U}$ is trivial forms a unital wide inverse subsemigroup $S(\LL)\subseteq \Bis(\G)$.
\end{lem}
\begin{proof}  
The restricted bundle $\LL|_U$ is trivial if and only if there is a continuous unitary section $c_U\in C_u(U,\LL)$. 
If $\LL|_{U}$ and $\LL|_{V}$ are trivial, and $c_U \in C_u(U,\LL),\ c_V \in C_u(V,\LL)$, then $\LL|_{UV}$ and $\LL|_{U^{-1}}$ are trivial because the convolution product $c_{U}* c_V \in C_u(UV,\LL)$ and $c_U^*\in C_u(U^{-1},\LL)$ (in a Fell line bundle the norm $| \cdot |$ is multiplicative, rather than just submultiplicative). 
Thus $S(\LL)\subseteq \Bis(\G)$ is an inverse semigroup. 
Since $\LL$ is locally trivial, $S(\LL)$ is wide.
Moreover, $\LL$ is trivial over $X$ as each fiber $L_x\cong \F$ is a one-dimensional algebra with unit $1_x$, and the unit section defined by $1(x) := 1_x$ is unitary and continuous. 
\end{proof}

According to Lemma~\ref{lem:inverse_semigroup_of_trivialtwist_bisections} we may and we shall always  assume that
$$
\text{\emph{ $\LL|_X = X \times \F$ is trivial}.}
$$
\begin{ex}[Twist by a cocycle]\label{ex:groupoid_cocycle}
The trivial bundle $\LL = \G\times \F$ with pointwise operations is a \emph{trivial twist}.
More generally, if $\sigma$ is a \emph{continuous normalised $2$-cocycle} on $\G$, that is a continuous map $\sigma : \G^{(2)} \to \{ z\in \F:  |z| = 1 \}$ satisfying 
\[
    \sigma \big( r(\gamma),\gamma \big) = 1 = \sigma \big( \gamma , d(\gamma) \big) , \qquad \sigma(\alpha,\beta)\sigma(\alpha\beta,\gamma) = \sigma(\beta,\gamma)\sigma(\alpha,\beta\gamma)
\]
for every composable triple $\alpha,\beta,\gamma \in \G$, 
then we may consider a twist $\LL_{\sigma}$ given by  the trivial bundle $\G\times \F$ with operations given by 
\[
    (\alpha,w) \cdot (\beta,z) := \big( \alpha \beta , \sigma(\alpha,\beta) wz \big) , \qquad  
    (\alpha,w)^{*} := \big( \alpha^{-1} , \overline{\sigma(\alpha^{-1},\alpha) w } \big) .
\]
Every twist based on a trivial line bundle is of this form.
\end{ex}

\subsection{Algebras of quasi-continuous sections}

Let us fix a twisted groupoid $(\G,\LL)$.
Recall our convention that $\G$ is locally compact and only locally Hausdorff. 
Thus the associated $*$-algebra will be defined on the set of quasi-continuous compactly supported sections:
\[
    \mathfrak{C}_c(\G,\LL) := \spane \{ f \in C_c(U,\LL): U\in  \Bis(\G) \} , 
\]
where we treat sections of $\LL|_U$ as sections of $\LL$ that vanish outside $U$. 
Note that $\mathfrak{C}_c(\G,\LL) = \spane\{f\in C_c(U,\LL): U\in \mathcal{C}\}$ for any cover  $\mathcal{C}\subseteq \Bis(\G)$ of $\G$, cf. \cite[Proposition 3.10]{Exel}.
In particular, we may always take $\mathcal{C} = S(\LL)$. If the groupoid $\G$ is Hausdorff, then $\mathfrak{C}_c(\G,\LL) = C_c(\G,\LL)$ is the usual space of continuous compactly supported sections. 
We define a $*$-algebra structure on $\mathfrak{C}_c(\G,\LL)$ by
\begin{equation}\label{eq:convolution_and_involution}
    (f*g)(\gamma) := \sum_{r(\eta) = r(\gamma)} f(\eta) \cdot g(\eta^{-1} \gamma), \qquad (f^*)(\gamma) := f(\gamma^{-1})^* ,
\end{equation}
where $f,g \in \mathfrak{C}_c(\G, \LL)$, cf. \cite[Proposition 3.11]{Exel}. 
Since we assume that $\LL|_X$ is trivial, we get that $C_c(X)=C_c(X,\LL)$ is a $*$-subalgebra of $\mathfrak{C}_c(\G,\LL)$. 

The domain and range maps $d,r : \G\to X$ induce the following three submultiplicative norms on $\mathfrak{C}_c(\G,\LL)$:
\begin{equation}\label{eq:Hahn-norms}
    \| f \|_{d_*} := \sup_{x\in X} \sum_{d(\gamma)=x} |f(\gamma)|, \qquad 
    \| f \|_{r_*} := \sup_{x\in X} \sum_{r(\gamma)=x} |f(\gamma)|, 
\end{equation}
\begin{equation}\label{eq:Hahn-norm}
			\|f\|_{I} := \max\{ \|f\|_{d_*} , \|f\|_{r_*} \}, 
\end{equation}
see \cite[the proof of Theorem 2.2.1]{Paterson} or \cite[II, 1.4]{Renault_book}. 
The  norm $\|\cdot \|_{I}$ is called the \emph{$I$-norm} and was introduced by Hahn~\cite{Hahn}. 
It has the advantage that it is preserved by the involution $*$, while for the remaining norms we only have $\| f^* \|_{d_*} = \|f\|_{r_*}$. 
If $\|\cdot\|$ is a $C^*$-norm on the $*$-algebra $\mathfrak{C}_c(\G,\LL)$ then minimality of the $C^*$-norm on a $C^*$-algebra implies that $\|f\|=\|f\|_{\infty}$ for $f\in C_c(X)$ (as we have this for every $f\in C_0(U)\subseteq C_c(X)$ where $U$ is precompact open), and then the $C^*$-equality gives that $\|f\|=\|f\|_{\infty}$ for every $f\in C_c(U,\LL)$ and $U\in \Bis(\G)$.
Using this and the triangle inequality one infers that there is a largest $C^*$-norm $\|\cdot\|_{C^*}$ on $\mathfrak{C}_c(\G,\LL)$. 
We prove  below (see Corollary \ref{cor:Cstar_boundedness}) that  $\|\cdot\|_{C^*}$ does not exceed $\|\cdot\|_{I}$.
This is well known when $\G$ is Hausdorff ($\F=\C$ and there is no twist), see \cite[Lemma 3.2.3]{Sims}, cf. \cite[Lemma 4.3]{BHM}. 
In the non-Hausdorff case this is much more delicate, see \cite[Corollary 6.6]{Clark_Zimmerman}.

\begin{defn}
We write $F_I(\G,\LL) := \overline{\mathfrak{C}_c(\G,\LL)}^{\|\cdot\|_{I}}$, $F_{d_*}(\G,\LL) := \overline{\mathfrak{C}_c(\G,\LL)}^{\|\cdot\|_{d_*}}$, $F_{r_*}(\G,\LL) := \overline{\mathfrak{C}_c(\G,\LL)}^{\|\cdot\|_{r_*}}$, and $C^*(\G,\LL) := \overline{\mathfrak{C}_c(\G,\LL)}^{\|\cdot\|_{C^*}}$ 
for the completions in the corresponding norms. When twist is trivial we omit writing it.
\end{defn}

The important feature of the above norms is that when restricted to any subspace $C_c(U,\LL) \subseteq \mathfrak{C}_c(\G,\LL)$, where $U\in \Bis(\G)$, they coincide with the supremum norm $\| \cdot \|_\infty$.
Hence each space $C_0(U,\LL)$ embeds naturally into the completions of $\mathfrak{C}_c(\G , \LL)$ in each of these norms. 
We introduce yet another norm which is  maximal with respect to this property.
In fact we will associate norms to subsemigroups of $\Bis(\G)$. 

\begin{lem} Let $S\subseteq \Bis(\G)$ be a semigroup covering $\G$. 
There is a submultiplicative norm on $\mathfrak{C}_c(\G, \LL)$ given by 
\begin{equation}\label{eq:projective_norm}
    \| f \|_{\max}^{S} := \inf \left\{ \sum_{k=1}^{n} \| f_k \|_\infty : f = \sum_{k=1}^{n} f_k, \ f_k \in C_c(U_k, \LL), \ U_k \in S \right\}.
\end{equation}
This is the largest submultiplicative norm on $\mathfrak{C}_c(\G,\LL)$ which agrees with $\|\cdot\|_\infty$  on each  $C_c(U,\LL) \subseteq \mathfrak{C}_c(\G,\LL)$, for $U\in S$. 
If $S$ is an inverse semigroup, then the involution is isometric in  $\|\cdot \|_{\max}^{S}$.
\end{lem}
\begin{proof}
Let $f,g \in \mathfrak{C}_c(\G, \LL)$.
For any $\varepsilon>0$ there are $f_k \in C_c(U_k, \LL)$ and $g_l\in C_c(V_k, \LL)$, where $U_k, V_l\in S$, such that $f = \sum_{k=1}^{n} f_k$, $g = \sum_{l=1}^{m} g_l$ and $\| f \|_{\max}^{S} + \varepsilon > \sum_{k=1}^{n} \| f_k \|_\infty$ and $\| g \|_{\max}^{S} + \varepsilon > \sum_{l=1}^{m} \| g_l \|_\infty$. 
Then $f * g = \sum_{k,l} f_k * g_l$, and $f_k * g_l\in C_c(U_k V_l,  \LL)$, $U_k V_l\in S$. Therefore
\[
    \| f * g \|_{\max}^{S} \leq \sum_{k,l} \| f_k * g_l \|_\infty \leq \sum_{k,l} \| f_k \|_\infty \| g_l \|_\infty < \big( \|f\|_{\max}^{S} +\varepsilon \big) \big( \|g\|_{\max}^{S}  + \varepsilon \big) .
\]
This shows that $\|\cdot\|_{\max}^S$ is submultiplicative. 
The proof of the triangle inequality is even simpler.  
If in addition $f \in C_c(U,\LL)$ for some $U\in S$, then there is $\gamma_0 \in U$ where $|f|$ attains its maximum, and then (retaining the above choice of $f_k$'s) 
$
    \| f \|_\infty = | f(\gamma_0) | \leq \sum_{k=1}^{n} |f_k(\gamma_0)| \leq  \sum_{k=1}^{n} \| f_k \|_\infty < \| f \|_{\max}^{S} + \varepsilon ,
$
which shows that $\|f\|_\infty\leq \|f\|^{S}_{\max}$. 
The converse inequality $\|f\|_{\max}^{S}\leq \|f\|_\infty$ holds by \eqref{eq:projective_norm}. 
Hence $\|\cdot\|^{S}_{\max}$ coincides with $\|\cdot\|_\infty$ on  $C_c(U,\LL)$ for $U\in S$. 
If $\|\cdot \|$ is any norm on $\mathfrak{C}_c(\G,\LL)$ that agrees with $\|\cdot\|_\infty$ on $C_c(U,\LL) \subseteq \mathfrak{C}_c(\G,\LL)$, for  $ U\in S$, then for any $f=\sum_{k=1}^{n} f_k$, $f_k\in C_{c}(U_k,\LL)$, $U_k\in S$, we have $\|f\|\leq \sum_{k=1}^{n} \|f_k\|= \sum_{k=1}^{n} \|f_k\|_\infty$, which implies $\|f\|\leq \|f\|_{\max}^{S}$. 
If $S$ is an inverse semigroup then $f^{*}=\sum_{k=1}^{n} f_k^{*}$, $f_k^{*}\in C_{c}(U_k^{*},\LL)$, where $U_k^{*}\in S$. 
Thus  $\| f^* \|_{\max}^{S}=\| f \|_{\max}^{S}$ by \eqref{eq:projective_norm}.
\end{proof}

\begin{defn}
For any unital inverse subsemigroup $S\subseteq \Bis(\G)$ covering $\G$, we denote by  $F^S(\G,\LL) := \overline{\mathfrak{C}_c(\G,\LL)}^{\|\cdot\|_{\max}^S}$ the associated Banach $*$-algebra. 
We also denote by $\|\cdot \|_{\max}$ the norm $\|\cdot \|_{\max}^{ \Bis(\G)}$ and  put $F(\G,\LL):=F^{ \Bis(\G)}(\G,\LL)$.
\end{defn}

\begin{rem}\label{rem:various_norms_and^projectivness}
The assumption that $S\subseteq \Bis(\G)$ is a unital means that  $S$ contains $X$, which is equivalent to that
$C_0(X)$ isometrically embeds into $F^S(\G,\LL)$.
We have 
\[
    \| \cdot \|_\infty \leq \| \cdot \|_{d_*},\, \|\cdot\|_{C^*},\,  \| \cdot \|_{r_*} \leq \| \cdot \|_{I} \leq \|\cdot \|_{\max} \leq \|\cdot \|_{\max}^S 
\]
on $\mathfrak{C}_c(\G, \LL)$. In particular, $\|\cdot \|_{\max}$ is  minimal  
amongst all norms $\|\cdot \|_{\max}^S$ on $\mathfrak{C}_c(\G, \LL)$ ($
S\subseteq T\subseteq \Bis(\G)\,\, \Longrightarrow\,\,\|\cdot \|_{\max}^T \leq \|\cdot \|_{\max}^S
$).
In general there is no maximal amongst the  norms $\|\cdot \|_{\max}^S$. 
In our analysis we may restrict only to wide inverse semigroups as  the `saturation' $\widetilde{S}:=\{U: U\subseteq V\in S\}$ of $S$ is a wide inverse subsemigroup of $\Bis(\G)$ and 
$\|\cdot\|_{\max}^S=\|\cdot\|_{\max}^{\widetilde{S}}$.
Moreover $\|\cdot \|_{\max}^S$ is majorized by a norm coming from an  inverse subsemigroup of $S(\L)$ from Lemma \ref{lem:inverse_semigroup_of_trivialtwist_bisections}.
Namely,   $\widetilde{S}\cap S(\LL)$ is a wide  inverse semigroup  and
\begin{equation}\label{eq:majorization_by_S(L)_semigroup}
\|\cdot\|_{\max}^S\leq \|\cdot\|_{\max}^{\widetilde{S}\cap S(\LL)}.
\end{equation}
\end{rem}

\begin{ex}
Let  $\G= G\times X$ be the transformation groupoid for a (discrete) group action $h:G\to \text{Homeo}(X)$. The natural 
 $G$-grading of $\G$ is given by 
$S:=\{\{g\}\times X\}_{g\in G}\cong G$.
For every  $f\in C_c(\G)$ we have
$$
 \|f\|_{d_*}=\max_{x\in X} \sum_{g\in G}|f(g,x)|,\,\,  \|f\|_{r_*}=\max_{x\in X} \sum_{g\in G}|f(g,\varphi_{g}(x))|\,\,\text{ and }\,\, \|f\|_{\max}^S=\sum_{g\in G} \max_{x\in X}|f(g,x)|.
$$
Monod \cite{Monod} denotes the algebra $F_{d_*}(\G) := \overline{\mathfrak{C}_c(\G)}^{\|\cdot\|_{d_*}}$  by $C(X,\ell^1 G)$ and  calls it the Banach algebra crossed product.  It plays crucial role in the study of amenability in \cite{Monod}.
However, it is much more common to reserve the name Banach algebra crossed product to the $*$-algebra   $\ell^1(G,C_0(X))\cong F^S(\G)$, see \cite{BK}, \cite{Jeu_Tomiyama},  and references therein. Note that the Hahn norm $\|\cdot\|_{I}$ is in general strictly smaller than $\|\cdot\|_{\ell^1(G,C_0(X))}= \|\cdot\|_{\max}^S$. 
Obviously, when $X$ is singleton, so that  $\G = G$ is a  group, then $S=\Bis(\G)$  and  all the  norms $\|\cdot\|_{d_*}, \|\cdot\|_{r_*}, \|\cdot\|_{I}, \|\cdot\|_{\max}^S$ coincide with the $\ell^{1}$-norm in $\ell^{1}(G)$.
\end{ex}


We fix a unital inverse subsemigroup $S\subseteq \Bis(\G)$ covering $\G$.
By \eqref{eq:majorization_by_S(L)_semigroup},  in the following definition we may always assume that $S$ is a wide inverse subsemigroup of $S(\LL)$. 

\begin{defn} 
For a class $\RR$ of representations of $F^{S}(\G, \LL)$ we denote by $F_{\RR}^S(\G, \LL) := \overline{\mathfrak{C}_c(\G, \LL)}^{\|\cdot\|_{\RR}}$ 
the Hausdorff completion  in the seminorm
$
    \|f\|_{\RR} := \sup \{ \| \psi(f) \|: \psi\in \RR \}. 
$ 
If $\|\cdot\|_{\RR}\leq \|\cdot\|_{\max}$, then  we will write $F_{\RR}(\G, \LL)$ for $F_{\RR}^S(\G, \LL)$.
\end{defn}
\begin{rem}  
If $\|\cdot\|_{\RR}\leq \|\cdot\|_{\max}$, then the algebra $F_{\RR}^S(\G, \LL)$ does not depend on the choice of $S$, as it
can be viewed as a Hausdorff completion of $F(\G, \LL)$.
In this paper we will be mainly interested in the case when  $\|\cdot\|_{\RR}\leq \|\cdot\|_{I}$ (and hence all the more  $\|\cdot\|_{\RR}\leq \|\cdot\|_{\max}$).
\end{rem}

All the aforementioned Banach algebras have approximate units.

\begin{lem}\label{lem:approximate_units} 
Let $S\subseteq \Bis(\G)$ be a unital inverse semigroup covering $\G$. An approximate unit $\{\mu_i\}_{i}$ in $C_0(X)$ 
is also an approximate unit in $F^{S}(\G,\LL)$, and for any representation $\psi : F^{S}(\G,\LL)\to B$ the net $\{\psi(\mu_i)\}_{i}$ is an approximate unit in $\overline{\psi(F^{S}(\G,\LL))}$. 
In particular, if $X$ is compact then $F^{S}(\G,\LL)$ and $\overline{\psi(F^{S}(\G,\LL))}$ are unital.
\end{lem}
\begin{proof}
It suffices to prove that $\{\mu_i * f\}_{i}$ converges to $f$ in $\|\cdot \|_{\max}^{S}$ for $f\in C_c(U, \LL)$, $U\in S$. 
But then $\mu_i * f - f\in C_c(U,\LL)$ and therefore $\|\mu_i * f - f\|_{\max}^{S} = \|\mu_i * f- f\|_{\infty} = \max_{\gamma \in U}| \mu_i(r(\gamma))f(\gamma) - f(\gamma)|$, which obviously tends to zero.
\end{proof}

We write $A \cong B$ to indicate that two Banach algebras $A$ and $B$ are isometrically isomorphic. 
We write $A \anti B$ if there is an isometry from $A$ onto $B$ which is anti-linear and anti-multiplicative. 
We write $A \donto B$ if there is a representation of $A$ in $B$ with dense range. 
By Remark~\ref{rem:various_norms_and^projectivness} the identity on $\mathfrak{C}_c(\G,\LL)$
extends to canonical representations 
$F^S(\G,\LL) \donto F_I(\G,\LL) \donto  F_{d_*}(\G,\LL), F_{r_*}(\G,\LL) .
$

\subsection{Opposite algebras and twists}

We denote by $A^{\op}$ the \emph{opposite algebra} to an  algebra $A$, meaning that the linear structure of $A^{\op}$ is the same as that of $A$ but the multiplication is performed in the reverse order. 
When $A$ is a Banach algebra then $A^{\op}$ is a Banach algebra with the same norm.
Fix a twisted groupoid $(\G,\LL)$ and an inverse semigroup $S\subseteq \Bis(\G)$ covering $\G$. 
An \emph{opposite twist} $\LL^{\op}$ on $\G$ is defined as follows (see \cite{Buss_Sims}). 
The fibers are $L_{\gamma}^{\op} := L_{\gamma^{-1}}$, $\gamma \in \G$, and the topology is induced by the sections
\begin{equation}\label{eq:widecheck_map}
    (\widecheck{f})(\gamma) := f(\gamma^{-1}), 
\end{equation}
where $f\in C_c(U,\LL)$, $U\in \Bis(\G)$. 
The involution on $\LL^{\op}$ is given by the involution maps $L_{\gamma^{-1}}\to L_{\gamma}$ from $\LL$,
 and the 
multiplication maps $ L_{\gamma}^{\op}\times L_{\eta}^{\op}\longrightarrow L_{\gamma\eta}^{\op}$  are given by $(z,w)\mapsto w \cdot z$ where $z\in L_{\gamma}^{\op}=\LL_{\gamma^{-1}}$, $w\in L_{\eta}^{\op}=\LL_{\eta^{-1}}$ and $d(\gamma)=r(\eta)$.
The map $f \mapsto \widecheck{f}$ yields a $*$-isomorphism
\[
    \mathfrak{C}_c(\G,\LL)^{\op} \cong \mathfrak{C}_c(\G,\LL^{\op}).
\]
In particular, the untwisted algebra $\mathfrak{C}_c(\G)^{\op}\cong \mathfrak{C}_c(\G)$ is self-opposite. 
It follows from the formulas for the norms $\|\cdot\|_{d_*}, \|\cdot\|_{r_*}, \|\cdot\|_{I}, \|\cdot\|_{\max}^{S}$ that the above isomorphism extends to isometric isomorphisms 
\[ 
    F^S(\G,\LL)^{\op} \cong F^S(\G,\LL^{\op}), \qquad  F_I(\G,\LL)^{\op} \cong F_I(\G,\LL^{\op}), \qquad  F_{d_*}(\G,\LL)^{\op} \cong F_{r_*}(\G,\LL^{\op}). 
\]
Thus the algebras $F^S(\G)$ and $F_I(\G)$ are self-opposite, while for the remaining two algebras we have only $F_{d_*}(\G)^{\op}\cong F_{r_*}(\G)$.
Let us consider a slightly more general situation.

\begin{defn} 
For a class $\EE$ of Banach spaces we  write $F_{\EE}^S(\G, \LL):=F_{\RR}^S(\G, \LL)$  and $\|\cdot\|_{\EE}:=\|\cdot\|_{\RR}$
for $\RR$ consisting of all  representations of $F^{S}(\G, \LL)$ in $B(E)$ for $E\in\EE$.
\end{defn}


If $\EE$ is a class of Banach spaces we denote by $\EE'$ the class of dual spaces to spaces in $\EE$. 
For any other class $\FF$ of Banach spaces we write $\EE\subseteq_{\iso}\FF$ if for every $E\in \EE$ there is $F\in \FF$ such that $E\cong F$. 

\begin{prop}\label{prop:duality} 
If $\EE'\subseteq_{\iso}\FF$ then the map $f \mapsto \widecheck{f}$ induces a contractive homomorphism $F_{\FF}^{S}(\G, \LL)^{\op} \donto F_{\EE}^{S}(\G, \LL^{\op})$.
If $\EE'\subseteq_{\iso} \FF$ and $\FF' \subseteq_{\iso}\EE$ then this homomorphism is an isometric isomorphism $F_{\FF}^{S}(\G, \LL)^{\op} \cong F_{\EE}^{S}(\G, \LL^{\op})$.
\end{prop}
\begin{proof}
For any representation $\pi : F^{S}(\G, \LL)^{\op} \to B(E)$ the formula $\pi'(f) := \pi(\widecheck{f})'$,  $f\in \mathfrak{C}_c(\G, \LL^{\op})$,
 defines a representation $\pi' :F^{S}(\G, \LL^{\op}) \to B(E')$. 
If $\EE'\subseteq_{\iso}\FF$ and $E\in \EE$, then we may replace $E'$ by $F\in \FF$ such that $E'\cong F$, so that we get a representation  $\pi' : F^{S}(\G, \LL^{\op}) \to B(F)$. 
This implies that $\| \widecheck{f} \|_{\EE}^{S}\leq \| f \|_{\FF}^{S}$ for $f \in \mathfrak{C}_c(\G, \LL^{\op})$.
Thus the isomorphism $\mathfrak{C}_c(\G, \LL)^{\op} \cong \mathfrak{C}_c(\G, \LL^{\op})$ extends to a representation $F_{\FF}^{S}(\G, \LL)^{\op} \donto  F_{\EE}^{S}(\G, \LL^{\op})$. 
Similarly, $\FF'\subseteq_{\iso}\EE$ implies $\|\cdot\|_{\FF}^{S}\leq \|\cdot\|_{\EE}^{S}$.
Hence  $F_{\FF}^{S}(\G, \LL)^{\op} \cong  F_{\EE}^{S}(\G, \LL^{\op})$ if $\EE'\subseteq_{\iso}\FF$ and $\FF'\subseteq_{\iso}\EE$. 
\end{proof}

\begin{cor}\label{cor:reflexive_spaces_oppositie_iso}
If $\EE$ consists of reflexive Banach spaces, then $ F_{\EE}^{S}(\G,\LL)^{\op} \cong F_{\EE'}^{S}(\G,\LL^{\op})$.
\end{cor}
\begin{proof}
Take $\FF = \EE'$ in Proposition~\ref{prop:duality}. 
\end{proof}

\begin{cor} \label{cor:Hilbert_spaces_oppositie_iso} 
We have $C^*(\G,\LL)^{\op}\cong C^*(\G,\LL^{\op})$ and $C^*(\G)^{\op}\cong C^*(\G)$.
\end{cor}
\begin{proof} 
Apply Corollary~\ref{cor:reflexive_spaces_oppositie_iso} to the case where $\EE$ consists of Hilbert spaces.
\end{proof}

The untwisted version of Corollary~\ref{cor:reflexive_spaces_oppositie_iso} is known (see \cite{Gardella_Thiel15}, \cite{Gardella_Lupini17}) and Corollary~\ref{cor:Hilbert_spaces_oppositie_iso} is a special case of the main result in \cite{Buss_Sims}. 
We can apply Proposition~\ref{prop:duality} beyond reflexive spaces.

\begin{cor}\label{cor:L_infty_lindenstrauss_spaces_etc}
Let $\EE$ be the class of $L^1$-spaces, $\FF_{1}$ the class of $L^{\infty}$-spaces, $\FF_{2}$ the class of $C_0$-spaces and $\FF_{3}$ the class of Lindenstrauss spaces, cf. Remark~\ref{rem:C_0_and_Lindenstrauss}.
Then
\[ 
    F_{\FF_1}^{S}(\G,\LL) \cong F_{\FF_2}^{S}(\G,\LL) \cong  F_{\FF_3}^{S}(\G,\LL) \cong F_{\EE}^{S}(\G,\LL^{\op})^{\op}.
\]
\end{cor}
\begin{proof}
We clearly have $\EE' \subseteq_{\iso}\FF_1 \subseteq_{\iso} \FF_2$ and $\FF_3' \subseteq_{\iso} \EE$. 
We have $\FF_2\subseteq \FF_3$ as $C_0$-spaces are abstract $M$-spaces and so their duals are (abstract) $L^1$-spaces, see for instance \cite[Theorem 7 on page 25 and Theorem 4 on page 135]{Lacey}. 
Thus $\FF_2'\subseteq \FF_3' \subseteq_{\iso} \EE$ and the assertion follows from the second part of Proposition \ref{prop:duality}. 
\end{proof}

\subsection{From twisted groupoids to twisted actions and back}
\label{subsec:twisted_groupoids_vs_twisted_actions}
We describe the correspondence between twisted groupoids and twisted actions of inverse semigroups on $C_0(X)$, which implicitly can be found in \cite{Buss_Exel}, \cite{Buss_Exel2}.
Let $(\G,\LL)$ be a twisted groupoid. 
Take any wide inverse subsemigroup $S\subseteq S(\LL)\subseteq \Bis(\G)$ and let $h : S\to \PHomeo(X)$ be the canonical action. 
Then we have the inverse semigroup action $\alpha : S \to \PAut(C_0(X))$  given by $\alpha_U(a) := a\circ h_{U^*}$ for $a\in  I_{d(U)} := C_0(d(U))$, $U\in S$.
For each $U\in S(\LL)$ take any continuous unitary section $c_U\in C_u(U,\LL)$, but if $U\subseteq X$ let $c_U(x) := (x,1) \subseteq X \times \F$, $x\in U$.  
For any $U,V\in S(\LL)$ there is a unique unitary function $u(U,V)\in C_u(r(UV))$ satisfying $u(U,V) \big( r(\gamma) \big) c_{UV}(\gamma\eta) = c_U(\gamma) \cdot c_V(\eta)$,  $\gamma \in U,\ \eta \in V,\ d(\gamma) = r(\eta)$.
In fact we have
\begin{equation}\label{eq:crucial_relations_for_twisted}
    \alpha_{U}(a) = c_U* a * c_U^*, \qquad u(U,V) = c_U* c_V * c_{UV}^*, \qquad  a\in C_0(d(U)),\ U, V\in S,
\end{equation} 
with operations given by \eqref{eq:convolution_and_involution}.

\begin{lem}\label{lem:twisted_action_from_twisted_groupoid} 
For any twisted groupoid $(\G,\LL)$ the pair $(\alpha,u)$ defined above is a twisted action of $S$ on $C_0(X)$.
\end{lem}
\begin{proof}
Since $C_0(X)$ is commutative \ref{enu:twisted actions1} is equivalent to $\alpha$ being an action. 
Axioms \ref{enu:twisted actions2}, \ref{enu:twisted actions3}, \ref{enu:twisted actions4} can be checked exactly as in the proof of \cite[Proposition 3.6]{Buss_Exel} using \eqref{eq:crucial_relations_for_twisted}. 
\end{proof}

\begin{ex} 
If the twist $\LL = \LL_{\sigma}$ comes from a $2$-cocycle $\sigma$ then $S(\LL_\sigma) = \Bis(\G)$ and for any $\eta\in U$,$\gamma\in V$, we have $u(U,V)(\eta\gamma) = \sigma(\eta, \gamma) \overline{\sigma((\eta\gamma)^{-1}, \eta\gamma)}$.
\end{ex}

Now let $(\alpha,u)$ be any twisted action of $S$ on $C_0(X)$. 
Since $C_0(X)$ is commutative $\alpha : S \to \PAut(C_0(X))$ is a semigroup homomorphism, and hence there is a semigroup homomorphism $h : S \to \PHomeo(X)$ with $\alpha_t(a) = a\circ h_{t^*}$ for $a\in I_{t^*}=C_0(X_{t^*})$, $t\in S$. 
Let $\G = S\ltimes_{h} X$ be the associated transformation groupoid.  
We define a  bundle $\LL$ over $\G$ whose elements are equivalence classes of triples $(a,t,x)$ for
$a\in C_0(X_{t})$, $x\in X_{t^*}\subseteq X$, $t\in S$, and two triples $(a,t,x)$ and $(a',t',x')$ are
equivalent if $x=x'$ and there is $v\in S$ with $v\leq t,t'$, where $x\in X_{v^*}$, and $(a \cdot u(vv^*,t))(h_{v}(x)) = (a' \cdot u(vv^*,t'))(h_{v}(x))$.  
In particular, we have a canonical surjection 
\[
    \LL \to \G ;\ [a,t,x] \mapsto [t,x] ,
\]
and we denote by $L_{[t,x]}$ the corresponding fibers. 
This makes $\LL$ a line bundle with operations defined by
$
    [a,t,x] + [b,t,x] := [a+b,t,x]$, $\lambda \cdot [a,t,x] := [\lambda a , t , x]$, $\big| [a,t,x] \big| := \big| a \big( h_{t}(x) \big) \big|$,
and the unique topology on $\LL$ such that the local sections $[t,x] \mapsto [a,t,x]$, for $x\in X_{t^*}$, $a\in C_0(X_t)$, $t\in S$, are continuous. 
Define the partial multiplication and involution on $\LL$ by
\begin{align*}
    [a,s,h_t(x)] \cdot [b,t,x]& := [\alpha_{s}(\alpha_{s}^{-1}(a) b) u(s,t),st,x] = [a (b\circ h_{s^*}) u(s,t),st,x], \\
    [b,t,x]^* &:= [\overline{b\circ h_{t}}\cdot \overline{u(t,t^*)} , t^* , h_{t}(x)] 
\end{align*}
for all $a\in C_0(X_s)$, $b\in C_0(X_t)$, $ x\in X_{(st)^*}$, $s,t\in S$. 

\begin{prop}\label{prop:twisted_groupoid_from_twisted_action}
The pair $(\G,\LL)$ described above is a well-defined twisted groupoid. 
\end{prop}
\begin{proof}
This follows from \cite[Theorem 3.22]{Buss_Exel2}, as the above construction is a special case of the construction in \cite[Section 3]{Buss_Exel2} applied to a Fell bundle over $S$ associated to $(\alpha,u)$, as described in \cite[p.\ 250]{Buss_Exel}.
\end{proof}

\begin{defn}
We say that a twisted inverse semigroup action $(\alpha,u)$ \emph{models a twisted \'etale groupoid}  $(\G,\L)$ if $(\G,\L)$ is isomorphic to the twisted groupoid associated to $(\alpha,u)$ as in Proposition~\ref{prop:twisted_groupoid_from_twisted_action}.
\end{defn}

\begin{lem}\label{lem:twisted_groupoids_come_from_twisted_actions}
Let $(\G,\LL)$ be a twisted \'etale groupoid. 
Let $S\subseteq S(\LL)\subseteq \Bis(\G)$ be any wide inverse semigroup and $(\alpha, u)$ a twisted action of $S$ as in Lemma~\ref{lem:twisted_action_from_twisted_groupoid}, obtained from a choice of unitary sections $c_U\in C_u(U,\LL)$, $U\in S$. 
Let $(S\ltimes X, \LL_{u})$ be the twisted groupoid associated to $(\alpha, u)$, as in Proposition~\ref{prop:twisted_groupoid_from_twisted_action}. 
Then
\[
    [a,U,x] \mapsto a \big( h_{U}(x) \big) c_U \big( d|_U^{-1}(x) \big), \qquad a\in C_0(r(U)),\ U\in S,\ x\in d(U),
\]
is an isomorphism of twisted groupoids $(S\ltimes X, \LL_{u}) \cong (\G,\LL)$. 
Thus $(\alpha, u)$ models $(\G,\LL)$.
\end{lem}
\begin{proof}
The map $S\ltimes X\ni [U,x]\mapsto d|_U^{-1}(x)\in \G$ is an isomorphism of topological groupoids. 
The sections $[U,x] \mapsto [a,U,x]$ and $\gamma\mapsto  a(r(\gamma)) c_U (\gamma)$,  $\gamma \in U$, $x\in d(U)$, $a\in C_0(r(U))$, $U\in S$, determine the topology of $\LL_{u}$ and $\LL$ respectively.
Thus the map described in the assertion is a homeomorphism. 
It is straightforward to check that it preserves the algebraic operations. 
\end{proof}

\subsection{Integration and disintegration in Banach algebras}
\label{subsec:integration}

We fix a  twisted action $(\alpha,u)$ of an inverse semigroup $S$ on $C_0(X)$ and let $(\G,\LL)$ be the associated twisted groupoid as described in the previous subsection. 
By Lemma~\ref{lem:twisted_groupoids_come_from_twisted_actions} every twisted \'etale groupoid arises in this way. 
The sets $U_t = \{ [t,x] : x \in X_t \}$, $t\in S$, together with $X$ form a unital wide inverse subsemigroup of $ S(\LL)\subseteq \Bis(\G)$. We denote it by $\overline{S}$. For each $t\in S$ we have a linear isometric isomorphism
\begin{equation}\label{eq:coeffcients_isomorphisms}
    C_c(X_{t})\ni a_t\mapsto a_t\delta_t\in C_c(U_t, \LL) , \qquad a_t\delta_t [t,x] := [a_t , t , x],\qquad x\in X_t,
\end{equation}
and we may view $(C_c(U_t, \LL), \|\cdot\|_{\infty})$ as a subspace of $(\mathfrak{C}_c(\G,\LL), \|\cdot\|_{\max}^S)$.

\begin{lem}\label{lem:grading_of_transformation_groupoid_algebra}  
 $\mathfrak{C}_c(\G,\LL)$ is spanned by elements $a_t\delta_t$, $a_t\in C_c(X_t)$, $t\in S$, and
\begin{enumerate}
    \item\label{enu:grading_of_transformation_groupoid_algebra2} $a_s \delta_s \cdot  a_t\delta_t
    = a_s (a_t\circ h_{s^*}) u(s,t) \delta_{st}$ and $(a_t\delta_t)^*=\overline{a_{t}\circ h_{t}} \cdot \overline{u(t,t^*)}\delta_{t^*}$;
    \item\label{enu:grading_of_transformation_groupoid_algebra3} $s\leq t$ implies $X_{s}\subseteq X_{t}$ and $a \delta_s = a \overline{u(ss^*,t)}\delta_t$ for any $a \in C_c(X_{s})$.
\end{enumerate}
\end{lem}
\begin{proof} 
We only explain \ref{enu:grading_of_transformation_groupoid_algebra3} as the rest is straightforward. 
Let $s\leq t$. 
Then $X_{s}\subseteq X_{t}$ by the composition law (Remark~\ref{rem:twisted_relations}) and so also $U_{s} \subseteq U_{t}$. 
For $a \in C_c(X_{s})$ we have $a \overline{u(ss^*,t)}\in C_c(X_{s})\subseteq C_c(X_{t})$ and so  $a\delta_s \in C_c(U_s,\LL)$ and $a \overline{u(ss^*,t)}\delta_t \in C_c(U_t,\LL)$.
Take $x\in X_{t^*}$. 
If $x\in X_{s^*}$ then $[t,x]= [s,x]$ and
$ 
    a \overline{u(ss^*,t)} \delta_t [t,x] = [a\overline{u(ss^*,t)},t,x] = [a,s,x] = a \delta_s  [t,x]
$
by the equivalence relations defining $\G$ and $\LL$. 
If $x\notin X_{s^*}$ then $a \delta_s [t,x] = 0$ by convention, and $a \overline{u(ss^*,t)}\delta_t [t,x] = [a\overline{u(ss^*,t)},t,x]=[0,t,x]=0$ because $a(h_t(x))=0$. 
Hence $a \delta_s  = a \overline{u(ss^*,t)}\delta_t$.
\end{proof}

\begin{prop}\label{prop:integration_of_rep}
Every covariant representation $(\pi,v)$ of $(\alpha,u)$ in a Banach algebra $B$ integrates to 
a representation $\pi\rtimes v : F^{\overline{S}}(\G,\LL) \to B $ such that
\[
    \pi\rtimes v (a_t\delta_t) = \pi(a_t)v_t,\qquad  a_t\in C_c(X_{t}),\ t\in S.
\]
If $B$ is a  $C^*$-algebra then $\pi\rtimes v$ is $*$-preserving.
If $(\pi,\tilde{v})$ is a $B'$-normalization (or $B_*$-normalization) of $(\pi,v)$ as in Proposition~\ref{prop:normalized_twisted_rep} (or Proposition~\ref{prop:B_*_normalized_twisted_rep}) then $\pi\rtimes v = \pi\rtimes \tilde{v}$.
\end{prop}
\begin{proof}
We claim that the formula $\pi\rtimes v(\sum_{t\in F} a_t\delta_t) = \sum_{t\in F}\pi(a_t)v_t $, where $F\subseteq S$ is  finite, yields a well-defined map $\pi\rtimes v : \mathfrak{C}_c(\G,\LL) \to B$, or equivalently that $\sum_{t\in F} a_t\delta_t = 0$ implies $\sum_{t\in F} \pi(a_t) v_t = 0$. 
This is proved in \cite[Lemma 8.4]{Exel} in the untwisted case, under the assumption that $\G$ is second countable, using measure theoretical methods. 
We prove it in general using topological tools and induction on the cardinality of $F$. 
Suppose that $\sum_{t\in F} a_t\delta_t = 0$. 
If $|F|=1$ then $\sum_{t\in F}\pi(a_t)v_t = 0$ (because $a_t\delta_t=0$ if and only if $a_t=0$). 
Now assume the claim is true for all sets with cardinality smaller than that of $F$. 
Pick any $t_0\in F$ and put $F_0 = F\setminus \{t_0\}$.
Then $\sum_{t\in F_0} a_t\delta_t = -a_{t_0}\delta_{t_0}$, so the closed support of $a_{t_0}\delta_{t_0}$, which we denote by $K$, is a compact set covered by $\{U_{t}\cap U_{t_0}\}_{t\in F_0}$. 
Since $\overline{S}$ is wide, $U_{t}\cap U_{t_0} = \bigcup_{s\leq t, t_0} U_{s}$.
So for each $t\in F_0$ we may find $s_{1,t},...,s_{n_t,t}\leq t,t_0$ such that $K\subseteq \bigcup_{t\in F} \bigcup_{i=1}^{n_t} U_{s_{i,t}}$.
Let $\{ \varrho_{s_{i,t}} \}_{i=1...,n_t ,t\in F_0}$ be a partition of unity on $K$ subordinate to this open cover (which exists as everything happens in the Hausdorff space $U_{t_0}$).
Write $\widetilde{\varrho}_{s_{i,t}}:=\varrho_{s_{i,t}}\circ r|_{U_{s_{i,t}}}\in C_c(X_{s_{i,t}})\subseteq  C_c(X_{t}\cap X_{t_0})$ for each $i$ and $t$.
Then by Lemma~\ref{lem:grading_of_transformation_groupoid_algebra}\ref{enu:grading_of_transformation_groupoid_algebra3} and Lemma~\ref{lem:range_of_covariant_rep}\ref{enu:range_of_covariant_rep2} (applied to each $t$, $2 n_t$-times) we get 
\[ 
    a_{t_0}\delta_{t_0} = \sum_{t\in F_0} \left(\sum_{i=1}^{n_t} \widetilde{\varrho}_{s_{i,t}} a_{t_0} \frac{u(s_{i,t} s_{i,t}^*,t_0)}{u(s_{i,t} s_{i,t}^*,t) } \right) \delta_t, \qquad 
    \pi(a_{t_0}) v_{t_0} = \sum_{t\in F_0} \pi \left( \sum_{i=1}^{n_t} \widetilde{\varrho}_{s_{i,t}} a_{t_0} \frac{u(s_{i,t} s_{i,t}^*,t_0)}{u(s_{i,t} s_{i,t}^*,t) } \right) v_t .
\]
Thus $\sum_{t\in F} a_t\delta_t = \sum_{t\in F_0} b_t \delta_t$ and $\sum_{t\in F} \pi(a_t)v_t = \sum_{t\in F_0} \pi(b_t)v_t$ for some $b_t\in C_c(X_t)$, $t\in F_0$. 
Hence the claim follows by the inductive hypothesis.
Now Lemma~\ref{lem:range_of_covariant_rep}\ref{enu:range_of_covariant_rep1} and Lemma~\ref{lem:grading_of_transformation_groupoid_algebra}\ref{enu:grading_of_transformation_groupoid_algebra2} readily imply that $\pi\rtimes v:\mathfrak{C}_c(\G,\mathcal{L}) \to B$ is an algebra homomorphism, which is $*$-preserving if $B$ is a $C^*$-algebra.
It is $\|\cdot\|_{\max}^{S}$ contractive because $\| \pi\rtimes v(a_t\delta_t) \| = \| \pi(a_t)v_t \| \leq \|a_t\|_{\infty}=\|a_t\delta_t\|_{\infty}$ for every $a_t\delta_t\in C_c(U_t, \LL)$, $t\in S$.
Hence it extends to a representation $\pi\rtimes v : F^{\overline{S}}(\G,\LL) \to B$, which is $*$-preserving if $B$ is a $C^*$-algebra.
The last statement is clear. 
\end{proof}
The following lemma will allow us to show the converse to Proposition~\ref{prop:integration_of_rep}. 

\begin{lem}\label{lem:C*_unit_rep}
Let $\psi : A \to E$ be a contractive linear map from a $C^*$-algebra into a Banach space $E$ which has a predual Banach space $E_*$. 
There is a contractive element $\psi(1_{A})\in E$ such that for every  approximate unit $\{\mu_i\}_{i}\subseteq A$ we have $\psi(1_A) = E_*\mhyphen\lim_{i} \psi(\mu_i)\in E$.
\end{lem}
\begin{proof}
For simplicity we assume $\F = \C$ (the real case can be reduced to the complex case by complexification). 
We identify $E_*$ with a subspace of $E'$. 
Let $f\in E_*$. 
Then $\psi' (f) = f\circ\psi : A\to \C$ is a bounded functional. 
Hence it decomposes to $\psi'(f) = \sum_{k=0}^3 i^{k}\tau_k$ where $\tau_k : A\to \C$, $k=0,...,3$, are positive functionals.
Applying the GNS-construction to each $\tau_k$, we get representations $\pi_{k} : A\to B(H_k)$ and cyclic vectors $\omega_k\in H_k$ such that $\psi'(f)(a) = \sum_{k=0}^3 i^{k} \langle \pi_k(a) \omega_k , \omega_k \rangle$.
Since $\{\pi_k(\mu_i)\}_{i}$ is strongly convergent to the identity on  $H_k$ we conclude that  $\{ f(\psi(\mu_i)) \}_{i}$ is convergent in $\C$ to the number $c_{f,\psi} := \sum_{k=0}^3 i^{k}\|\omega_k\|^2 = \sum_{k=0}^3 i^{k} \|\tau_k\|$ that depends only on $\psi$ and $f$ (it does not depend on $\{\mu_i\}_{i}$). 
Now, since the net $\{\psi(\mu_i)\}_{i}$ is bounded in $E$, the Banach--Alaoglu Theorem says that there is a subnet $\{\psi(\mu_{i_j})\}_{j}$ with an $E_*$-limit $\psi(1_A) := E_*\mhyphen\lim_{j} \psi(\mu_{i_j}) \in E$.  
So in particular $f(\psi(1_A))=\lim_{j} f(\psi(\mu_{i_j})) = c_{f,\psi}$ for every $f\in E_*$. 
This implies that every net $\{\psi(\mu_i)\}_{i}$, where $\{\mu_i\}_{i}$ is an approximate unit, is $E_*$-convergent to $\psi(1_A)$. 
\end{proof}

\begin{thm}\label{thm:disintegration}
Let $(\G,\LL)$ be a twisted \'etale groupoid.  Let $(\alpha,u)$ be any twisted inverse semigroup action of $S$ that models $(\G,\LL)$, and denote by   $\overline{S}$ be the unitization of the image of $S$ in $S(\L)$ (one can always take  a wide unital inverse subsemigroup $S=\overline{S}\subseteq S(\L)$). 
Then $\psi=\pi\rtimes v$ yields a bijective correspondence between representations $\psi : F^{\overline{S}}(\G,\LL) \to B$ in a Banach algebra $B$ and
\begin{enumerate} 
    \item \label{enu:disintegration1} $B'$-normalized covariant representations $(\pi,v)$ of $(\alpha,u)$ in $B$;
    \item \label{enu:disintegration2} $B_*$-normalized covariant representations $(\pi,v)$ of $(\alpha,u)$ in $B$, if $(B, B_*)$ is  a dual Banach algebra;
    \item \label{enu:disintegration3} covariant representations $(\pi,v)$ of $(\alpha,u)$ on $E$, if $B=B(E)$ and $E$ is a reflexive Banach space.
\end{enumerate}
If each $X_t$, $t\in S$, is compact  then the pairs $(\pi,v)$ in \ref{enu:disintegration1}--\ref{enu:disintegration3} coincide with covariant representations $(\pi,v)$ of $(\alpha,u)$ in $B$ such that $v_t = \pi(1_{X_t})v_{t}\in B$ for all $t\in S$.
\end{thm}
\begin{proof}
Let $\psi : F^{\overline{S}}(\G,\LL) \to B$ be a representation. 
Then $\pi := \psi|_{C_0(X)}$ is a representation of $C_0(X)$. 
For every $t\in S$, the map $C_c(X_{t})\ni a_t \mapsto \psi(a_t\delta_t) \in B$ extends to a linear contraction $C_0(X_{t}) \to B\subseteq B''$. 
Let $v_t\in B''$ be the  element associated to this map in Lemma~\ref{lem:C*_unit_rep}, where $E=B''$ and $E_*=B'$. 
Namely, for any approximate unit $\{\mu_i^t\}_{i} \subseteq C_0(X_t)$ we have  $v_t = B'\mhyphen\lim_{i} \psi(\mu_i^t\delta_t)$. 
Recall that we consider multiplication in $B''$ which is $B'$-continuous in the second variable.
For $a\in C_c(X_t) = C_c(X_{tt^*})$ we have $(a \delta_{tt^*}) * (\mu_i^t \delta_t) \to a\delta_t$ in $C_0(U_t,\LL)$. 
Thus
\[
    \pi(a)v_t = \psi( a\delta_{tt^*} ) B'\mhyphen\lim_{i} \psi (\mu_i^t \delta_t) = B'\mhyphen\lim_{i} \psi (a\delta_{tt^*} * \mu_i^t \delta_t) = \psi(a\delta_t) .
\]
For $a\in C_c(X_{t^*})$ we have $(\mu_i^t \delta_t) * (a \delta_{t^*t}) = (\alpha_t(a) \mu_i^t) \delta_t \to  \alpha_t(a) \delta_{t}$ in $C_0(U_{t^*},\LL)$. 
Multiplication in $B''$ is $B'$-continuous in the first variable if the second variable is in $B$, so
\[
    v_t \pi(a) = B'\mhyphen\lim_{i} \psi (\mu_i^t \delta_t * a \delta_{t^*t}) = \psi \big( \alpha_t(a) \delta_{t} \big) = \pi(\alpha_t(a)) v_t .
\]
If in addition $t = e \in E(S)$ then $v_e \pi(a) = \psi(\alpha_e(a) \delta_{e}) = \psi(a \delta_{e}) = \pi(a)$.
If $a\in C_c(X_{st})$ for $s,t\in S$ then $\alpha_{s}^{-1}(a) = a\circ h_{s} \in C_c(X_{s^*}\cap X_t)$ and therefore $(a\delta_s) * (\mu_i^t \delta_t) = a (\mu_i^t \circ h_{s^*}) u(s,t) \delta_{st} = \alpha_{s}(\alpha_{s}^{-1}(a) \mu_i^t) u(s,t)\delta_{st} \to a u(s,t) \delta_{t}$ in $C_0(U_{st},\LL)$. 
Thus
\[
    \pi(a)v_s v_t = B'\mhyphen\lim_{i}\psi(a \delta) \psi (\mu_i^t \delta_t) = \psi \big( au(s,t) \delta_{t} \big) = \pi \big( a u(s,t) \big) v_t. 
\]
By construction $v_t=B'\mhyphen\lim_{i} \p(\mu_i^t) v_t)$. 
Hence we see that $(\pi, v)$ is a $B'$-normalized covariant representation of $(\alpha,u)$ with $\psi = \pi\rtimes v$. 

In view of Proposition \ref{prop:integration_of_rep}, this gives \ref{enu:disintegration1} as a $B'$-normalized a covariant representation $(\pi, v)$ of $(\alpha,u)$ with $\psi=\pi\rtimes v$ has to be the one we constructed above.
Similarly, we get \ref{enu:disintegration2} and \ref{enu:disintegration3} by passing to an appropriate normalization as described in Proposition~\ref{prop:B_*_normalized_twisted_rep} and Corollary~\ref{cor:normalization_to_spatial_cov_rep}.
The last part of the assertion follows from Remark~\ref{rem:unital_actions}.
\end{proof}

\begin{cor}\label{cor:disintegration} 
Under the notation of Theorem~\ref{thm:disintegration}, we have a natural isometric isomorphism $F^{\overline{S}}(\G,\LL)\cong C_0(X)\rtimes_{(\alpha,u)} S$.  
Moreover, if $\EE$ is a class of Banach spaces then denoting by $\EE_{\alg}$ the class of all covariant representations of $(\alpha,u)$ in Banach algebras $B(E)$, where $E\in \EE$, and  by $\EE_{\spa}$ 
the class of all covariant representations of $(\alpha,u)$ on Banach spaces $E\in\EE$, then
\[
    F_{\EE}^{\overline{S}}(\G,\LL) \cong C_0(X) \rtimes_{(\alpha,u),\EE_{\alg}} S \donto C_0(X)\rtimes_{(\alpha,u),\EE_{\spa}}S .
\]
The last homomorphism is isometric, for instance, if all spaces in $\EE$ are reflexive, or if each  $X_t$, $t\in S$, is compact.
\end{cor}
\begin{proof}
Apply Theorem~\ref{thm:disintegration}.
\end{proof}

\begin{cor}\label{cor:representations into C*-algebras}
For any  unital inverse subsemigroup $S\subseteq S(\L)$ covering $\G$ and  any $C^*$-algebra $B$, a homomorphism $F^S(\G,\LL)\to B$ is contractive (i.e.\ is a representation) if and only if it is $*$-preserving, and then it factors through a  $*$-homomorphism
$C^*(\G,\LL)\to B$.
\end{cor}
\begin{proof} By Remark \ref{rem:various_norms_and^projectivness} we may assume $S$ is wide and so Theorem \ref{thm:disintegration} applies.
Any representation $\psi : F^S(\G,\LL) \to B$ is $*$-preserving because any integrated representation is, by Proposition~\ref{prop:integration_of_rep}.
Conversely, if $\psi$ is a $*$-homomorphism it is $\|\cdot\|_{C^*\max}$-contractive on $\mathfrak{C}_c(\G, \LL)$ and thus it defines a representation $\overline{\psi}:C^*(\G,\LL)\to B$.
Since $\psi$ is the composition of the canonical representation $F^S(\G,\LL)\to C^*(\G,\LL)$ and $\overline{\psi}$, it is a representation itself.
\end{proof}

\subsection{Borel extension-disintegration}
\label{sec:Borel_disintegration}

We obtain yet another disintegration theorem by extending representations to a certain algebra of Borel sections, using weak integrals and the extension of functionals via the Riesz--Markov--Kakutani representation theorem. 
This works for representations in the algebra $B(E)$ where $E$ is a Banach space with  a predual (but is not necessarily reflexive).

Let us fix a twisted \'etale groupoid $(\G,\LL)$ and a unital  inverse subsemigroup $S\subseteq \Bis(\G)$ covering $\G$. 
 For any open set $U\subseteq \G$ we denote by 
$\B(U,\LL)$ the linear space of  bounded Borel sections of  $\LL|_{U}$. We view it as a subspace of $\B(\G,\LL)$ in an obvious way. Then the space
$$
\B^S(\G,\LL):=\spane \{ f \in \B(U,\LL): U\in  S \}
$$
is  a $*$-algebra with the  convolution and involution given by standard formulas \eqref{eq:convolution_and_involution}.
In particular, we have $\B(U,\LL)\cdot \B(V,\LL)=\B(UV,\LL)$ and $\B(U,\LL)^*=\B(U^*,\LL)$ for all $U,V\in S$. 
For any  open set $U\subseteq \G$  denote by $\B_c(U,\LL)$ the set of $f\in \B(U,\LL)$ whose strict support $\supp(f):=\{\gamma \in U: f(\gamma)\neq 0\}$  is contained in a compact subset of $U$.

\begin{lem}
$ \B_c(\G,\LL) = \spane \{ f \in \B_c(U,\LL): U\in  S \}$ for any $S\subseteq \Bis(\G)$ covering $\G$. 
\end{lem}
\begin{proof}
Let $f\in \B_c(\G,\LL)$ and pick a compact set $K\subseteq \G$  containing $\supp(f)$. 
By considering a family of bisections that are compactly contained in elements of $S$, we may cover $K$ by $\{U_i\}_{i=1}^n\in S$ which has a  refinement $\{V_i\}_{i=1}^n\in \Bis(\G)$  
such that  for each $i$ the closure of $V_i$ in $U_i$ is compact, for all $i$. 
Then the closure $K_i$ of each $K\cap V_i$ in $U_i$ is compact.  Thus $f_i:=f|_{V_{i}\setminus\bigcup_{l=1}^{i-1}V_i}\in \B_c(U_i,\LL)$, $i=1,\dots,n$,
and $f=\sum_{i=1}^n f_i$. 
\end{proof}
The above lemma implies that  $\B_c(\G,\LL)$ is a $*$-subalgebra of $\B^S(\G,\LL)$.
 In the untwisted (usually second countable) case the algebra $\B_c(\G)$  appears in a number of sources, see for instance in
\cite[page 42]{Paterson}, \cite{Gardella_Lupini17},  \cite{BussMartinez}, \cite{BGHL}. Here  we find $\B^S(\G,\LL)$  more useful, as it contains the indicator functions $1_{U}$, $U\in S$.

From now on we assume that  $S\subseteq S(\LL)\subseteq \Bis(\G)$. Without changing  $\B^S(\G,\LL)$
 we may also assume that $S$ is closed under inclusions, and so it is wide. 
Let us fix unitary sections $c_U\in C_u(U,\LL)$, $U\in S$, and denote by $(\alpha,u)$ 
the associated twisted action of $S$ on $C_0(X)$, see Subsection \ref{subsec:twisted_groupoids_vs_twisted_actions}. 
Thus $t\in S$ is a bisection that we will denote also by $U_t$, when we view it as a set. 
Since $\alpha$ is given by the topological action $h : S\to \PHomeo(X)$ it  extends to an action $\tilde{\alpha}$ on $\B(X)$ in the obvious way.
As $u(s, t)\in C_u(X_{st})\subseteq \B(X_{st})$, $s,t\in S$, we may view $(\tilde{\alpha},u)$ as a twisted inverse semigroup action on the the $C^*$-algebra $\B(X)$.
Adapting arguments from Subsection \ref{subsec:integration}, one sees that  covariant representations of $(\tilde{\alpha},u)$ integrate to representations of  $\B^S(\G,\LL)$.
More specifically, we extend \eqref{eq:coeffcients_isomorphisms}, to a linear  isomorphism
 \begin{equation}\label{eq:coeffcients_isomorphisms2}
    \B(X_{t})\ni b_t\mapsto b_t\delta_t\in \B(U_t, \LL) , \quad\text{where}\quad 
		b_t\delta_t (\gamma):= 	
		\begin{cases}
			b_t(r(\gamma))c_{U_t}(\gamma),
		& \gamma\in U_t,		\\ 	0&  \gamma\not\in U_t. 		\end{cases}
\end{equation}
The Borel analogue of Lemma \ref{lem:grading_of_transformation_groupoid_algebra} holds, with the same formulas but with  $a_t,a\in\B(X_t)$.
Similarly, we have the following Borel version of Proposition \ref{prop:integration_of_rep}.

\begin{prop}\label{prop:Borel_integration_of_rep} 
Consider the extended action $(\widetilde{\alpha},u)$ of $S$ on $\B(X)$ described above. 
Let $B$ be a Banach algebra.
The relations $\widetilde{\psi} (a_t\delta_t) = \pi(a_t)v_t$, $a_t\in \B(X_{t})$, $t\in S$, establish a bijective correspondence 
between covariant representations $(\widetilde{\pi},v)$ of $(\widetilde{\alpha},u)$ in  $B$ satisfying  $v_t = \pi(1_{X_t})v_{t}\in B$ for all $t\in S$,
and  homomorphisms $\widetilde{\psi}:\B^S(\G,\LL) \to B$ 
such that $\|\widetilde{\psi}(a_t\delta_t)\|\leq \|a_t\|_{\infty}$ for all $a_t\in \B(X)$, $t\in S$.
\end{prop}
\begin{proof} 
If $\widetilde{\psi}:\B^S(\G,\LL) \to B $  is a homomorphism with  $\|\widetilde{\psi}(a_t\delta_t)\|\leq \|a_t\|_{\infty}$ for  $a_t\in \B(X)$, $t\in S$, 
then  $\widetilde{\pi}:=\widetilde{\psi}|_{\B(X)}$  is a representation of $\B(X)$ in $B$, and each $v_{t}:=\widetilde{\psi}(1_{X_t}\delta_t)=\widetilde{\psi}(1_{U_t})$ is a contractive operator. 
The pair $(\widetilde{\pi},v)$ is the desired covariant representation  of $(\widetilde{\alpha},u)$ (the  calculations in the proof of Theorem~\ref{thm:disintegration} work and simplify 
as we replace approximate units by characteristic functions).
Conversely, if  $(\widetilde{\pi},v)$ is a covariant representation  of $(\widetilde{\alpha},u)$, then as in the proof of Proposition \ref{prop:integration_of_rep}, one sees that  $\widetilde{\psi}(\sum_{t\in F} a_t\delta_t) = \sum_{t\in F}\widetilde{\pi}(a_t)v_t $,  where $a_t\in \B(X_t)$, $t\in S$, and $F\subseteq S$ is  finite, yields the desired homomorphism $\widetilde{\psi} :\B^S(\G,\LL) \to B$. 
Here the proof of the inductive step in the proof of Proposition~\ref{prop:integration_of_rep} simplifies as we do not need to use continuous partitions of unity.
Namely,  assume that $\sum_{t\in F} a_t\delta_t = 0$. 
Then picking any $t_0\in F$, putting $F_0 := F\setminus \{t_0\}$ and denoting by $s_t$ the element corresponding to $U_{t}\cap U_{t_0}$, the formulas in Lemma~\ref{lem:grading_of_transformation_groupoid_algebra}\ref{enu:grading_of_transformation_groupoid_algebra3} and Lemma~\ref{lem:range_of_covariant_rep}\ref{enu:range_of_covariant_rep2} imply that 
$\sum_{t\in F} a_t\delta_t =\sum_{t\in F_0} (a_t+ 1_{X_{s_t}}a_{t_0}\frac{u(s_{t} s_{t}^*,t_0)}{u(s_{t} s_{t}^*,t) })\delta_t$ and 
$\sum_{t\in F} \widetilde{\pi}(a_t)v_t=\sum_{t\in F_0}\widetilde{\pi}(a_t+ 1_{X_{s_t}}a_{t_0}\frac{u(s_{t} s_{t}^*,t_0)}{u(s_{t} s_{t}^*,t) })v_t$ where 
$a_t+ 1_{X_{s_t}}a_{t_0}\frac{u(s_{t} s_{t}^*,t_0)}{u(s_{t} s_{t}^*,t) }\in \B(X_{t})$, $t\in F_0$. 
\end{proof}

The following proposition is inspired by \cite[Lemma 4.9]{BussMartinez} and \cite[Lemma 6.7]{BGHL}, that we generalize in Corollary~\ref{cor:Borel_for_C*-algebras} below. 
The main idea goes back at least to \cite[II.1.17]{Renault_book}, see also the proof of \cite[Theorem 3.1.1]{Paterson}. 

\begin{prop}\label{prop:Borel_extensions_of_represenations}
Let $E$ be a Banach space with a predual $E_*$. 
Let $S\subseteq S(\LL)\subseteq \Bis(\G)$ be a unital inverse semigroup that covers $\G$.
Every $\|\cdot\|_{\max}^S$-contractive homomorphism $\psi : \mathfrak{C}_c(\G,\LL) \to B(E)$ extends to a  homomorphism $\widetilde{\psi}: \B^S(\G,\LL) \to B(E)$, which is determined by 
the formula 
\begin{equation}\label{eq:extended representation}
    \langle \xi, \widetilde{\psi}(b\delta_{U})\eta\rangle = \int_{r(U)} b \,d\mu^U_{\xi,\eta}, \qquad b\in \B(r(U)),\ U\in S,\ \xi\in E_{*},\ \eta\in E,
\end{equation}
where $\mu^U_{\xi,\eta}$ is the unique (complex or signed) Radon measure corresponding to the functional 
$C_c(r(U))\ni a \mapsto \langle \xi,\psi(a\delta_{U})\rangle$. 
\end{prop}
\begin{proof}
Without loss of generality we may assume that $S$ is closed under taking open subsets and so it is a wide inverse semigroup.
We denote by $(\alpha,u)$ the associated twisted inverse semigroup action of $S$. 
For every $\xi\in E'$, $\eta\in E$, $t\in S$, we denote by $\mu_{\xi,\eta}^t$ the unique  Radon measure such that $\int_{X_t} a \, d\mu_{\xi,\eta}^t=\langle \xi, \psi(a\delta_t)\eta\rangle$ for all $a\in C_c(X_t)$ (which is given by the Riesz--Markov--Kakutani representation theorem). 
Thus we have a bounded linear functional $\psi^t_{\xi,\eta}:\B(U_t,\LL)\to \F$ where $\psi^t_{\xi,\eta}(b\delta_t):=\int_{X_t} b \, d\mu_{\xi,\eta}^t$,
for $b\in \B(X_t)$, and $\|\psi^t_{\xi,\eta}\|\leq  \|\xi\|\cdot \|\eta\|$. 
In fact for a fixed $t\in S$ and $b\in \B(X_t)$, the map $E'\times E\ni (\xi,\eta)\mapsto \psi^t_{\xi,\eta}(b\delta_t)\in \F$ is a bounded bilinear form (for every Borel $U\subseteq X_t$ the map $E'\times E\ni (\xi,\eta)\mapsto \mu_{\xi,\eta}^t(U)$ is bilinear) with norm not exceeding $\|b\|_{\infty}$. 
Restricting this form to $E_*\times E\subseteq E'\times E$ we conclude that for each $b\in \B(X_t)$ there is an operator $\psi_{t}(b \delta_t)\in B(E)=B(E_*')$ with $\|\psi_{t}(b\delta_t)\|\leq \|b\|_{\infty}$ given by
$$
    \langle \xi, \psi_{t}(b\delta_t)\eta\rangle =\psi^t_{\xi,\eta}(b\delta_t)=\int_{X_t} b \, d\mu_{\xi,\eta}^t,\qquad t\in S,\ \xi\in E_{*},\ \eta\in E. 
$$
This is essentially \eqref{eq:extended representation}, written in different notation, and clearly 
$\psi_{t}(b\delta_t)=\psi(b\delta_t)$ when $b\in C_c(X_t)$.
Now take $s,t\in S$. 
We claim that $\psi_{s}(b_s\delta_s)\psi_{t}(b_t\delta_t)=\psi_{st}(b_s\delta_s \cdot b_t\delta_{t})$ for all $b_s\in \B(X_s)$, $b_t\in \B(X_t)$.
Recall that $b_s \delta_s \cdot b_t\delta_t = b_s (b_t\circ h_{s^*}) u(s,t) \delta_{st}$.
Assume first that $b_s\in C_c(X_s)$, $b_t\in C_c(X_t)$.		
Then, for every $(\xi,\eta)\in E'\times E$,
$$
    \int_{X_s} b_s \, d\mu_{\xi,\psi(b_t\delta_t)\eta}^s=\langle \xi, \psi(b_s\delta_s)\psi(b_t\delta_t)\eta\rangle=
        \langle \xi, \psi(b_s\delta_s \cdot b_t\delta_t)\eta\rangle=\int_{X_{st}} b_s (b_t\circ h_{s^*}) u(s,t) \, d\mu_{\xi,\eta}^{st}.
$$
By uniqueness of the Radon measure corresponding to a linear functional, this implies that $d\mu_{\xi,\psi(b_t\delta_t)\eta}^s=(b_t\circ h_{s^*}) u(s,t) \, d\mu_{\xi,\eta}^{st}$ as Borel measures supported on $X_{st}\subseteq X_{s}$. 
Thus fixing $b_t\in C_c(X_t)$ we conclude that $\psi_{s}(b_s\delta_t)(b_t\delta_t)=\psi_{st}(b_s\delta_s \cdot b_t\delta_{t})$  for every $b_s\in \B(X_s)$. 
Hence for any $b_s\in \B(X_s)$ and $b_t\in C_c(X_t)$ we get
\[
\begin{split}
    \int_{X_t} b_t \, d\mu_{\psi_s(b_s\delta_s)'\xi,\eta}^t&=\langle \psi_s(b_s\delta_s)'\xi, \psi(b_t\delta_t)\eta\rangle
        =\langle \xi, \psi_s(b_s\delta_s)\psi(b_t\delta_t)\eta\rangle=\langle \xi,\psi_{st}(b_s\delta_s \cdot b_t\delta_{t})\eta\rangle \\
        &=\int_{X_{st}} b_s (b_t\circ h_{s^*}) u(s,t) \, d\mu_{\xi,\eta}^{st}=\int_{X_{t}} b_t  (b_s u(s,t))\circ h_{s}  \, d\mu_{\xi,\eta}^{st}\circ h_{s}.
\end{split}
\]
Therefore, similarly to above, $d\mu_{\psi_s(b_s\delta_s)'\xi,\eta}^t=(b_s u(s,t))\circ h_{s}  \, d\mu_{\xi,\eta}^{st}\circ h_{s}$. 
Using this for for all $b_s\in \B(X_s)$, $b_t\in \B(X_t)$ we get
\[
\begin{split}
    \langle \xi, \psi_{s}(b_s\delta_s)\psi_{t}(b_t\delta_t)\eta\rangle &=\int_{X_t} b_t \, d\mu_{\psi_s(b_s\delta_s)'\xi,\eta}^t
        =\int_{X_{t}} b_t  (b_s u(s,t))\circ h_{s}  \, d\mu_{\xi,\eta}^{st}\circ h_{s} \\
        &=\int_{X_{st}} b_s (b_t\circ h_{s^*}) u(s,t) \, d\mu_{\xi,\eta}^{st}=\langle \xi,\psi_{st}(b_s\delta_s \cdot b_t\delta_{t})\eta\rangle.
\end{split}
\]
Thus $\psi_{s}(b_s\delta_s)\psi_{t}(b_t\delta_t)=\psi_{st}(b_s\delta_s \cdot b_t\delta_{t})$ as claimed.

The above identity now  implies that the pair $(\widetilde{\pi},v)$, where  $\widetilde{\pi}:=\psi_e$ and  $v_t:=\psi_t(1_{X_t})$, for $t\in S$, is a covariant representation of the associated  action $(\widetilde{\alpha},u)$ of $S$ on  $\B(X)$. 
By Proposition~\ref{prop:Borel_integration_of_rep}, $(\widetilde{\pi},v)$ yields a homomorphism $\widetilde{\psi}:\B^S(\G,\LL) \to B(E)$, which  is the desired extension. 
\end{proof}

\begin{thm}\label{thm:Borel_extension_of_representations}
Let $(\G,\LL)$ be a twisted \'etale groupoid.
Let $S\subseteq S(\LL)\subseteq \Bis(\G)$ be a unital wide inverse semigroup, and  let $(\alpha,u)$  and $(\widetilde{\alpha},u)$ be the associated twisted actions of $S$ on $C_0(X)$ and $\B(X)$, respectively. 
We have natural isometric $*$-isomorphisms and an inclusion
$$
    C_0(X)\rtimes_{(\alpha,u)} S \cong F^S(\G,\LL) \subseteq \overline{\B^S(\G,\LL)}^{\| \cdot \|_{\max}^{S}}\cong\B(X)\rtimes_{(\widetilde{\alpha},u)} S
$$
where $\| \cdot \|_{\max}^{S}$ is the largest submultiplicative norm on $\B^S(\G,\LL)$ which agrees with $\|\cdot\|_{\infty}$ on each slice $\B(U,\LL)$, $U\in S$ (it  agrees with the norm of $F^S(\G,\LL)$). 
Moreover, every representation $\psi : F^{S}(\G,\LL) \to B(E)$ where $E$ has a predual,  is  the restriction of  $\widetilde{\pi}\rtimes v$ for some  covariant representation $(\widetilde{\pi},v)$ of $(\widetilde{\alpha},u)$ on $E$. 
\end{thm}
\begin{proof}
We have $C_0(X)\rtimes_{(\alpha,u)} S \cong F^S(\G,\LL)$ by Corollary~\ref{cor:disintegration}. 
Let $\| \cdot \|_{\B}^{S}$ be the largest submultiplicative norm on $\B^S(\G,\LL)$ which does not exceed $\|\cdot\|_{\infty}$ on each slice $\B(U,\LL)$, $U\in S$.
Proposition~\ref{prop:Borel_integration_of_rep} implies the natural isometric isomorphism $\overline{\B^S(\G,\LL)}^{\| \cdot \|_{\B}^{S}}\cong\B(X)\rtimes_{(\widetilde{\alpha},u)} S$. 
Clearly, $\|f\|_{\B}^{S}\leq \|f \|_{\max}^{S}$ for $f\in \mathfrak{C}_c(\G,\LL)$. On the other hand if we view $E=F^S(\G,\LL)''$ as a Banach space with a predual, then multiplication yields a $\|\cdot\|_{\max}^S$-isometric homorphism $\psi:\mathfrak{C}_c(\G,\LL)\to B(E)$ which by Proposition~\ref{prop:Borel_extensions_of_represenations} extends to a $\| \cdot \|_{\B}^{S}$-contractive homomorphism  $\widetilde{\psi}:\B^S(\G,\LL) \to B(E)$. 
Hence the norms $\|\cdot\|_{\max}^{S}$ and  $\|\cdot\|_{\B}^{S}$ coincide on $\mathfrak{C}_c(\G,\LL)\subseteq \B^S(\G,\LL)$.
Thus without confusion we may denote $\|\cdot\|_{\B}^{S}$ by $\|\cdot\|_{\max}^{S}$. 
This gives the inclusion $F^S(\G,\LL)=\overline{\mathfrak{C}_c(\G,\LL)}^{\| \cdot \|_{\max}^{S}}\subseteq   \overline{\B^S(\G,\LL)}^{\| \cdot \|_{\max}^{S}}$.  
The last part of the assertion is now basically Proposition~\ref{prop:Borel_extensions_of_represenations}, modulo Proposition~\ref{prop:Borel_integration_of_rep}. 
\end{proof}

\begin{rem} 
It follows that the representation $\psi : F^{S}(\G,\LL) \to B(E)$ in the last part of the above theorem is of the form $\pi\rtimes v$ for a covariant representation $(\pi,v)$ of $(\alpha,u)$ in $B(E)$ with $v$ taking values in $B(E)$ (rather than in $B(E)''$). 
However, unless $E$ is reflexive (as then Thereom~\ref{thm:disintegration}\ref{enu:disintegration3} applies), it is not clear whether we can make $(\pi,v)$ a covariant representation on $E$. 
The problem is we do not know whether \ref{item:covariant_representation2'} in Definition~\ref{defn:covariant_representation_on_space} holds.
\end{rem}

\begin{cor}\label{cor:Borel_for_C*-algebras} 
Let $S\subseteq S(\LL)\subseteq \Bis(\G)$ be a unital wide inverse semigroup. 
The inclusion $\mathfrak{C}_c(\G,\LL)\subseteq \B^S(\G,\LL)$ induces the embedding of $C^*$-algebras $C^*(\G,\LL)\subseteq \overline{\B^S(\G,\LL)}^{ \|\cdot\|_{C^*}}$ where $\|\cdot\|_{C^*}$ is the largest $C^*$-norm on the $*$-algebra $\B^S(\G,\LL)$. 
Moreover, every representation of $C^*(\G,\LL)$ on a Hilbert space $H$ extends to a representation of $\overline{\B^S(\G,\LL)}^{ \|\cdot\|_{C^*}}$ on $H$.
\end{cor}
\begin{proof} 
If $\psi:\mathfrak{C}_c(\G,\LL)\to B(H)$ is a $*$-homomorphism  it is $\|\cdot\|_{\max}^S$-contractive and so by Theorem~\ref{thm:Borel_extension_of_representations}, 
it extends to $\|\cdot\|_{\max}^S$-contractive homomorphism $\widetilde{\psi}:\B^S(\G,\LL)\to B(H)$, which by Corollary~\ref{cor:representations into C*-algebras} has to be a $*$-homomorphism.
This implies that the maximal $C^*$-norm on $\B^S(\G,\LL)$ agrees with the maximal $C^*$-norm on $\mathfrak{C}_c(\G,\LL)$.
Hence the inclusion $C^*(\G,\LL)\subseteq \overline{\B^S(\G,\LL)}^{ \|\cdot\|_{C^*}}$. Now our argument also implies the second part of the assertion.
\end{proof}


%
%
\section{Representations on $L^p$-spaces}
\label{sec:Groupoid Lp-operator algebras}

Fix $p\in [1,\infty]$ and let $q\in [1,\infty]$ be the number satisfying $1/p + 1/q=1$, with the convention that $1/\infty=0$. 
Take a twisted \'etale groupoid $(\G,\LL)$ with a locally compact Hausdorff unit space $X$. 
Regular representations of $\G$ on $\ell^p$-spaces (with $p<\infty$) were considered in \cite{cgt}, and of $(\G,\LL)$ (but using a different picture of the twist) in \cite{Hetland_Ortega}. 
We spell out some details using our notation. 
We denote by $\ell^p(\G,\LL)$ the Banach space of all sections of $\LL$ for which the norm $\|f\|_{p} = (\sum_{\gamma\in \G} |f(\gamma)|^{p})^{1/p}$ when $p<\infty$ and
$\|f\|_{\infty} = \sup_{\gamma\in \G}|f(\gamma)|$ when $p=\infty$, is finite.
We have an isometric isomorphism $\ell^p(\G,\LL)\cong \ell^p(\G)$. 

\begin{prop}\label{prop:regular_disintegrated} 
For each $p\in [1,\infty]$ the formula
\[
    \Lambda_p(f) \xi (\gamma) := (f * \xi)(\gamma) = \sum_{r(\eta) = r(\gamma)} f(\eta)\cdot \xi(\eta^{-1} \gamma), \qquad f \in \mathfrak{C}_c(\G,\LL),\ \xi \in \ell^{p}(\G,\LL) ,
\]
defines a representation $\Lambda_p : F(\G,\LL) \to B(\ell^{p}(\G,\LL))$, which is injective on $\mathfrak{C}_c(\G,\LL)$, and  
\begin{enumerate}
    \item\label{enu:regular_disintegrated1} for $p \in (1 , \infty)$ we have $\| \Lambda_p(f) \| \leq \|f\|_{*d}^{1/p} \cdot \|f\|_{*r}^{1/q} \leq \|f\|_{I}$, while $\|\Lambda_1(f)\| = \|f\|_{*d}$ and $\|\Lambda_{\infty}(f)\|= \|f\|_{*r}$, for all $f \in \mathfrak{C}_c(\G , \LL)$; 
    \item\label{enu:regular_disintegrated2} putting $F^p_{\red}(\G,\LL) := \overline{\Lambda_p(\mathfrak{C}_c(\G,\LL))}$, we have $F^p_{\red}(\G,\LL)^{\op} \cong F^{q}_{\red}(\G,\LL^{\op})$, with the isometric isomorphism induced by the map $f \mapsto \widecheck{f}$. 
    In particular, $F^p_{\red}(\G)^{\op}\cong F^{q}_{\red}(\G)$.
\end{enumerate}
\end{prop}
\begin{proof} 
It is routine to check that the formula in the assertion gives a well defined homomorphism $\Lambda_p : \mathfrak{C}_c(\G,\LL) \to B(\ell^{p}(\G,\LL))$, and $\|\Lambda_p(f)\| \leq \|f\|_{\infty}$ for any $f\in C_c(U,\LL)$, $U\in \Bis(\G)$. 
Hence $\Lambda_p$ is $\|\cdot\|_{\max}$-contractive and thus it uniquely extends to a representation of $F(\G,\LL)$, see also Remark \ref{rem:not_important_remark} below. 
For each  $\gamma \in \G$ choose a norm one element $1_{\gamma}\in L_\gamma$ and treat it as a section of $\LL$ which is zero at $\eta \neq \gamma$.
Then $\{1_\gamma\}_{\gamma\in \G}$ is a Schauder basis for $\ell^{p}(\G,\LL)$ and for any $f\in  \mathfrak{C}_c(\G)$ we have 
\[
    \| \Lambda_p(f) 1_{\gamma} \|_p = \begin{cases} 
        \left(\sum_{d(\eta) = d(\gamma)} |f(\eta \gamma^{-1})|^p\right)^{1/p}, &  \text{ if $p<\infty$} \\
        \sup_{d(\eta)=d(\gamma)}|f(\eta \gamma^{-1})|, & \text{ if $p=\infty$}.
    \end{cases}
\]
This implies that $\Lambda_p$ is injective on $\mathfrak{C}_c(\G,\LL)$. 
Also using this one sees that  $\|\Lambda_1(f)\| = \|f\|_{*d}$ and $\|\Lambda_{\infty}(f)\| = \|f\|_{*r}$.
Then the inequalities $\|\Lambda_p(f)\|\leq \|f\|_{*d}^{1/p} \cdot \|f\|_{*r}^{1/q} \leq \|f\|_{I}$ follow from the Riesz--Thorin interpolation theorem.
This proves \ref{enu:regular_disintegrated1}. 
To see \ref{enu:regular_disintegrated2} it suffices to note that under the standard isomorphism $\ell^{p}(\G,\LL)' \cong \ell^{q}(\G,\LL^{\op})$ given by the pairing $\langle \xi, \eta\rangle :=\sum_{\gamma\in \G}  \xi(\gamma) \eta(\gamma)$, for $\xi\in \ell^{p}(\G,\LL)$, $ \eta\in \ell^{q}(\G,\LL^{\op})$, we have  $\Lambda^{p}(f)' = \Lambda^{q}(\widecheck{f})$ for all $f \in \mathfrak{C}_c(\G)$ 
(note that $\xi(\gamma) \eta(\gamma)\in L_{r(\gamma)}=\F$ because the bundle $\LL$ is trivial over $X$). 
\end{proof}

\begin{rem}\label{rem:not_important_remark}
Let $\tilde{h}:\Bis(\G)\to \PHomeo(\G)$ be the extension of the canonical action described in Example~\ref{ex:actions_from_groupoids}. 
The formulas 
\[
    \pi(a)\xi(\gamma) := a(r(\gamma))\xi(\gamma) , \qquad 
    v_U\xi(\gamma) :=   \begin{cases}
                            \xi(\tilde{h}_{U^*}(\gamma)), &  \gamma\in r^{-1}(r(U)), \\
                            0 , & \text{ otherwise}, 
                        \end{cases}
\]
where $a\in C_0(X)$, $\xi \in \ell^p(\G)$, $\gamma \in \G$, $U\in \Bis(\G)$, define a covariant representation $(\pi, v)$ of the corresponding inverse semigroup action on $C_0(X)$ on $\ell^p(\G)$, and $\Lambda_p = \pi\rtimes v : F(\G)\to B(\ell^p(\G))$. 
More generally, if $(\alpha, u)$ is a twisted inverse semigroup action of $S$ on $C_0(X)$ and $(\G,\LL)$ is the associated twisted groupoid, then the formulas
\[
    \pi(a) \xi[s,x] = a \big( h_s(x) \big) \xi([s,x]), \qquad 
    v_t\xi[s,x] =   \begin{cases}
                        u(t,t^*s) \big( h_{s}(x) \big) \xi \big( [t^*s,x] \big), &  x\in h_{s}^{-1}(X_{t}), \\
                        0 , & \text{otherwise}, 
                    \end{cases}
\]
where $a\in C_0(X)$, $\xi \in \ell^p(\G)$, $x\in X_{s^*}$, $t,s \in S$, define a covariant representation of $(\alpha,u)$ such that $\pi\rtimes v : F(\G,\LL) \to B(\ell^p(\G))$ is equivalent to $\Lambda_p : F(\G,\LL) \to B(\ell^p(\G,\LL))$.
\end{rem}

\begin{defn}
We call $F^{p}_{\red}(\G,\LL)$ the \emph{reduced $L^p$-operator algebra of $(\G,\LL)$}.
\end{defn} 

\begin{rem}
This definition is consistent with those in \cite{cgt}, \cite{Gardella_Lupini17} and \cite{Hetland_Ortega} (where it is assumed that $p<\infty$). 
Moreover, $F^{2}_{\red}(\G,\LL) = C_{\red}^*(\G,\LL)$ is the reduced $C^*$-algebra of $(\G , \LL)$, and we have $F^{1}_{\red}(\G,\LL) \cong F_{d_*}(\G,\LL)$ and $F^{\infty}_{\red}(\G,\LL)\cong F_{r_*}(\G,\LL)$ by Proposition~\ref{prop:regular_disintegrated}\ref{enu:regular_disintegrated1}.
\end{rem}

We now use our main results to prove a general theorem for representations of groupoid Banach algebras in $L^p$-operator algebras.

\begin{thm}\label{thm:L_p_norm_estimates} Let $(\G,\LL)$ be a twisted \'etale groupoid and $p,q\in [1,\infty]$ with $1/p+1/q=1$. Let $\psi:F^{S}(\G,\LL)\to B(L^p(\mu))$ be a representation for some measure $\mu$ and some unital inverse semigroup $S\subseteq \Bis(\G)$  covering $\G$.
If $\F=\C$, or $p=2$, or if $\psi(C_0(X)^+)$ consists of positive operators, then 
\[
    \| \psi(f)\| \leq \| f \|_{d_*}^{1/p} \, \| f \|_{r_*}^{1/q} \leq \|f\|_{I} , \qquad f\in  \mathfrak{C}_c(\G,\LL).
\]
If $p=\infty$, the estimate $\| \psi(f)\| \leq  \| f \|_{r_*}$  holds also when $L^\infty(\mu)$ is replaced by $C_0(\Omega)$ for a locally compact Hausdorff space $\Omega$.
\end{thm}

\begin{rem} 
We do not know if the above assertion holds for $\F=\R$ without any further restrictions. 
The reason is that representations as in Example~\ref{ex:real_representation_from_bicontractive} may occur.
\end{rem}
\begin{rem} When $\F=\C$, Theorem \ref{thm:L_p_norm_estimates} implies that   every covariant representation $(\pi,v)$ of  $(\alpha,u)$ 
on an $L^p$-space integrates to a $\|\cdot\|_{I}$-contractive representation $\pi\rtimes v$   of $\mathfrak{C}_c(\G,\LL)$, cf. 
Proposition \ref{prop:integration_of_rep}.
In particular, this shows that the contractive homomorphism in \cite[Proposition 6.5]{cgt} is isometric, solving an issue formulated therein. 

\end{rem}
\begin{cor}\label{cor:Cstar_boundedness}
Every $*$-homomorphism  $\psi:\mathfrak{C}_c(\G,\LL)\to B$ into a $C^*$-algebra is automatically $\|\cdot\|_{I}$-contractive.
In fact,  $\| \psi(f)\| \leq \| f \|_{d_*}^{1/2} \, \| f \|_{r_*}^{1/2} \leq \|f\|_{I}$, for  $f\in  \mathfrak{C}_c(\G,\LL)$.
\end{cor}
\begin{proof}
By Corollary \ref{cor:representations into C*-algebras}, 
 $\psi$ extends to a representation $\psi:F^{S}(\G,\LL)\to B$, for any unital inverse semigroup $S\subseteq \Bis(\G)$  covering $\G$. Hence the assertion follows by Theorem  \ref{thm:L_p_norm_estimates} 
as we may embed $B$ into $B(L^2(\mu))$ for some $\mu$.
\end{proof}
\begin{rem} The above corollary improves upon one of the main results in \cite{Clark_Zimmerman}, which states  $\|\cdot\|_{I}$-boundedness of
the largest $C^*$-norm on $\mathfrak{C}_c(\G)$ where $\G$ is second countable.
When  $\G$ is Hausdorff this is well-known, see \cite[Lemma 3.2.3]{Sims}, \cite[Lemma 4.3]{BHM}. In the non-Hausdorff setting this seems to be new  and much more subtle. In particular, the proof in \cite[Section 6]{Clark_Zimmerman} requires 5 pages of  estimates,  passes through inductive limit topology  and uses Renault's disintegration theorem. We use  Borel extension-disintegration instead. 
\end{rem}

We first note that when $p<\infty$  we may assume that $\psi$ is non-degenerate.

\begin{prop}\label{prop:from_degenerate_to_nondegenerate_Lp}
Let $p\in [1,\infty)$. 
For any representation $\psi:F^{S}(\G,\LL)\to B(L^p(\mu))$ there is an isometric isomorphism $\Phi : \overline{C_0(X)L^p(\mu)} \to L^p(\overline{\mu})$, where $\overline{\mu}$ is localizable, and then $\overline{\psi}(f) = \Phi(\psi(f)|_{\overline{C_0(X)L^p(\mu)}})$, $f\in F^S(\G,\LL)$, defines a non-degenerate representation $\overline{\psi} : F^S(\G,\LL) \to B(L^p(\overline{\mu}))$ such that $\|\overline{\psi}(f)\|=\|\psi(f)\|$ for all $f\in F^S(\G,\LL)$. 
Moreover, if $\F=\C$ or if $\psi|_{C_0(X)}$ is positive we may choose $\overline{\psi}$ so that $\overline{\psi}|_{C_0(X)}$ consists of multiplication operators.
\end{prop}
\begin{proof}
By Lemma~\ref{lem:approximate_units}, $F^{S}(\G,\LL)$ and $C_0(X)$ have a common contractive approximate unit $\{\mu_{i}\}_{i}$. 
In particular, $\overline{\psi(F^{S}(\G,\LL))L^p(\mu)}=\overline{\psi(C_0(X))L^p(\mu)}$ and as $L^p(\mu)$ does not contain an isomorphic copy of $c_0$ (see \cite[2.6]{Gardella_Thiel1}) there is a norm one projection $P : L^p(\mu) \to \overline{\psi(C_0(X))L^p(\mu)}$ by \cite[Corollary 3.10]{Gardella_Thiel1}. 
In fact $P$ is a strong limit of $\{\psi(\mu_{i})\}_{i}$, and so it is positive if  $\psi|_{C_0(X)}$ is.
Thus by \cite[Theorem 6]{Tzafriri} there is an isometric isomorphism $\Phi : \overline{\psi(C_0(X))L^p(\mu)} \to L^p(\overline{\mu})$ for some measure $\overline{\mu}$, which can be chosen to be localizable, and if $P$ is positive then we can choose the  isomorphism $\Phi$ to be positive, cf.\ \cite[Corollary 2]{Ando}.  
Also for all $f\in F^{S}(\G,\LL)$ we have $\|\psi(f)\|=\|\psi(f)P\|$ because $\|\psi(f) \xi\| = \lim_{i} \|\psi(f) \psi(\mu_{i}) \xi \| = \lim_{i} \| \psi(f) P \psi(\mu_{i}) \xi \| \leq \|\psi(f)P \| \|\xi\|$. 
This implies that $\overline{\psi}$ is a non-degenerate representation, and $\overline{\psi}|_{C_0(X)}$ is positive if $\psi|_{C_0(X)}$ is. 
Hence for $p\neq 2$ the last part of the assertion follows from  Theorem~\ref{thm:L^p_representations of C_0(X)}.
For $p=2$  it follows from  Lemma~\ref{lem:Hilbert_space_representations of C_0(X)}.
\end{proof}

\begin{rem}
Proposition~\ref{prop:from_degenerate_to_nondegenerate_Lp} fails when $p=\infty$ and $X$ is not compact, see Corollary~\ref{cor:degeneracy_of_L_infty_representations}.
We do not know whether it extends to $C_0$ or Lindenstrauss spaces.
\end{rem}

\begin{proof}[Proof of Theorem \ref{thm:L_p_norm_estimates}]
We may assume that $S$ is a wide inverse subsemigroup  of $S(\LL)$. Indeed, by \eqref{eq:majorization_by_S(L)_semigroup} we have a canonical representation $F^{\widetilde{S}\cap S(\LL)}(\G,\LL)\donto F^{S}(\G,\LL)$. Thus composing $\psi$ with this representation, we may replace $S$ with  $\widetilde{S}\cap S(\LL)$. Then there is  a twisted action $(\alpha, u)$ of $S$ that models $(\G,\LL)$, see Lemma~\ref{lem:twisted_groupoids_come_from_twisted_actions}. Thus we have $\G = S\ltimes_{h} X$ where $h : S \to \PHomeo(X)$ is an inverse semigroup action and $\alpha_t : C_0(X_{t^*}) \to C_0(X_{t})$;\ $\alpha_t(a) = a\circ h_{t^*}$,
 for $a\in C_0(X_{t^*})$, $t\in S$, and   $F^{S}(\G,\LL)\cong C_0(X)\rtimes_{(\alpha,u)} S$. 
When $p=1$ we may  assume that $L^1(\mu)$ has a predual as we may pass to the representation $\psi'':F^{S}(\G,\LL)\to B(L^1(\mu)'')$ given by $\psi''(a):=\psi(a)''$, $a\in F^{S}(\G,\LL)$.
By the duality between abstract $M$-spaces and abstract $L$-spaces, see \cite[pages 25, 135]{Lacey}, we get that $L^1(\mu)''$ is isomorphic as a Banach lattice  to  $L^1(\overline{\mu})$ for some localizable measure $\overline{\mu}$. 
Thus we may replace $\psi$ with  $\psi'':F^{S}(\G,\LL)\to B(L^1(\overline{\mu}))$  (then $\psi''(C_0(X)^+)$ consists of positive operators,  if $\psi(C_0(X)^+)$ does). 
Hence for any $p\in [1,\infty]$, $L^p(\mu)$ has a predual and so Theorem~\ref{thm:Borel_extension_of_representations} applies. Accordingly, $\psi$ extends to a representation 
\[
    \widetilde{\psi}=\widetilde{\pi}\rtimes v:\overline{\B^S(\G,\LL)}^{\| \cdot \|_{\max}^{S}}\cong\B(X)\rtimes_{(\widetilde{\alpha},u)} S \to B(L^p(\mu))
\]
 where $(\widetilde{\pi},v)$ is a covariant representation of the extended action $(\widetilde{\alpha},u)$ on $L^p(\mu)$.
Let us now assume that $p<\infty$. 
Then by Proposition~\ref{prop:from_degenerate_to_nondegenerate_Lp}  we may  assume that  $\widetilde{\pi}$ is given by multiplication operators. 
Hence we have a representation $\pi_0:\B(X)\to L^{\infty}(\mu)$ such that 
\[
    \widetilde{\pi}(a)\xi= \pi_0(a)\cdot \xi, \quad \text{ for  all } \xi \in L^{p}(\mu),\ a\in \B(X). 
\]
Since $\pi_0$ is a $*$-homomorphism between $C^*$-algebras, the continuous functional calculus applies  (also in the real case \cite[1.1.5]{Schroder}), and in particular we have $|\pi_0(a)|^{\alpha}=\pi_0(|a|^{\alpha})$ for any $\alpha>0$. 
Using this for $q\in [1,+\infty]$ with $1/p+1/q=1$, and any $\xi \in L^{p}(\mu)$, we get 
\begin{equation}\label{eq:spectral_calc_consequences}
    | \pi_0(a)|=  \pi_0(|a|^{\frac{1}{p}})\pi_{0}(|a|^{\frac{1}{q}}), \qquad  |\pi_0(a)\cdot \xi|^p=\pi_0(|a|^p) \cdot|\xi|^p.
\end{equation}
Let $f = \sum_{t\in F} f_t \delta_t \in\B^S(\G,\LL)$, where $F \subseteq S$ is finite and $f_t\in \B(X_t)$, see \eqref{eq:coeffcients_isomorphisms2}.
Since we work with Borel sections we may disjointify them, that is we may assume that the strict supports of  $f_t \delta_t$'s are pairwise disjoint subsets of $\G$
(here extension to Borel sections is crucial, especially in the non-Hausdorff case).
Then for each $x\in X$ we get  
$\sum_{t\in F} |f_t(x)| =\sum_{t\in F} | (f_t\delta_t)\circ r|_{U_t}^{-1}(x)|=\sum_{r(\gamma)=x} |f(\gamma)|
$
and  $\sum_{t\in F} |f_t (h_{t} (x))|=\sum_{t\in F} | (f_t\delta_t)\circ d|_{U_t}^{-1}|=\sum_{d(\gamma)=x} |f(\gamma)|$.
In particular, 
\begin{equation}\label{eq:epasilon_estimates_for_Hahn}
\| f \|_{r_*}=  \|\sum_{t\in F} |f_t|  \|_\infty  \quad \text{ and } \quad
\quad  \| f \|_{d_*}=  \|\sum_{t\in F} |f_t| \circ h_{t} \|_\infty.
\end{equation} 
Let $\xi \in L^{p}(\mu)$ and $\eta \in L^{q}(\mu)$. 
By covariance of $(\widetilde{\pi},v)$, condition  \ref{item:covariant_representation1}, and contractiveness of $v_t$,  for any $a\in \B(X_{t})$ we have
$\int |\pi(a)v_t \xi |^p d\mu = \int | v_t \pi (a\circ h_{t}) \xi |^p d\mu \leq \int | \pi(a\circ h_{t}) \xi |^p d\mu$, so
\begin{equation}\label{eq:spectral_calc_covariant}
\int |\pi_0(a)  v_t \xi |^p d\mu  \leq \int | \pi_0(a\circ h_{t}) \xi |^p d\mu.
\end{equation} 
Applying H\"older's inequality to the measure $\sum_{t\in F} \mu$, when $p>1$   ($q<\infty$), we get
\[\begin{split}
    \left| \int \big( \pi \rtimes v(f) \xi \big) \eta \, d\mu \right|& \leq \sum_{t\in F} \int  \big| (\pi_0 (f_t ) (v_t \xi) \eta \big| \, d\mu
	\stackrel{\eqref{eq:spectral_calc_consequences}}{=}\sum_{t\in F} \int  \big| \pi_0 ( |f_t|^{1/p} ) (v_t \xi) \big| \, \big| \pi_0 ( |f_t|^{1/q} ) \eta \big| \, d\mu 
		\\
	&\stackrel{\text{H\"older}}{\leq} \left( \sum_{t\in F} \int \big| \pi_0 ( |f_t|^{1/p} ) v_t \xi \big|^p \, d\mu \right)^{1/p}  \left( \sum_{t\in F} \int \big| \pi_0 ( |f_t|^{1/q} ) \eta \big|^{q} \, d\mu \right)^{1/q} 
\\
&\stackrel{\eqref{eq:spectral_calc_covariant}}{\leq} \left( \sum_{t\in F} \int \big| \pi_0 ( |f_t|^{1/p}\circ h_{t}  )  \xi \big|^p \, d\mu \right)^{1/p}  \left( \sum_{t\in F} \int \big| \pi_0 ( |f_t|^{1/q} ) \eta \big|^{q} \, d\mu \right)^{1/q}  
\\
&\stackrel{\eqref{eq:spectral_calc_consequences}}{=} \left( \int \pi_0 \left( \sum_{t\in F} |f_t| \circ h_{t} \right) |\xi|^p \, d\mu \right)^{1/p}  \left( \int \pi_0\left( \sum_{t\in F} |f_t| \right) |\eta|^{q} \, d\mu \right)^{1/q} 
\\
          &\stackrel{\|\pi_0\|\leq 1}{\leq} \left\| \sum_{t\in F} |f_t| \circ h_{t} \right\|_{\infty}^{1/p} \, \| \xi \|_{p} \, \left\| \sum_{t\in F} |f_t| \right\|_{\infty}^{1/q} \, \|\eta\|_{q}.
\end{split}
\]
Hence  $\|\pi \rtimes v(f) \| \leq 
\| f \|_{d_*}^{1/p} \, \| f \|_{r_*}^{1/q}$ by \eqref{eq:epasilon_estimates_for_Hahn}. For $p=1$ the proof is even simpler, as then
$$
    \left| \int  \pi \rtimes v(f) \xi  \, d\mu \right| \leq \sum_{t\in F} \int  \big| v_t \pi (  |f_t\circ h_t| )  \xi\big| \, d\mu
        \leq  \sum_{t\in F}  \int   \pi (  |f_t\circ h_t| )  \big|\xi\big| \, d\mu  \leq   \left\| \sum_{t\in F} |f_t| \circ h_{t} \right\|_{\infty} \|\xi\|_{1}.
$$
Thus $\| \pi\rtimes v(f) \| \leq \|f\|_{d_*}$ again by \eqref{eq:epasilon_estimates_for_Hahn}.

Now let us consider a representation $\psi:F^{S}(\G,\LL)\to B(C_0(\Omega))$ (and assume $\psi|_{C_0(X)}$ is positive when $\F=\R$). 
This in particular covers the case where $\psi$ is a representation on $L^{\infty}(\mu)$. 
The dual space $E'$ of $E:=C_0(\Omega)$ is  isomorphic as a Banach lattice to an $L^1$-space $L^{1}(\mu)$.
Thus the formula $\tilde{\psi}(f) := \psi(\widecheck{f})'$,  $f\in \mathfrak{C}_c(\G, \LL^{\op})^{\op}$, defines a representation $\tilde{\psi} :F^S(\G, \LL^{\op})^{\op} \to B(E')\cong B(L^{1}(\mu) )$ (with $\tilde{\psi}|_{C_0(X)}$ positive when $\psi|_{C_0(X)}$ is positive). Therefore by the previous step 
  $\|\psi(f)\|=\|\psi(f)'\|=\|\tilde{\psi}(\widecheck{f})\|\leq  \| \widecheck{f}\|_{*d}=  \| f\|_{*r}$.
\end{proof}

\subsection{Full $L^p$-groupoid algebras and spatial covariant representations}

The following definition is motivated by Theorem \ref{thm:L_p_norm_estimates} and the desire to have a consistent theory in both complex and real case.
\begin{defn}\label{defn:full_Lp}  
Let $p \in [1,\infty]$.
The \emph{full $L^p$-operator algebra} is the completion $F^p(\G,\LL) = \overline{\mathfrak{C}_c(\G,\LL)}^{\|\cdot\|_{L^p} }$ in the norm $\|f\|_{L^p} := \sup_{\psi\in\RR} \| \psi(f) \|$, where $\RR$ consists of homomorphisms $\psi:\mathfrak{C}_c(\G,\LL)\to B(E)$ such that: 
\begin{enumerate}[label={(R\arabic*)}]
    \item\label{enu:full_Lp1} $E=L^p(\mu)$ for some measure $\mu$;
    \item\label{enu:full_Lp2} $\psi$ is $\|\cdot\|_{\max}^{S}$-contractive for some unital inverse subsemigroup $S\subseteq \Bis(\G)$ covering $\G$;
    \item\label{enu:full_Lp3}  operators in $\psi(C_c(X)^+)$ preserve the cone of positive functions in $E$.
\end{enumerate}
\end{defn}

\begin{thm}\label{thm:initial_on_L^p_full}
Let $p,q\in [1,\infty]$ with $1/p+1/q=1$.
\begin{enumerate}

	\item\label{enu:initial_on_L^p_full1} We have $F^1(\G,\LL)=F^1_{\red}(\G,\LL) = F_{d_*}(\G,\LL)$ and $F^{\infty}(\G,\LL) = F^{\infty}_{\red}(\G,\LL) = F_{r_*}(\G,\LL)$.	

 	\item\label{enu:initial_on_L^p_full5} $\|f\|_{L_p} \leq \| f \|_{L^1}^{1/p} \, \| f \|_{L^\infty}^{1/q} \leq \|f\|_{I}$ for  $f\in  \mathfrak{C}_c(\G,\LL)$, and so condition \ref{enu:full_Lp2} in Definition~\ref{defn:full_Lp} can be replaced by $\|\cdot\|_{I}$-contractiveness of $\psi$ (which alone is stronger than \ref{enu:full_Lp2}).

	\item\label{enu:initial_on_L^p_full2}  $\|\cdot \|_{L^{2}} = \|\cdot\|_{C^*\max}$ is the largest $C^*$-norm on $\mathfrak{C}_c(\G,\LL)$, so $F^{2}(\G,\LL) = C^*(\G,\LL)$.

		\item\label{enu:initial_on_L^p_full4} If $\F=\C$ or $p=2$, condition \ref{enu:full_Lp3} in Definition \ref{defn:full_Lp} can be dropped, so $F^p(\G,\LL)$ is universal for all 
		 homomorphisms from $\mathfrak{C}_c(\G,\LL)$  into $L^p$-operator algebras, which satisfy \ref{enu:full_Lp2} (equivalently are $\|\cdot\|_I$-contractive).
\item \label{enu:initial_on_L^p_full4.5} If $p<\infty$, the definition of $\|\cdot\|_{L^p}$  is not affected if one  restricts to non-degenerate homomorphisms or if
one replaces \ref{enu:full_Lp2}  with ``$\psi(C_c(X))$ consists of multiplication operators''. 
%


    \item\label{enu:initial_on_L^p_full7} If $p=\infty$ condition \ref{enu:full_Lp1} in Definition~\ref{defn:full_Lp} can be replaced by ``$E = C_0(\Omega)$ for some locally compact Hausdorff space'', and if $\F = \C$ we can replace \ref{enu:full_Lp1} by ``$E$ is a Lindenstrauss space'' (and drop condition \ref{enu:full_Lp3}). 
  
    \item\label{enu:initial_on_L^p_full3} $F^p(\G,\LL)^{\op}\cong F^{q}(\G,\LL^{\op})$ and   $F^p(\G,\LL)\anti F^{q}(\G,\LL)$ where these isometric maps are induced by the involutions 
		$\widecheck{\,}$ 
		and $*$ on $\mathfrak{C}_c(\G,\LL)$, respectively.
\end{enumerate}
\end{thm}
\begin{proof} 
By  Theorem~\ref{thm:L_p_norm_estimates} we have $\|f\|_{L^p}\leq \| f \|_{d_*}^{1/p} \, \| f \|_{r_*}^{1/q} \leq \|f\|_{I}$
for 
$f\in  \mathfrak{C}_c(\G,\LL)$. Applying this to $p=1$ and $p=\infty$ we get $\|f\|_{L^1}\leq \| f \|_{d_*}$ and $\|f\|_{L^{\infty}}\leq \| f \|_{d_*}$. 
On the other hand, by Proposition~\ref{prop:regular_disintegrated}\ref{enu:regular_disintegrated1} we have $\|f\|_{*d}=\|\Lambda_1(f)\|\leq \|f\|_{L^1} $ and $\|f\|_{*r}=\|\Lambda_{\infty}(f)\|\leq  \|f\|_{L^{\infty}}$. This shows \ref{enu:initial_on_L^p_full1}. As a consequence the estimates from Theorem~\ref{thm:L_p_norm_estimates} translate to \ref{enu:initial_on_L^p_full5}. 
For $p=2$,  a homomorphism $\psi:\mathfrak{C}_c(\G,\LL)\to B(L^2(\mu))$ is $\|\cdot\|_{\max}^{S}$-contractive if and only if it is a $*$-homomorphism, see Corollary~\ref{cor:representations into C*-algebras}. Also by Lemma~\ref{lem:Hilbert_space_representations of C_0(X)} we may assume that for any such homomorphism  $\psi(C_c(X))$ acts by multiplication operators, and so operators in $\psi(C_c(X)^+)$ are automatically positive.  This implies \ref{enu:initial_on_L^p_full2} (and  also \ref{enu:initial_on_L^p_full4}  for $p=2$). 
 Proposition \ref{prop:from_degenerate_to_nondegenerate_Lp} implies 
 \ref{enu:initial_on_L^p_full4.5} and   \ref{enu:initial_on_L^p_full4} for $p<\infty$.
Statement \ref{enu:initial_on_L^p_full4} for $p=\infty$ follows from the case $p=1$ and Corollary~\ref{cor:L_infty_lindenstrauss_spaces_etc}.  Now
\ref{enu:initial_on_L^p_full7} follows from the last part of Theorem~\ref{thm:L_p_norm_estimates}, and for $\F=\C$ from Corollary~\ref{cor:L_infty_lindenstrauss_spaces_etc} and \ref{enu:initial_on_L^p_full4}. The isomorphisms in \ref{enu:initial_on_L^p_full3} for $p=1,\infty$ and $q=\infty,1$ follow from \ref{enu:initial_on_L^p_full1} and 
Proposition~\ref{prop:regular_disintegrated}\ref{enu:regular_disintegrated2}.
Assuming $p\in (1,\infty)$, $F^p(\G,\LL)^{\op}\cong F^{q}(\G,\LL^{\op})$ holds by Corollary~\ref{cor:reflexive_spaces_oppositie_iso}.  
In fact, if $\F=\R$ we need to modify the proof of Proposition~\ref{prop:duality} by considering only representations where $\psi|_{C_0(X)}$ is positive, but under the standard isomorphism $L^p(\mu)'\cong L^q(\mu)$, an operator $a\in B(L^p(\mu))$ is positive if and only if its dual $a'\in B(L^q(\mu))$ is positive, and so the proof works.
The isomorphism $F^p(\G,\LL)^{\op}\cong F^{q}(\G,\LL^{\op})$ translates to an anti-isomorphism $F^p(\G,\LL)\anti F^{q}(\G,\LL)$, see \cite[Remark 3.2]{Buss_Sims}.
\end{proof}
\begin{rem}\label{rem:second_countable_vs_measures}
If $\G$ is second countable,  Theorem \ref{thm:initial_on_L^p_full} implies that  the complex algebra $F^p(\G)$ is universal for $\|\cdot\|_{I}$-contractive homomorphisms 
from $\mathfrak{C}_c(\G)$ into $B(L^p(\lambda))$ where $\lambda$ is a standard $\sigma$-finite Borel measure, which is consistent with \cite[Definition 6.4]{Gardella_Lupini17}.
Indeed, second countability of $\G$ implies that  $F^p(\G,\LL)$ is separable, and therefore  \cite[Proposition 1.25]{Phillips} shows that   $F^p(\G,\LL)$, for $p\in [1,\infty)$,
is  isometrically represented on a separable $L^p$-space $L^p(\mu)$. We get the same claim for $p=\infty$ by using the  claim for $p=1$ and Theorem \ref{thm:initial_on_L^p_full}\ref{enu:initial_on_L^p_full3}. 
It is well known that separable space $L^p(\mu)$ is isometrically isomorphic to $L^p(\lambda)$ where $\lambda$   is a $\sigma$-finite  Borel measure  on a standard Borel space (and this isomorphism respects indicator functions).
\end{rem}

We may view representations of $F^p(\G,\LL)$ as representations of $F^S(\G,\LL)$ for any unital inverse semigroup $S\subseteq \Bis(\G)$ covering $\G$. 
In particular, we may use our disintegration theorem to treat them as covariant representations for an inverse semigroup  action. 
To this end, until the end of this section 
\begin{center}
\emph{we fix a twisted action $(\alpha, u)$ that models $(\G,\LL)$}
\end{center}
as in Lemma~\ref{lem:twisted_groupoids_come_from_twisted_actions}. 
Thus we have $\G = S\ltimes_{h} X$ where $h : S \to \PHomeo(X)$ is an inverse semigroup action and $\alpha_t : C_0(X_{t^*}) \to C_0(X_{t^*})$;\ $\alpha_t(a) = a\circ h_{t^*}$, for $a\in C_0(X_{t})$, $t\in S$.
We also fix for a while a localizable measure space $(\Omega, \Sigma,\mu)$ and discuss covariant representations of $\alpha$ that are given by data coming from the space $(\Omega, \Sigma,\mu)$. 

\begin{defn}\label{defn:covariant_representation_spatial}
A covariant representation of $(\alpha,u)$ on $L^p(\mu)$, $p\in [1,\infty]$, as in Definition~\ref{defn:covariant_representation_on_space}, is called \emph{spatial} if $\pi : C_0(X) \to B(L^p(\mu))$ acts by multiplication operators on $L^p(\mu)$ and $v : S \to \SPIso(L^p(\mu))$ takes values in the inverse semigroup of spatial partial isometries  (Definition~\ref{def:spatial^partial_isos}).
\end{defn}


\begin{prop} \label{prop:covariant_rep_is_necessarily_spatial}
If $p\in (1,\infty)$ and $(\pi,v)$ is a  covariant representation  of $(\alpha,u)$ on $L^p(\mu)$ where $\pi$ is given by multiplication operators, then 
$(\pi,v)$ is spatial. If $\F=\C$ and $p\in (1,\infty) \setminus \{2\}$ then every non-degenerate covariant representation $(\pi,v)$ of $(\alpha,u)$ on $L^p(\mu)$ is spatial. 
\end{prop}  
\begin{proof} By Lemma \ref{lem:reflexive_covariant_representations}, for $e\in E(S)$ we have $v_e = B_*\mhyphen\lim\pi(\mu_i^e)$, where $\{\mu_{i}^e\}_{i}$ is an approximate unit in $C_0(X_{e})$
and  $B_* := L^q(\mu) \mathbin{\widehat{\otimes}} L^p(\mu)$. This means that  $\int \eta   (v_e\xi)\,d\mu = \lim_{i} \int \eta  \pi(\mu_i^e)\xi \,d\mu$ for all $(\eta,\xi)\in L^q(\mu)\times L^p(\mu)$.
Using that $\pi(\mu_i^e)$ are operators of multiplication, the latter condition readily implies that $v_e$ commutes with projections given by characteristic functions. 
Thus by \cite[Proposition 4.9]{BDEGGMM},  $v_e$ is an $L^p$-projection. Hence by Theorem~\ref{thm:spatial^partial_isometries_description}, the map $v$ takes values in $\SPIso(L^p(\mu))$, and so 
$(\pi,v)$ is spatial.  The second part now follows from Theorem~\ref{thm:L^p_representations of C_0(X)}. 
\end{proof}

In what follows we use the notation of Subsection~\ref{subsect:Partial isometries}.

\begin{defn}\label{defn:covariant_triple_spatial}
A \emph{spatial covariant triple} $(\pi_0,\Phi, \omega)$ for $(\alpha,u)$ and $\mu$ consists of 
a representation $\pi_0 : C_0(X)\to L^\infty(\mu)$, an inverse semigroup $\Phi = \{ [\Phi_s] \}_{s\in S}\subseteq \PAut([\Sigma])$ of partial set automorphisms $\Phi_{s} : \Sigma_{D_{s}} \to \Sigma_{D_{s^*}}$, $D_{s}, D_{s^*}\in \Sigma$, $s\in S$, satisfying $[\Phi_s]\circ [\Phi_t]=[\Phi_{st}]$ for $s,t\in S$,
and a family (cocycle) $\omega = \{\omega_{s}\}_{s\in S}$ of partial unimodal maps $\omega_s\in UL^{\infty}(\mu|_{D_{s^*}})$, $s\in S$. 
These must satisfy: 
\begin{enumerate}[labelindent=40pt,label={(SCT\arabic*)},itemindent=1em]
    \item\label{enu:covariant_representation_spatial1} $T_{\Phi_t}(\pi_0(a)) = \pi_0(\alpha_t(a))$ for all $a\in C_0(X_{t*})$, $t\in S$; 
    \item\label{enu:covariant_representation_spatial2} $\pi_0(\mu_i^e) 1_{D}\nearrow 1_D$, for every  measurable $D\subseteq D_e$ with $\mu(D)<\infty$, every positive  approximate unit $\{\mu_i^e\}_{i} \subseteq C_0(X_e)$ and all $e\in E(S)$;
    \item\label{enu:covariant_representation_spatial3} $\pi_0(a)\omega_{s}T_{\Phi_s}(\omega_t) = \pi_0(au(s,t))\omega_{st}$ $\mu$-almost everywhere, for $s,t \in S$, $a\in C_0(X_{st})$. 
\end{enumerate}
We say that $(\pi_0, \Phi, \omega)$ is \emph{non-degenerate} if $\pi_0(\mu_i) 1_{D}\nearrow 1_D$, for every measurable $D$ with $\mu(D)<\infty$ and  a positive  contractive approximate unit $\{\mu_i\}_{i} \subseteq C_0(X)$. 
\end{defn}

\begin{prop} \label{prop:spatial_cov_reps22222}
For each $p\in [1,\infty)$ we have a bijective correspondence between spatial covariant representations $(\pi, v)$ of $(\alpha , u)$ on $L^{p}(\mu)$ and spatial covariant triples $(\pi_0, \Phi, \omega)$ for $(\alpha , u)$ and $\mu$. This correspondence is given by 
\[ 
    \big( \pi(a)\xi \big) (\omega) = \pi_0(a)(\omega) \xi(\omega), \quad v_{t} := \omega_t  \left( \frac{d\mu \circ \Phi_{t*}}{d\mu|_{D_{t^*}}} \right)^{\frac{1}{p}} T_{\Phi_{t}}, \quad a\in C_0(X),\ \xi \in L^p(\mu),\ t\in S.
\]
The representation $(\pi, v)$ is non-degenerate if and only if the corresponding triple $(\pi_0,\Phi, \omega)$ is non-degenerate.
\end{prop}
\begin{proof}
A representation $\pi : C_0(X)\to B(L^p(\mu))$ acting by multiplication operators is equivalent to
a homomorphism $\pi_0 : C_0(X)\to L^\infty(\mu)$. 
By Proposition~\ref{prop:spatial^partial_isometries}, every map $v : S \to \SPIso(L^p(\mu))$ is given by  $v_{t} = \omega_t (\frac{d\mu\circ\Phi^{*}}{d\mu|_{D_{\Phi^*}}})^{\frac{1}{p}} T_{\Phi_{t}}$, $t\in S$, where $\Phi = \{ [\Phi_s] \}_{s\in S} \subseteq \PAut([\Sigma])$ and $\omega_t\in UL^{\infty}(\mu|_{D_{t^*}})$, $t\in S$. 
The final statement of Proposition~\ref{prop:spatial^partial_isometries} (and uniqueness of the presentation in the Banach--Lamperti theorem, see Proposition~\ref{prop:group_of_spatial_isometries}) tells us that $\pi(a) v_s v_t = \pi (au(s,t)) v_{st}$ holds for all $s,t \in S$, $a\in C_0(X_{st})$ if and only if   \ref{enu:covariant_representation_spatial3} in Definition~\ref{defn:covariant_triple_spatial} holds and $\Phi=\{[\Phi_s]\}_{s\in S}\subseteq \PAut([\Sigma])$ is an inverse semigroup. 
Condition \ref{enu:covariant_representation_spatial2} in Definition~\ref{defn:covariant_triple_spatial} is equivalent to the equality $\overline{\pi(C_0(X_e))L^{p}(\mu)} = L^{p}(\mu|_{D_e})$ for all $e\in E(S)$.    
Assuming this, condition \ref{enu:covariant_representation_spatial1} in Definition~\ref{defn:covariant_triple_spatial} is equivalent to $v_t \pi(a) v_{t^*} = \pi(\alpha_t(a))$  
for all $a\in C_0(X_{t^*})$, $t\in S$. 
This gives the assertion.
\end{proof}

\begin{thm}\label{thm:spatial_representations_of_groupoid_algebras}
Let $(\G,\LL)$ be a twisted groupoid and let $(\alpha,u)$ be any twisted inverse semigroup action that models $(\G,\LL)$.
\begin{enumerate}
    \item \label{enu:spatial_representations_of_groupoid_algebras1} For each $p\in [1,\infty]$ there is an isometric representation $\pi\rtimes v : F^p(\G,\LL)\to B(L^p(\mu))$ where $(\pi,v)$ is a spatial covariant representation  of $(\alpha,u)$ on $L^p(\mu)$. If  $p< \infty$ or when $X$ is compact, then $\pi$ can be chosen to be non-degenerate.
    \item \label{enu:spatial_representations_of_groupoid_algebras2} For $p\in (1,\infty)$ every representation of $F^p(\G,\LL)$ on $L^p(\mu)$ where $C_0(X)$ acts by multiplication operators, and $\mu$ is localizable, is of the form $\pi\rtimes v$ for a spatial covariant representation $(\pi,v)$ of $(\alpha,u)$. This gives  a bijective correspondence between such representations of $F^p(\G,\LL)$ on $L^p(\mu)$ and  spatial covariant triples for $(\alpha,u)$ and $\mu$.
    \item\label{enu:spatial_representations_of_groupoid_algebras3} For any $p, p'\in (1,\infty)$ we have a bijective correspondence between  representations of $F^p(\G,\LL)$ and $F^{p'}(\G,\LL)$ on $L^p$ and $L^{p'}$-spaces such that $C_0(X)$ acts by multiplication operators. This correspondence matches $\pi_p : F^p(\G,\LL)\to B(L^p(\mu))$ and $\pi_{p'} : F^{p'}(\G,\LL) \to B(L^{p'}(\mu))$, given by 
\[
    \pi_p (a_t \delta_t) = \pi_0(a_t) \omega_t \left( \frac{d\mu\circ\Phi_{t^*}}{d\mu|_{D_{\Phi_{t^*}}}} \right)^{\frac{1}{p}} T_{\Phi_{t}}, \qquad  
    \pi_{p'} (a_t\delta_t) = \pi_0(a_t) \omega_t \left( \frac{d\mu\circ\Phi_{t^{*}}}{d\mu|_{D_{\Phi_{t^*}}}} \right)^{\frac{1}{p'}} T_{\Phi_{t}},
\]
    where $a_t\in C_c(X_{t}), t\in S$, and $(\pi_0,\Phi, \omega)$ is a non-degenerate spatial covariant triple for $(\alpha,u)$ and a localizable $\mu$.
\end{enumerate} 
If  all domains $X_{t}$, $t\in S$, are compact, then statements \ref{enu:spatial_representations_of_groupoid_algebras2} and \ref{enu:spatial_representations_of_groupoid_algebras3} hold for all $p,p'\in [1,\infty]$.
\end{thm}
\begin{proof} 
\ref{enu:spatial_representations_of_groupoid_algebras1}. For $p\in \{1,\infty\}$, the assertion follows from Theorem \ref{thm:initial_on_L^p_full}\ref{enu:initial_on_L^p_full1}. 
For $p\in [1,\infty)$, Theorem \ref{thm:initial_on_L^p_full}\ref{enu:initial_on_L^p_full4.5}  implies that for each $a\in F^p(\G,\LL)$ and $k\in \N$
there is a non-degenerate representation $\psi_{a,k}:F^p(\G,\LL)\to B(L^p(\mu_{a,k}))$ such that   $\psi_{a,k}(C_0(X))$ consists of multiplication operators
and $\|a\|_{F^p(\G,\LL)}\leq \|\psi_{a,k}(a)\|+1/k$. Thus the $\ell^p$-direct sum of $\{\psi_{a,k}\}_{a\in F^p(\G,\LL), k\in \N}$ is
 an isometric non-degenerate representation $\psi$ of $F^p(\G,\LL)$ where $\psi(C_0(X))$ consists of multiplication operators.  
Disintegrating $\psi$ and applying Proposition \ref{prop:covariant_rep_is_necessarily_spatial} we get that $\psi=\pi\rtimes v$ where $(\pi,v)$ is a spatial covariant representation  of $(\alpha,u)$.

\ref{enu:spatial_representations_of_groupoid_algebras2}. If $p\in (1,\infty)$, Theorem~\ref{thm:disintegration}\ref{enu:disintegration3} implies that every representation of $F^p(\G,\LL)$ on  $L^p(\mu)$ is of the form $\pi\rtimes v$ for a covariant represenation $(\pi,v)$ on the space $L^p(\mu)$. If  all $X_{t}$, $t\in S$, are compact, the same holds for any  $p\in [1,\infty]$ by the last part of Theorem~\ref{thm:disintegration}.
Assume $\mu$ is localizable.  Then $(\pi,v)$ is necessarily spatial by Proposition~\ref{prop:covariant_rep_is_necessarily_spatial}, and so it is given by a spatial covariant triple by Proposition~\ref{prop:spatial_cov_reps22222}. 
This proves \ref{enu:spatial_representations_of_groupoid_algebras2}.
Item \ref{enu:spatial_representations_of_groupoid_algebras3} now follows readily. 
\end{proof}

\begin{rem} 
If $\F=\C$, then Theorem~\ref{thm:spatial_representations_of_groupoid_algebras} combined with Theorem~\ref{thm:L^p_representations of C_0(X)}
gives a description of all non-degenerate representations of  $F^p(\G,\LL)$ for $p\in (1,\infty)\setminus\{2\}$, and a bijective correspondence between non-degenerate representations of $F^p(\G,\LL)$ and $F^{p'}(\G,\LL)$  for all $p, p'\in (1,\infty)\setminus\{2\}$. 
This is all the more striking, as usually there are no non-zero continuous homomorphisms between $F^p(\G,\LL)$ and $F^{p'}(\G,\LL)$ for $p\neq p'$.
\end{rem}

\section{Banach algebras associated to inverse semigroups} 
\label{subsec:Banach_algebras_inverse_semigroups}
Let $S$ be an inverse semigroup. 
The algebra  $\F S$  of formal finite $\F$-linear combinations of elements in $S$ and  multiplication induced from $S$ is the universal algebra for semigroup homomorphisms of $S$. 
Since  $S$ is more than just a semigroup, we will consider here completions of $\F S$ such that the commutative subalgebra $\F E$, generated by the idempotents $E := E(S)$, is  a uniform Banach algebra.
This boils down to an assumption of \emph{joint contractiveness} that we introduce in Definition \ref{defn:jointly_contractive} below.
When applied to tight representations this will give us a useful description of  Banach algebras associated to ample groupoids.
Some of our results here are already new and might be interesting for $C^*$-algebraists.
 
We start by recalling the groupoid models for $\F S$ and $\F E$ from \cite{Steinberg}. 
Recall that the semigroup $E$ is a semilattice: the `meet' and `preorder' are defined by $e\wedge f := e f$ and $e\leq f$ if and only if $ef = e$.
We write $e<f$ if $e\leq f$ and $e \neq f$. 
If $E$ has a zero, we say $e$ and $f$ are orthogonal if $ef = 0$.
The set $\widehat{E}\subseteq \{0, 1\}^E$ of all non-zero homomorphisms $\phi : E \to \{0,1\}$ is a totally disconnected locally compact Hausdorff space in the product topology of the Cantor space $\{0, 1\}^E$ (which is the topology of pointwise convergence). 
In fact, putting $Z(e) := \{ \phi \in \widehat{E} : \phi(e) = 1 \}$, the sets $Z(e) \setminus \bigcup_{f \in F} Z(f)$, where $e\in E$ and $F\subseteq eE$ is finite, form a basis of compact-open sets for $\widehat{E}$. 
We call $\widehat{E}$ the \emph{spectrum} of $E$. 
A map $\phi : E\to \{0,1\}$ is in $\widehat{E}$ if and only if its support $\supp(\phi) := \{ e\in E : \phi(e)=1\}$ is a \emph{filter} of $E$, i.e. a non-empty, upward-closed and downward directed and subset of $E$. 
There is a natural action $h : S \to \PHomeo(\widehat{E})$ given by 
\[
    h_{t} : Z(t^*t)\to Z(tt^*) ;\ h_{t}(\phi)(e) := \phi ( t^*et ) , \qquad t\in S ,\ e\in E ,\ \phi\in \widehat{E} 
\]
(see \cite[Proposition 10.3]{Exel}). 
The corresponding transformation groupoid $\G(S) := S \ltimes_{h} \widehat{E}$ is ample. 
It is covered by compact open bisections $U_{t} = \{ [t,\phi] : \phi\in Z(t^*t) \}$, $t\in S$. 
The associated Steinberg algebra is
\[
    A(\G(S)) := \spane \{ 1_{U}: \text{$U\in \Bis(\G(S))$ compact open} \} = \spane \{ 1_{U_t} : t\in S \} \subseteq C_c(\G(S)).
\]
By a homomorphism from the semigroup $S$ to an algebra $B$ we mean a semigroup homomorphism $v:S\to B$ into the multiplicative semigroup of $B$.

\begin{prop}[{\cite[Theorem 5.1]{Steinberg}}] \label{prop:Steinberg_algebra_of_inverse_semigroup} 
The map $t \mapsto 1_{U_t}$ determines an isomorphism $\F S\cong A(\G(S))$. 
Thus for any $\F$-algebra $B$, the relation $\psi(1_{U_t}) = v_t$, $t\in S$, establishes a bijective correspondence between algebra homomorphisms $\psi : A(\G(S))\to B$ and homomorphisms $v : S \to B$.
\end{prop}

We have the following commutative subalgebra of $A(\G(S))$:
\[
    A(\widehat{E}) := \spane \{ 1_{U} : \text{$U\subseteq \widehat{E}$ compact open} \} = \spane \{ 1_{Z(e)} : e\in E \} \subseteq C_c(\widehat{E}) \subseteq C_c(\G(S)).
\]
Proposition~\ref{prop:Steinberg_algebra_of_inverse_semigroup} applied to $E$ gives a bijective correspondence between algebra homomorphisms $\pi : A(\widehat{E}) \to B$ and homomorphisms $v : E \to B$. 
We now aim to characterise injectivity of $\pi$ in terms of $v$. 
For finite $F\subseteq E$ we write $\bigwedge{F} := \bigwedge_{e\in F} e= \prod_{e\in F}e$.

\begin{lem}[Inclusion-exclusion principle]\label{lem:inclusion-exclusion}
Let $v:E\to B$ be a homomorphism into an algebra $B$. 
For any finite $F\subseteq E$ and $F_0\subseteq F$ put
\begin{equation}\label{eq:general_projections_to_be_contractive}
    P_{F_0}^F := \prod_{f\in F\setminus F_0} (v_{\bigwedge{F_0}} - v_{\bigwedge{F_0}\wedge f}) = \sum_{F_0\subseteq G\subseteq F} (-1)^{|G\setminus F_0|} v_{\bigwedge{G}} ,
\end{equation}
with the convention $P_{F}^F=v_{\bigwedge{F}}$ and $P_{\emptyset}^F=0$. 
Then $\{ P_{F_0}^{F} \}_{F_0\subseteq F}$ are mutually orthogonal idempotents such that 
$v_e = \sum_{e\in F_0\subseteq F} P_{F_0}^F$, $e\in F$.
\end{lem}
\begin{proof}
This can be proved by induction on $|F|$. 
If $|F| = 1$ the assertion is trivial.
Suppose the assertion holds whenever $|F| < n$, and fix $F$ with $|F| = n$. 
Fix also $e_0\in F$ and let $H=F\setminus\{e_0\}$. 
For any $F_0\subseteq F$ we then have 
\[
    P_{F_0}^F = \begin{cases} 
                    P_{F_0}^H - P_{F_0}^H v_{e_0} , & e_0\not\in F_0, \\
                    P_{F_0 \setminus \{ e_0 \} }^H v_{e_0}, & e_0 \in F_0,\ F_0\neq\{e_0\}, \\
                    v_{e_0} - \sum_{H_0 \subseteq H} P_{H_0}^H , & F_0 \neq \{e_0\}.
                \end{cases}
\]
By the inductive hypothesis $\{ P_{H_0}^{H} \}_{H_0\subseteq H}$ are mutually orthogonal and $v_e = \sum_{e\in H_0\subseteq H} P_{H_0}^H$, $e\in H$. 
Using this and the above relations it is readily checked that $\{ P_{F_0}^{F} \}_{F_0\subseteq F}$ are also mutually orthogonal, and $v_e = \sum_{e\in F_0\subseteq F} P_{F_0}^F$, $e\in F$.
\end{proof}


\begin{cor}\label{cor:injectivity_representation_A(E)} 
Let $\pi : A(\widehat{E}) \to B$ be an algebra homomorphism and let $v : E\to B$ be the corresponding semigroup homomorphism.
Then $\pi$ is injective if and only if $\prod_{f\in F} (v_e - v_f) \neq 0$ for every $e\in E\setminus \{0\}$ and finite $F\subseteq eE\setminus\{e\}$.
\end{cor}
\begin{proof} 
Idempotents \eqref{eq:general_projections_to_be_contractive} corresponding to idempotents $\{ 1_{Z(f)} \}_{f\in F} \subseteq A(E)$ are characteristic functions of basic cylinder sets $Z \big( \bigwedge F_0 \big) \setminus \bigcup_{f\in F\setminus F_0} Z\big( \bigwedge F_0\wedge f \big)$. 
Hence, by Lemma ~\ref{lem:inclusion-exclusion}, every element in $A(E)$ is a finite linear combination of mutually orthogonal idempotents of the form $1_{Z(e) \setminus \bigcup_{f \in F} Z(f)}$, where $e\in E\setminus\{0\}$ and $F\subseteq eE\setminus\{e\}$, which are mapped by $\pi$ to mutually orthogonal idempotents of the form $\prod_{f\in F} (v_e - v_f)$. 
Moreover, $1_{Z(e) \setminus \bigcup_{f \in F} Z(f)}\neq 0$ because $\supp\phi:=\{f\in E: f\geq e\}$ is a filter that defines $\phi$ that is in $Z(e) \setminus \bigcup_{f \in F} Z(f)$. 
Hence $\pi$ is injective if and only if it is non-zero on elements of the form $ 1_{Z(e) \setminus \bigcup_{f \in F} Z(f)}$.
\end{proof}

For representations in Banach algebras we need an analytic condition.

\begin{defn}\label{defn:jointly_contractive}
Let $B$ be a  Banach algebra over $\F$ and let $M\subseteq \F$ be a set. 
We say that a collection of mutually orthogonal idempotents $\{P_i\}_{i\in F}\subseteq B$ is \emph{jointly $M$-contractive} if 
\begin{equation*}\label{eq:jointly_contractive}
    \| \sum_{i\in F} a_i{P_i} \| \leq \max_{i\in F} |a_{i}|, \qquad a_i\in M,\ i \in F .
\end{equation*}
\end{defn}

\begin{rem}\label{rem:types_of_projections}
The above concept embraces a number of other notions existing in the literature. 
For instance, if $B$ is a unital Banach algebra then $\{P,1-P\}$ are jointly $\{0,1\}$-contractive if and only if $\{ P,1-P \}\subseteq B_1$, that is $P$ is bicontractive \cite{Byrne_Sullivan}, \cite{Bernau_Lacey2}. The projection in Example \ref{ex:real_representation_from_bicontractive} is bicontractive  but not an $L^p$-projection.
If $\F = \C$ and $\T = \{ z\in \C : |z|=1 \}$, then $\{P,1-P\}$ are jointly $\T$-contractive if and only if $zP +w (1-P)$ are isometries (contractive invertible elements with contractive inverses), for all $z,w\in \T$, which by definition means that $P$ is \emph{bicircular}, see \cite{Stacho_Zalar}. 
Every bicircular projection is bicontractive but not conversely, see \cite{Stacho_Zalar}. 
In fact, a projection $P$ is bicircular if and only if it is hermitian \cite{Jamison}, and so bicontractive projections in Example \ref{ex:real_representation_from_bicontractive} are not bicircular by Theorem \ref{thm:spatial^partial_isometries_description}.
For any $p\in [1,\infty]$ and any Banach space $E$, mutually orthogonal $L^p$-projections  $\{P_i\}_{i\in F}\subseteq B(E)$ are jointly $\F$-contractive (the Banach algebra they generate embeds isometrically into $c_0(F)$). 
\end{rem}

\begin{defn}\label{defn:representation_semigroup_contractive}
A \emph{representation of the inverse semigroup $S$ in a Banach algebra} $B$ is a semigroup homomorphism $v : S \to B_1$ such that for every finite $F\subseteq E$ the idempotents $\{ P_{F_0}^F \}_{F_0\subseteq  F}$ given by \eqref{eq:general_projections_to_be_contractive} are jointly $\F$-contractive.
\end{defn}

\begin{lem}\label{lem:representation_characterization}
Let $v : S \to B_1$ be a  semigroup homomorphism into the contractive elements in a Banach algebra $B$, and let $A := \clsp \{ v_{e} : e\in E \} \subseteq B$. 
The following are equivalent:
\begin{enumerate}
    \item\label{enu:representation_characterization1} $v$ is a representation of $S$ in $B$;
    \item\label{enu:representation_characterization2} the map $1_{Z(e)}\mapsto v_{e}$ extends to a representation $\pi:C_0(\widehat{E})\to A\subseteq B$; 
    \item\label{enu:representation_characterization3} $A\cong C_0(X)$ for a closed subset $X\subseteq \widehat{E}$;
    \item\label{enu:representation_characterization4} $A$ embeds isometrically into a  $C^*$-algebra.
\end{enumerate}
The above equivalent conditions always hold when $B=B(K)$ and $v|_E$ takes values in the $L^p$-projections on a Banach space $K$, for some fixed $p\in[1,\infty]\setminus\{2\}$. 
In particular, any homomorphism $v : S \to \SPIso(L^p(\mu)) \subseteq B(L^p(\mu))$, where $p\in[1,\infty]$, is a representation of $S$, and any $*$-homomorphism $v : S \to B_1$ where $B$ is a $C^*$-algebra is a representation of $S$. 
\end{lem}
\begin{proof} 
\ref{enu:representation_characterization1}$\implies$\ref{enu:representation_characterization2}.  
By Proposition~\ref{prop:Steinberg_algebra_of_inverse_semigroup}, the map $1_{Z(e)}\mapsto v_{e}$ extends to an algebra homomorphism $\pi : A(\widehat{E}) \to A\subseteq B$ given by $\pi(\sum_{e\in F} a_e 1_{Z(e)}) = \sum_{e\in F} a_e P_{e}$, for finite $F\subseteq E$ and $a_e\in \F$, $e\in F$.  
Let $\{ P_{F_0}^{F} \}_{F_0\subseteq  F}$ be the orthogonal idempotents associated to $F$ in Lemma~\ref{lem:inclusion-exclusion}.
The corresponding idempotents associated to $e\mapsto 1_{Z(e)}$ are $\{1_{Z(\bigwedge F_0)\setminus \bigcup_{f\in F\setminus F_0} Z(\bigwedge F_0\wedge f)} \}_{F_0\subseteq  F}$.
So there are $b_{F_0}\in \F$, $F_0\subseteq F$, such that 
\[
    \sum_{e\in F} a_e 1_{Z(e)} = \sum_{F_0\subseteq F} b_{F_0} 1_{Z\big(\bigwedge F_0\big) \setminus \bigcup_{f\in F\setminus F_0} Z\big(\bigwedge F_0\wedge f\big)} \quad \text{ and }\quad
    \sum_{e\in F} a_e P_{e} = \sum_{F_0\subseteq F} b_{F_0} P_{F_0}^{F}. 
\]
The joint $\F$-contractiveness of $\{P_{F_0}^{F}\}_{F_0\subseteq  F}$ implies 
\[
    \Big\| \sum_{e\in F} a_e P_{e} \Big\| \leq \max_{P_{F_0}^{F} \neq 0} | b_{F_0} | \leq \max_{Z\big(\bigwedge F_0\big) \setminus \bigcup_{f\in F\setminus F_0} Z \big( \bigwedge F_0\wedge f \big) \neq \emptyset } | b_{F_0} | = \Big\| \sum_{e\in F} a_e 1_{Z(e)} \Big\|_{\infty}.
\]
Hence the map $\pi:A(\widehat{E})\to A\subseteq B$ is contractive and so it extends (uniquely) to a representation $\pi : C_0(\widehat{E}) \to A\subseteq B$, because $A(\widehat{E})$ is dense in $C_0(\widehat{E})$. 

Implication \ref{enu:representation_characterization2}$\implies$\ref{enu:representation_characterization3} follows from Proposition~\ref{prop:minimality_of_sup_norm}, and \ref{enu:representation_characterization3}$\implies$\ref{enu:representation_characterization4} is trivial. 
Condition~\ref{enu:representation_characterization4} allows us to assume that $A\subseteq B(H)$ for a Hilbert space $H$. 
Then $v_{e}$, $e \in E$, are orthogonal projections on $H$ (as they are contractive idempotents), so the  idempotents \eqref{eq:general_projections_to_be_contractive} are orthogonal projections (as they are self-adjoint idempotents). 
In particular, idempotents \eqref{eq:general_projections_to_be_contractive} are mutually orthogonal $L^2$-projections, and so they are jointly $\F$-contractive.
This shows \ref{enu:representation_characterization4}$\implies$\ref{enu:representation_characterization1}.

When $v|_E$ takes values in the $L^p$-projections on a Banach space $K$, for  $p\in[1,\infty]\setminus\{2\}$, 
then  $v|_E$ takes values in the $p$-Cunningham algebra  $C_p(E) := \clsp \mathbb{P}_{p}(E)$  which is isometrically isomorphic to   $C_0(X)$, see \cite{BDEGGMM},
 and hence the equivalent conditions hold.
\end{proof}

\begin{defn}
We define the \emph{universal inverse semigroup Banach algebra} $F(S)$ as the completion of $\F S$ in the norm 
\[
    \| {\textstyle \sum_{t\in F}} a_f t \|_{F(S)} := \sup\{ \| {\textstyle \sum_{t\in F}} a_f v_t \| : \text{$v:S\to B$ is a representation in a Banach algebra $B$} \} .
\]
For $p\in[1,\infty]$ we define the $L^p$-operator algebra $F^p(S)$ in a similar way, as the completion of $\F S$ in the maximal norm for homomorphisms into spatial partial isometries on $L^p$-spaces.
\end{defn}

\begin{thm}\label{thm:inverse_semigroups_Banach_representations}
 Let $\G(S) := S \ltimes_{h} \widehat{E}$ be the associated  universal groupoid for an inverse semigroup  $S$ and let $\overline{S}:=\{U_t\}_{t\in S}\cup \{\widehat{E}\}\subseteq \Bis(\G(S))$ be the associated unital inverse semigroup covering $\G(S)$.
We have an isometric isomorphism $F(S) \cong F^{\overline{S}}(\G(S))$.
Namely, the relations 
\begin{equation}\label{eq:semigroup_correspondence}
\psi(1_{U_t}) = v_t,\qquad  t\in S,
\end{equation} establish a bijective correspondence between representations $\psi : F^{\overline{S}}(\G(S)) \to B$ and $v : S \to B_1$ in a Banach algebra $B$. 
Moreover, $\pi := \psi|_{C_0(\widehat{E})}$ is isometric if and only if for every $e\in E$ and every finite $F\subseteq eE\setminus\{e\}$ we have $\prod_{f\in F} (v_e -v_f) \neq 0$.
\end{thm}
\begin{proof}
By the last part of Theorem~\ref{thm:disintegration} every representation of $F^{\overline{S}}(\G(S))$ in $B$ is of the form $\pi\rtimes v$ where $\pi  :C_0(\widehat{E}) \to B$ is a representation and $v  :S\to B_1$ is a semigroup homomorphism such that $\pi(1_{Z(e)})=v_e$, $e\in E$,
 and $v_t\pi(a)v_{t^*} = \pi(a\circ h_{t^*})$ for all $a\in C(Z(t^*t))$, $t\in S$. 
In particular, Lemma~\ref{lem:representation_characterization} implies that $v:S\to B_1$ is a representation of the inverse semigroup $S$.
Conversely, if $v:S\to B_1$ is a representation of the inverse semigroup $S$, then by Proposition~\ref{prop:Steinberg_algebra_of_inverse_semigroup} there is a unique algebra homomorphism  $\psi : A(\G(S))\to B$ such that $\psi(1_{U_t}) = v_{t}$, $t\in S$. 
Putting $\pi:=\psi|_{A(\widehat{E})}$, we have $v_t \pi(a )v_{t^*} = \pi(a\circ h_{t^*})$ for all $a\in 1_{Z(t^*t)} A(\widehat{E})$, $t\in S$. 
By Lemma~\ref{lem:representation_characterization}, $\pi : A(\widehat{E}) \to B$ extends to a representation $\pi : C_0(\widehat{E})\to B$ and hence the latter condition holds for all $a\in C(Z(t^*t)) = \overline{1_{Z(t^*t)} A(\widehat{E})}$.
Thus $\psi$ extends to a representation $\psi = \pi\rtimes v : F^{\overline{S}}(\G(S))\to B$. 
This gives the first part of the assertion. 

For the second part recall that basic non-empty open sets are of the form  $Z(e) \setminus \bigcup_{f \in F} Z(f)$ for $e\in E$ and finite $F\subseteq eE\setminus\{e\}$.
Moreover, for the corresponding representations we have $\prod_{f\in F} (v_e -v_f) = \pi(1_{Z(e) \setminus \bigcup_{f \in F} Z(f)})$. 
Thus if $\pi$ is injective then $\prod_{f\in F} (P_e -P_f)\neq 0$, see also Corollary~\ref{cor:injectivity_representation_A(E)}. 
If $\pi$ is not isometric then  $\ker\pi  =C_0(U)$ for a non-empty open set $U\subseteq \widehat{E}$.
Hence $\pi$ vanishes on some non-empty set  $1_{Z(e) \setminus \bigcup_{f \in F} Z(f)}$, which implies that $\prod_{f\in F} (P_e -P_f)= 0$.
\end{proof}


\begin{cor}\label{cor:Lp_semigroup_representations} 
Let $p\in [1,\infty]$. 
The isomorphism from Theorem~\ref{thm:inverse_semigroups_Banach_representations} descends to an isometric isomorphism $F^p(S) \cong F^p(\G(S))$. 
If $\F=\C$ then $F^p(S)$ is universal for all representations of $S$ in $L^p$-operator algebras, and if $p\neq 2$, then every non-degenerate representation $\psi:F^p(S)\to B(L^p(\mu))$, where $\mu$ is localizable, is given by a semigroup homomorphism $V : S \to\SPIso(L^p(\mu))$ into spatial partial isometries.
\end{cor}
\begin{proof} 
By Theorem~\ref{thm:spatial_representations_of_groupoid_algebras},  $F^p(\G(S))$ is universal for spatial representations of $F^S(\G(S))$ on $L^p(\mu)$, which via Theorem~\ref{thm:inverse_semigroups_Banach_representations} correspond to homomorphisms $v : S \to \SPIso(L^p(\mu)) \subseteq B(L^p(\mu))$, cf.\ the last part of Lemma~\ref{lem:representation_characterization}.
This gives $F^p(S) \cong F^p(\G(S))$. 
The remaining part follows from Theorems  \ref{thm:initial_on_L^p_full}\ref{enu:initial_on_L^p_full4} and \ref{thm:L^p_representations of C_0(X)} 
\end{proof}

\begin{rem}
When $S$ has a zero it is natural to consider zero-preserving homomorphisms. 
The corresponding \emph{contracted algebras} $\F_0 S$, $F_0^p(S)$, $p\in[1,\infty]$, and $F_0(S)$, could be defined as  quotients of $\F S$, $F^p(S)$, $p\in[1,\infty]$, and $F(S)$ respectively by the one-dimensional space generated by $0\in S$.
The \emph{contracted groupoid} $\G_0(S)$ modeling these algebras is the transformation groupoid for the $S$-action restricted to the closed $S$-invariant set $\widehat{E}\setminus \{1\}$. 
\end{rem}

We now pass to representations of $S$ that exploit the Boolean ring structure of the range of the semilattice $E$. 
Here by a \emph{Boolean ring} (a \emph{generalized Boolean algebra} in the nomenclature of \cite{Steinberg}) we 
mean a possibly non-unital strucutre, so it is  a quintuple $(R,0,\vee, \wedge, \setminus)$ where $\setminus$ is a binary operation (relative complement). 
We  generalize here Exel's theory of tight representations \cite{Exel}, or rather its `non-unital version' \cite{Donsig_Milan},  to the Banach algebra context. We also characterise injectivity of the relevant representations, 
which seems to be a new results even  in the  $C^*$-algebraic setting. 

To this end, we fix again a semilatice $E$ of idempotents in an inverse semigroup $S$.

\begin{defn} 
A \emph{cover} of $e\in E$ is a finite set $F\subseteq eE$ such that for every nonzero $z\leq e$ we have $zf\neq 0$ for some $f\in F$.
We say that a map $v : E\to R$ into a Boolean ring $R$ is \emph{tight} (or \emph{cover-to-join}) if the element $v_e - \bigvee_{f\in F} v_f = \bigwedge_{f\in F} (v_e\setminus v_f)$ is zero whenever $F$ covers $e\in E$.
A semigroup homomorphism $v : S\to B$ into an $\F$-algebra $B$ is \emph{tight} if $\prod_{f\in F} (v_e - v_f) = 0$ for every cover $F$ of $e\in F$ 
(this is equivalent to saying that $v|_E$ is tight as a map into a Boolean ring of idempotents in $\spane \{ v_e : e\in E \} \subseteq B$).
\end{defn}

\begin{rem}
If $E$ has a zero, then $\emptyset$ covers $0$ and so $v_0=0$ for every tight map $v:E\to R$ into a Boolean ring $R$.
If $E$ is a Boolean ring, then a semigroup homomorphism $v:E\to R$ is a Boolean ring homomorphism if and only if $v$ is tight.
\end{rem}

By  \cite[Theorem 12.9]{Exel}, the set
\[
    \widehat{E}_{\tight} := \{ \phi : E \to \{0,1\} \ \text{is a tight homomorphism} \}
\]
is a closed subset of $\widehat{E}$. 
In fact, $\widehat{E}_{{\tight}}$ is the closure of the set of $\phi\in \widehat{E}$ whose support $\supp(\phi ) = \{ e\in E : \varphi(e) = 1 \}$ is an ultrafilter. 
We call $\widehat{E}_{{\tight}}$ the \emph{tight spectrum} of $E$.
When $E = E(S)$ is the semilattice of idempotents in an inverse semigroup $S$ then $\widehat{E}_{{\tight}}$ is invariant for the $S$-action on $\widehat{E}$, and hence the corresponding tranformation groupoid $\G_{\tight}(S) := S\ltimes \widehat{E}_{{\tight}}$, called the \emph{tight groupoid of $S$}, is the restriction of $\G(S)$ to the set $\widehat{E}_{{\tight}}$. 
By \cite[Corollary 2.14]{Steinberg_Szakacs} the isomorphism from Proposition~\ref{prop:Steinberg_algebra_of_inverse_semigroup} factors to the isomorphism 
\begin{equation}\label{eq:tight_algebra}
    A(\G_{\tight}(S)) \cong \F S/ \Big\langle \prod_{f\in F} e - f : \text{$F$ covers $e\in E$} \Big\rangle,   
\end{equation}
and this algebra is universal for tight semigroup homomorphisms. 

\begin{defn} 
The \emph{tight Banach algebra} $F_{\tight}(S)$ is the completion of the algebra \eqref{eq:tight_algebra} in the maximal norm induced by tight representations $v : S\to B$ in Banach algebras $B$.
Similarly, for $p\in[1,\infty]$, we define the \emph{tight $L^p$-Banach algebra} $F^p_{\tight}(S)$, using tight representations in spatial partial isometries on $L^p$-spaces. 
\end{defn}

\begin{thm}\label{thm:tight_inverse_semigroups_Banach_representations}
Let $\G_{{\tight}}(S)$ be the tight groupoid for an inverse semigroup $S$ and let $\overline{S}$ be the canonical image of $S$ in $\Bis(\G_{{\tight}}(S))$ with added the unit fiber (if necessary).
The isomorphism \eqref{eq:tight_algebra} extends to an isometric isomorphism 
\[
    F_{\tight}(S)\cong F^{\overline{S}}(\G_{\tight}(S))
\]
and \eqref{eq:semigroup_correspondence} gives a bijective correspondence between representations $\psi:F^{\overline{S}}(\G_{\tight}(S))\to B$ and tight representations $v:S\to B_1$ in a Banach algebra $B$.
Moreover, $\psi$ is isometric on $C_0(\widehat{E}_{\tight})$ if and only if $\prod_{f\in F} (v_e -v_f) \neq 0$ for every $e\in E$ and finite $F\subseteq eE$ that does not cover $e$.
\end{thm}
\begin{proof} 
One may mimic the proof of Theorem~\ref{thm:inverse_semigroups_Banach_representations}, using that the algebras in \eqref{eq:tight_algebra} are universal for cover-to-join maps. 
In particular, for the second part note that basic non-empty open sets in $\widehat{E}_{\tight}$ are of the form  $Z(e) \setminus \bigcup_{f \in F} Z(f) \cap \widehat{E}_{\tight}$ for $e\in E$ and finite $F\subseteq eE\setminus\{e\}$. 
Such a set is empty if and only if $F$ is a cover of $e$.
Indeed, if $F$ is not a cover, then there is $0\neq z\leq e$ with $zf=0$ for all $f\in F$, and taking
any ultrafilter containing $z$ the corresponding $\phi\in \widehat{E}_{\tight}$ is in $Z(e) \setminus \bigcup_{f \in F} Z(f)$.
\end{proof}
\begin{cor}\label{cor:tight_Lp}
Let $p\in [1,\infty]$. The isomorphism \eqref{eq:tight_algebra} extends to an isometric isomorphism 
$
F^p_{\tight}(S)\cong F^p(\G_{\tight}(S)).
$ 
If $\F=\C$, then $F^p_{\tight}(S)$ is a universal for all tight representations of $S$ in $L^p$-operator algebras, and if $p\neq 2$, then every non-degenerate representation $\psi:F^p(S)\to B(L^p(\mu))$, where $\mu$ is localizable, is given by a tight  homomorphism $V : S \to\SPIso(L^p(\mu))$ into spatial partial isometries.
\end{cor}
\begin{proof} Follow the proof of Corollary~\ref{cor:Lp_semigroup_representations} using Theorem~\ref{thm:tight_inverse_semigroups_Banach_representations}.
\end{proof}

En passant we improve the main result of \cite{Exel_recon} by removing the second-countability assumption, and replacing the assumption that the semigroup forms a basis for the topology by the requirement that it is wide and separates the points. This seemingly slight generalization is very useful in applications, see e.g. \cite{BKM}. 

\begin{cor}\label{cor:disconnected_are_tight}
Let $X$ be a totally disconnected locally compact Hausdorff space and let $E$ be a semilattice consisting of compact open sets that separate the points of $X$ and cover $X$. 
The pairing $\phi(e) = 1_{e}(x)$, for $x\in X$, $\phi\in \widehat{E}_{{\tight}}$, $e\in E$, gives a homeomorphism $X\cong \widehat{E}_{{\tight}}$.
\end{cor}
\begin{proof} 
The map $E\ni e \mapsto 1_e\in C_0(X)$ is a tight representation, and so it gives rise to a representation $\pi : C_0(\widehat{E}_{{\tight}}) \to C_0(X)$ which is surjective by the Stone--Weierstrass Theorem. 
To see that $\pi$ is isometric note that $\pi( 1_{Z(e) \setminus \bigcup_{f} Z(f)} |_{\widehat{E}_{{\tight}}} ) = \prod_{f\in F} (1_e - 1_f) = 1_{e\setminus \bigcup  F}$ for every $e\in E$ and finite $F\subseteq eE$.
So if $\prod_{f\in F} (1_e -1_f) = 0$, then $e = \bigcup F$ which means that $F$ is a cover of $e$.
Hence $\pi$ is isometric by Theorem~\ref{thm:tight_inverse_semigroups_Banach_representations}. 
It is immediate that the homeomorphism dual to the isomorphism $C_0(\widehat{E}_{{\tight}}) \cong C_0(X)$ is of the described form.
\end{proof}

\begin{cor}\label{cor:realization_of_ample_groupoids}
Let $\G$ be an ample groupoid, and let $S\subseteq \Bis(\G)$ be a wide inverse semigroup of compact open bisections that separate points in $X$. 
We have a natural isomorphism $\G\cong \G_{\tight}(S)$ of topological groupoids, and so we have a natural isomorphism $F_{\tight}(S)\cong F^S(\G)$.
\end{cor}
\begin{proof}  
Recall that $\G_{\tight}(S)$ is the transformation groupoid for the action $\hat{h}$ where 
$  \hat{h}_{t} : \hat{Z}(t^*t) \to \hat{Z}(tt^*)$,   
$\hat{Z}(t^*t) := \{\phi \in \widehat{E}_{\tight}: \phi(t^*t)=1\}$ and $\hat{h}_{t}(\phi)(e) := \phi(t^*et)$ for  $t\in S$, $e\in E$, $\phi\in \hat{Z}(t^*t)$. 
Similarly $\G$ is the transformation groupoid for the action 
\[
    h:S\to \PHomeo(X) ;\ h_t = r\circ d|_{t}^{-1}: t^*t\to tt^*, \qquad t\in S .
\]
Note that for $e\leq t^*t$ we have $h_t(e)=tet^*$. 
Let $\Psi:\widehat{E}_{{\tight}}\to X$ be the homeomorphism from Corollary \ref{cor:disconnected_are_tight}.
For $\phi \in \widehat{E}_{{\tight}}$, $\phi(t^*t)=1$ if and only if $1_{t^*t}(\Psi(\phi))=1$. Thus $\Psi:\hat{Z}(t^*t)\to t^*t$. 
Moreover, for $\phi\in \hat{Z}(t^*t)$ and $e\in E$,
we have
\[
    1_{e} \Big( \Psi \big( \hat{h}_{t}(\phi) \big) \Big) = \hat{h}_{t}(\phi)(e) = \phi(t^*et) = \phi \big( h_t^{-1} (e) \big) = 1_{h_t^{-1}(e)}(\phi) = 1_e \Big( h_t \big( \Psi(\phi) \big) \Big) .
\]
Thus $\Psi$ intertwines the $S$-actions and so the transformation groupoids are isomorphic.
\end{proof}

\subsection{Banach algebras associated to directed graphs}
\label{subsect:directed_graphs}

Let $Q = (Q^0,Q^1, r_Q, s_Q)$ be a \emph{directed graph} (sometimes also called a \emph{quiver}). 
So $Q^0$ is the set of vertices, $Q^1$ is the set of edges and $r_Q, s_Q : Q^{1}\to Q^0$ are the range and source maps. 
The vertices in $Q_{\textup{reg}} := \{ v\in Q^0: 1 < |r_Q^{-1}(v)|<\infty\}$ are called \emph{regular} and the remaining $Q_{\textup{sing}} := Q^0\setminus Q_{\textup{reg}}$ are called \emph{singular}.
We denote by  $Q^n$, $n>0$, the set of finite paths $\mu=\mu_1 \cdots \mu_n$, where $s_Q(\mu_i) = r_Q(\mu_{i+1})$ for all $i=1,\ldots ,n-1$. 
Then $|\mu|=n$ stands for the length of $\mu$ and $Q^* = \bigcup_{n=0}^\infty Q^n$ is the set of all finite paths (vertices are treated as paths of length zero). 
We denote by $Q^\infty$  the set of infinite paths and put $Q^{\leq \infty} := Q^*\cup Q^{\infty}$. 
The maps $r_Q,s_Q$ extend naturally to $Q^*$ and $r_Q$ extends to $Q^\infty$.  
We recall the defintion of the Leavitt path algebra associated to $Q$, see \cite{AAM}.

\begin{defn}\label{defn:graph_famillies} 
A \emph{$Q$-family in an algebra $B$} (over $\F=\R,\C$) is a pair $(P,T)$ where $P=\{P_v\}_{v\in Q^0}\subseteq B$ consists of pairwise orthogonal idempotents and $T=\{T_e, T_e^*\}_{e\in Q^1}\subseteq B$
consists of elements where $T_e^{*}$ is a generalized inverse of $T_e$, $e\in Q^1$, such that 
\begin{enumerate}[label={(CK\arabic*)}]
    \item $T_e^*T_e=P_{s(e)}$  for all $e\in Q^{1}$;
    \item   $\{T_e T_e^*\}_{e \in E^1}$ are mutually orthogonal and $T_e T_e^*\leq P_{r(e)}$ for all $e\in Q^{1}$;
    \item $P_v=\sum_{e \in r^{-1}(v)} T_e T_e^*$ for all $v\in Q^{0}_{\reg}$.
\end{enumerate} 
\end{defn}

Let $(P,T)$ be a $Q$-family in $B$. 
For $\mu = \mu_1 \cdots \mu_n \in Q^{*}$ we put $T_\mu := T_{\mu_1} \cdots T_{\mu_n}$ and $T_\mu^* := T_{\mu_n}^*\cdots T_{\mu_1}^*$. 
Then the elements $T_{\mu}T_{\nu}^*$ form an inverse semigroup with multiplication from $B$ and so $\spane\{T_{\mu}T_{\nu}^*: \mu,\nu \in Q^*\}$ is a subalgebra of $B$ generated by $P\cup T$.
In fact, it is a $*$-algebra with the involution determined by $(T_{\mu}T_{\nu}^*)^* = T_{\nu}T_{\mu}^*$, $\mu,\nu \in Q^*$. 
By definition the \emph{Leavitt path algebra} $\L(Q)$ is the algebra generated by a universal $Q$-family $(p,t)$. 
Thus 
\[
    \L(Q) = \spane\{t_{\mu}t_{\nu}^*: \mu,\nu \in Q^*\} ,
\]
and for any $Q$-family $(P,T)$ in $B$ there is a unique algebra homomorphism $\psi_{(P,T)} : \L(Q)\to B$ such that $\psi_{(P,T)}(p_v)=P_v$, $\psi_{(P,T)}(t_e)=T_e$ and $\psi_{(P,T)}(t_e^*)=T_e^*$, for $v\in Q^0$, $e\in Q^1$. 
We refer to \cite{AAM} for more details. 
The elements $\{t_{\mu}t_{\nu}^*: \mu,\nu \in Q^*\}$ form an inverse semigroup that can be described in terms of $Q$ as follows.
 Put  $S_{Q}:=\{(\mu,\nu)\in Q^* \times Q^*:s_Q(\mu) = s_Q(\nu)\}\cup \{0\}$ and 
define multiplication in $S_Q$ by the formula
\[
    (\mu,\nu)(\alpha,\beta) :=  \begin{cases}
                                    (\mu \alpha',\beta) & \text{if $\alpha = \nu \alpha'$}, \\
                                    (\mu, \beta \nu') & \text{if $\nu = \alpha \nu'$}, \\
                                    0 &\text{otherwise.}
                                \end{cases}
\]
Then the involution is given by $(\mu,\nu)^* =(\nu,\mu)$. 
It is  known to experts that tight representations of $S_{Q}$ in are in bijective correspondence with  $Q$-families. 
We deduce this from our results. 
The groupoid model for $\L(Q)$ (isomorphic to the tight groupoid $\G_T(S_Q)$) can be described as follows \cite[Examples 2.1, 3.2]{Clark_Sims}. 
We put $Q^*_{\mathrm{sing}} := \{ \mu\in Q^*: s_Q(\mu) \in Q_{\mathrm{sing}}\}$. 
For any $\eta\in Q^*\setminus Q^0$ let ${\eta}Q^{\leq \infty} := \{ \mu = \mu_1 \cdots \in Q^{\leq \infty} : \mu_1 \cdots \mu_{|\eta|} = \eta \}$ and for $v\in Q^0$ put ${v}Q^{\leq \infty} := \{\mu \in Q^{\leq \infty}: r_Q(\mu)=v\}$.
The \emph{boundary space} of $Q$ \cite[Section 2]{Webster} is the set 
\[ 
    \partial Q := Q^\infty\cup Q^*_{\mathrm{sing}} 
\]
equipped with the topology generated by the `cylinders' $Z(\mu):=\partial Q\cap {\mu}Q^{\leq \infty}$, $\eta \in Q^*$, and their relative complements. 
In fact, the sets $Z(\mu) \setminus \bigcup_{e \in F} Z(\mu e)$, where $\mu\in Q^*$ and $F\subseteq {s(\mu)}Q^1$ is a finite set of edges, form a basis of compact-open sets for the Hausdorff topology on $\partial Q$ \cite[Section 2]{Webster}. 
The one-sided \emph{topological Markov shift} associated to $Q$ is the map $\sigma:\partial Q\setminus Q^0 \to \partial Q$ defined, for $\mu=\mu_1\mu_2 \cdots \in \partial Q\setminus Q^0$, by the formulas  
\[
    \sigma(\mu) := \mu_2\mu_3 \cdots \ \text{ if $\mu \notin Q^1$}, \qquad \sigma(\mu) := s(\mu_1) \  \text{if $\mu=\mu_1 \in Q^1$} .
\]
This is a partial local homeomorphism on $\partial Q$. 
By definition the \emph{groupoid of the graph} $Q$ is the Deaconu--Renault groupoid $\G_Q := \G(\partial Q,\sigma)$ of $\sigma$. 
Thus
\[
    \G_Q = \{ (\mu x, |\mu|-|\eta|, \eta x) : \mu, \eta \in Q^*, \ x\in \partial Q,\ s_Q(\mu) = s_Q(\nu) = r_Q(x) \} 
\]
is an ample Hausdorff groupoid with the topology generated by the `cylinders' $Z(\mu,\eta) := \{ (\mu x, |\mu|-|\eta|, \eta x) \in \G_Q \}$, for $(\mu,\nu)\in S_{Q}$, and their relative complements. 
Hence the sets $Z(\mu,\eta) \setminus \bigcup_{\alpha \in F} Z(\mu \alpha)$, where $(\mu,\nu)\in S_{Q}$ and $F\subseteq {s(\mu)}Q^*$ is finite, form a basis of compact open bisections for $\G_Q$. The algebraic structure is given by 
$    (\mu,n,\eta) (\eta,m,\kappa) := (\mu,n+m,\kappa)$ and  $(\mu,n,\eta)^{-1} := (\eta,-n,\mu)$.
We identify the unit space $\{(x,0,x):x\in \partial Q\}$ with $\partial Q$. 
The unitization of the natural image of $S_Q$ in $\Bis(\G_Q)$ is $\overline{S_Q}:=\{Z(\mu,\eta)\}_{(\mu,\nu)\in S_{Q}} \cup \{\emptyset, \partial Q \} $.
The associated Steinberg algebra is 
\[
    A(\G_Q) := \spane \{ 1_{Z(\mu,\eta)} : \mu, \eta \in Q^*,\ s_Q(\mu)=s_Q(\eta) \} .
\]
Note that $A(\G_Q)\subseteq C_c(\G_Q)$ is a dense $*$-subalgebra of $F(\G_Q)$, and $A(\partial Q) := \spane\{ 1_{Z(\mu)} : \mu\in Q^* \} \subseteq C_c(\partial Q)$ is a dense $*$-subalgebra in $C_0(\partial Q)$.

\begin{prop}[\cite{Clark_Sims}]\label{prop:Leavitt_groupoid_model}
There is an algebra isomorphism $\Psi:\LL(Q)\to A(\G_Q)$ that maps $t_{\mu}t_{\nu}^*$ to $1_{Z(\mu,\nu)}$, $(\mu,\nu)\in S_{Q}$.
It is uniquely determined by $\Psi(p_v) = 1_{Z(v)}$, $\Psi(t_e)=1_{Z(e,s_Q(e))}$ and $\Psi(t_e^*)=1_{Z(s_Q(e),e)}$, for $v\in Q^0$, $e\in Q^1$. 
In particular we have an isomorphism $ A(\partial Q) \cong \spane\{t_{\mu}t_{\mu}^*: \mu \in Q^*\}$ where $1_{Z(\mu)}\mapsto t_{\mu}t_{\mu}^*$.
\end{prop}

\begin{lem}\label{lem:Webster_lemma}
For any $Q$-family $(P,T)$ in an algebra $B$ and any finite $F\subseteq Q^*$ the elements
\[
    P^F_\mu := \prod_{\mu\mu'\in F\setminus \{\mu\}} T_\mu T_\mu^*-  T_{\mu\mu'} T_{\mu\mu'}^*, \qquad \mu \in F, 
\]
(with the convention $P^F_\mu:= T_\mu T_\mu^*$ if the product is over the empty set)
are mutually orthogonal idempotents such that $T_{\mu}T_{\mu}^* = \sum_{\mu\mu'\in F\setminus \{\mu\}}P^F_\mu$, so they span the algebra $\spane \{ T_{\mu}T_{\mu}^* : \mu \in F \}$.
If $B$ is a Banach algebra, then the map $(A(\partial Q), \|\cdot\|_{\infty}) \rightarrow \spane\{T_{\mu}T_{\mu}^*: \mu \in Q^*\}\subseteq B$; $1_{Z(\mu)}\mapsto t_{\mu}t_{\mu}^*$ is contractive if and only if for every finite $F\subseteq E^*$ the idempotents $\{P^F_\mu\}_{\mu \in F}$ are jointly $\F$-contractive (Definition~\ref{defn:jointly_contractive}).
\end{lem}
\begin{proof}
The first part follows from \cite[Lemma 3.1]{Sims_Webster}, see also \cite[Lemma 3.4]{Webster}.
This corresponds to the fact every element $f\in A(\partial Q)$ can be written in the form $f=\sum_{\mu \in F} a_{\mu} 1_{Z(\mu, F)}$ where $Z(\mu, F):=Z(\mu) \setminus \bigcup_{\mu' \in F\setminus \{\mu\}} Z(\mu\mu')$ are pairwise disjoint compact open sets for all $\mu\in F$ and a fixed finite set $F\subseteq Q^*$. 
Then $\| f \|_{\infty}=\max_{\mu\in F} |a_{\mu}|$, and the map in the second part of the assertion sends $f$ to $\sum_{\mu \in F} a_\mu{P^F_\mu}$, so it is contractive if $\sum_{\mu \in F} a_\mu{P^F_\mu}\leq \max_{\mu\in F} |a_{\mu}|$ for all choices of $F$ and $a_\mu$. 
\end{proof}

\begin{defn}\label{defn:Banach_space_Q_family}
A \emph{Banach $Q$-family} is a $Q$-family $(P,T)$ in a Banach algebra $B$ where $P\cup T\subseteq B_1$ and for any finite $F\subseteq Q^*$ the idempotents $\{P^F_\mu\}_{\mu \in F}$ are jointly $\F$-contractive. 
A \emph{spatial $Q$-family on $L^p(\mu)$}, where $p\in [1,\infty]$, is a $Q$-family $(P,T)$ in the algebra $B(L^p(\mu))$ where $P\cup T\subseteq \SPIso(L^p(\mu))$ are spatial partial isometries (then $T_e^*$ is the adjoint of $T_e$ in the sense of the inverse semigroup $\SPIso(L^p(\mu))$). 
We define the \emph{universal graph Banach algebra} $F(Q)$ of $Q$ to be the completion of $\LL(Q)$ in 
\[
    \| b \|_{F(Q)} := \sup \{ \psi_{(P,T)}(b): \text{$(P,T)$ is a $Q$-family in a Banach algebra} \}.
\]
For $p\in [1,\infty]$ we define the \emph{graph $L^p$-operator algebra} $F^p(Q)$ to be the completion of $\L(Q)$ in the norm $\|b\|_{L^p}=\sup\{\psi_{(P,T)}(b): \text{$(P,T)$ is spatial $Q$-family on an $L^p$-space} \}$. 
\end{defn}

\begin{rem}
Lemma~\ref{lem:representation_characterization} implies that a spatial $Q$-family on $L^p(\mu)$ is a Banach $Q$-family in $B(L^p(\mu))$, and any $Q$-family $(P,T)$ in a $C^*$-algebra $B$ consisting of contractive elements is a Banach $Q$-family in $B$.
\end{rem}

\begin{thm}\label{thm:representations_of_graph_algebras}
For any directed graph $Q$ we have natural isometric isomorphisms 
\begin{equation}\label{eq:graph_algebra_isomorphisms}
F(Q)\cong F^{\overline{S_{Q}}}(\G_Q)\cong F_{\tight}(S_Q)
\end{equation}
More precisely,  
for any Banach algebra $B$ the relations 
\begin{equation}\label{eq:graph_algebra_isomorphism_relations}
\psi(p_v)=P_v=V_{(v,v)},\qquad  \psi(t_e)=T_e=V_{(e,s(e))},\qquad  \psi(t_e^*)=T_e^*=V_{(s(e),e)},
\end{equation}
 for $v\in Q^0$, $e\in Q^1$, establish bijective correspondences between representations $\psi:F(Q) \to B$, Banach $Q$-families $(P,T)$ in $B$, and tight Banach algebra representations $V:S_Q\to B_1$.
	
\end{thm}
\begin{proof}
Note that $S:=\overline{S_{Q}}\setminus \{\partial Q\}= \{ Z(\mu,\nu) : (\mu, \eta) \in S_Q\} \cup \{\emptyset\}$ forms a wide inverse semigroup of compact open bisections in $\G_Q$ that separates the points in $\G_{Q}$.
Thus $\G_Q\cong \G_{\tight}(S)$ is the tight groupoid of $S$ by Corollary~\ref{cor:realization_of_ample_groupoids}. 
We have natural semigroup isomorphisms $S\cong S_Q\cong \{t_{\mu}t_{\nu}^*: \mu,\nu \in Q^*\}$ from  Proposition~\ref{prop:Leavitt_groupoid_model}. 
Thus by Theorem~\ref{thm:tight_inverse_semigroups_Banach_representations} and Lemmas~\ref{lem:representation_characterization} and \ref{lem:Webster_lemma}  we have isometric isomorphims \eqref{eq:graph_algebra_isomorphisms}, and the relations in the assertion establish bijective correspondences between the corresponding objects. 

\end{proof}
\begin{cor}\label{cor:graph_Lp}
 For $p\in [1,\infty]$ the isomorphisms \eqref{eq:graph_algebra_isomorphisms} descend to 
$
F^p(Q)\cong F^{p}(\G_Q)\cong F_{\tight}^p(S_Q),
$
and for any localizable $\mu$,  \eqref{eq:graph_algebra_isomorphism_relations} establish bijective correspondences between  spatial $Q$-families $(P,T)$ on $L^p(\mu)$,  tight homomorphisms $V:S_Q\to\SPIso(L^p(\mu))$  and representations $\psi:F^p(Q) \to B(L^p(\mu))$ such that $\psi(t_{\mu}t_{\mu}^*)$, $\mu \in Q^*$, are operators of multiplication.
 
  If $\F=\C$, then $F^p(Q)$ is universal for  $Q$-families in all $L^p$-operator algebras and for $p\neq 2$ every non-degenerate representation $\psi:F^p(Q)\to B(L^p(\mu))$ comes from a spatial $Q$-family $(P,T)$ on $L^p(\mu)$.
\end{cor}
\begin{proof}
Combine Theorem \ref{thm:representations_of_graph_algebras} and Corollary \ref{cor:tight_Lp}.
\end{proof}
\begin{rem}
Originally $L^p$-Cuntz algebras \cite{PhLp1} and  graph $L^p$-operator algebras $F^p(Q)$ \cite{cortinas_rodrogiez}   were defined  (for $p\in [1,\infty)$ and $\F=\C$) using spatial representations on standard measure spaces. As $F^p(Q)$ is separable, one can limit only to such measures, see Remark \ref{rem:second_countable_vs_measures}. 
The last part of Corollary \ref{cor:graph_Lp} 
says that for $p\in [1,\infty]\setminus \{2\}$ the use of spatial $Q$-families  
is inevitable. For $p=2$, Corollary \ref{cor:graph_Lp} gives the tight semigroup  picture for  $C^*$-algebras associated to arbitrary graphs, which seems to be absent from the literature.
\end{rem}

\begin{rem}[Banach--Cuntz algebras]\label{rem:Cuntz-algebras}
For any $n>1$ consider the graph $Q_n$ with a single vertex $v$ and $n$ edges $e_1,...,e_n$. 
Then our $F(Q_n)$ is a Banach algebra analogue of the Cuntz $C^*$-algebra $F_n$.
Daws and Horv\'{a}th~\cite{Daws_Horwath} proposed a different construction: the \emph{Daws--Horv\'{a}th algebra} $\mathcal{DH}_n$ is by definition the quotient of the Banach $*$-algebra $\ell^{1}(S_{Q_{n}}\setminus \{0\})$ by an ideal $J$ generated by the element $\delta_{(v,v)} - \sum_{i=1}^{n} \delta_{(e_i,e_i)}$.
Thus it is the universal Banach algebra generated by a contractive $Q_n$-family, or equivalently the completion of $\LL(Q_n)$ in the largest submutliplicative norm $\|\cdot\|_{DH}$ such that $\|t_{e_i}\|_{DH}=\|t_{e_i}^*\|_{DH}=\|t_{v} \|_{DH}=1$ for $i=1, \ldots ,n$.
As shown in \cite{Daws_Horwath}, $\mathcal{DH}_n =\overline{\LL(Q_n)}^{\|\cdot\|_{DH}}$ is naturally a simple purely infinite Banach $*$-algebra.
Thus we have an injective $*$-homomorphism $\mathcal{DH}_n \donto F(Q_n)$ which is the identity on $\LL(Q_n)$. 
It seems unlikely that this homomorphism is surjective. 
To prove this one could try to show that there are no bounded traces on $F(Q_n)$, cf. \cite[Section 4]{Daws_Horwath}.
\end{rem}

\bibliographystyle{plain}

\end{document}